\documentclass[12pt,reqno]{amsart}
\usepackage{mathrsfs}
\usepackage{amsgen}
\usepackage{amscd}
\usepackage{amsmath}
\usepackage{latexsym}
\usepackage{amsfonts}
\usepackage{amssymb}
\usepackage{amsthm}
\usepackage{graphicx}
\usepackage{color}
\usepackage{amsthm,amssymb}
\usepackage{graphicx}
\usepackage{enumerate}
\usepackage{amssymb,amsmath,graphicx,amsfonts,euscript}
\usepackage{color}
\usepackage{mathrsfs}

\usepackage[margin=3cm, a4paper]{geometry}

\def\H{{\cal H}}

\def\N{\mathbb{N}}
\def\R{\mathbb{R}}
\def\Z{\mathbb{Z}}

\def\C{\mathbb{C}}

\def\H2{H^2(\R^N)}
\def\L2{L^2(\R^N)}
\def\to{\rightarrow}

\def\H{{\cal H}}

\def\H1{H^1(\R)}

 \newcommand{\Del}[1]{}

\numberwithin{equation}{section}

\newtheorem{thm}{Theorem}[section]
\newtheorem{cor}[thm]{Corollary}
\newtheorem{lem}[thm]{Lemma}

\newtheorem{prop}[thm]{Proposition}
\newtheorem{definition}[thm]{Definition}

\theoremstyle{remark}
\newtheorem{remark}[thm]{Remark}
\newtheorem*{exam*}{Examples}

\begin{document}

\setcounter{page}{1}

\title[Large global solutions for NLS]{Large global solutions for nonlinear Schr\"odinger equations II, mass-supercritical, energy-subcritical cases}

\author{Marius Beceanu}
\address{Department of Mathematics and Statistics\\
University at Albany SUNY\\
Earth Science 110\\
Albany, NY, 12222, USA\\}
\email{mbeceanu@albany.edu}
\thanks{}

\author{Qingquan Deng}
\address{Department of Mathematics\\
Hubei Key Laboratory of Mathematical Science\\
Central China Normal University\\
Wuhan 430079, China.\\}
\email{dengq@mail.ccnu.edu.cn}
\thanks{}

\author{Avy Soffer}
\address{Department of Mathematics\\
Rutgers University\\
110 Frelinghuysen Rd.\\
Piscataway, NJ, 08854, USA\\}
\address{Department of Mathematics\\
Hubei Key Laboratory of Mathematical Science\\
Central China Normal University\\
Wuhan 430079, China.\\}
\email{soffer@math.rutgers.edu}
\thanks{}

\author{Yifei Wu}
\address{Center for Applied Mathematics\\
Tianjin University\\
Tianjin 300072, China}
\email{yerfmath@gmail.com}
\thanks{}

\subjclass[2010]{Primary  35Q55}


\keywords{Nonlinear Schr\"{o}dinger equation,
global well-posedness, scattering}

\maketitle

\begin{abstract}\noindent
In this paper, we consider the  defocusing mass-supercritical, energy-subcritical nonlinear Schr\"odinger equation,
$$
    i\partial_{t}u+\Delta u= |u|^p u, \quad (t,x)\in \R^{d+1},
$$
with $p\in (\frac4d,\frac4{d-2})$.
We prove that under some restrictions on $d,p$, any radial function in the rough space $H^{s_0}(\R^d),\textit{for some } s_0<s_c$ with the support away from the origin, there exists an incoming/outgoing decomposition, such that the initial data in the outgoing part leads to the global well-posedness and scattering forward in time; while the initial data in the incoming part leads to the global well-posedness and scattering   backward in time. The proof is based on Phase-Space analysis of the nonlinear dynamics.
\end{abstract}

\tableofcontents

\section{Introduction}
 The aim of this work is the study of global existence and scattering theory for the inter-critical nonlinear
Schr\"{o}dinger equation (NLS) in $3-5$ dimensions:
 \begin{equation}\label{eqs:NLS-cubic}
   \left\{ \aligned
    &i\partial_{t}u+\Delta u=\mu |u|^p u,
    \\
    &u(0,x)  =u_0(x),
   \endaligned
  \right.
 \end{equation}
 with $\mu=\pm1, p>0$.
Here $u(t,x):\R\times\R^d\rightarrow \C$ is a complex-valued function. The case $\mu=1$ is referred to the defocusing case, and $\mu=-1$  is referred to the focusing case. The class of solutions to equation (\ref{eqs:NLS-cubic}) are invariant under the scaling
\begin{equation}\label{eqs:scaling-p}
u(t,x)\to u_\lambda(t,x) = \lambda^{\frac2p} u(\lambda^2 t, \lambda x) \ \ {\rm for}\ \ \lambda>0,
\end{equation}
which maps the initial data as
\begin{eqnarray}
u(0)\to u_{\lambda}(0):=\lambda^{\frac2p} u_{0}(\lambda x) \ \ {\rm for}\ \ \lambda>0.\nonumber
\end{eqnarray}
Denote
$$
s_c=\frac d2-\frac2p.
$$
Then the scaling  leaves  $\dot{H}^{s_{c}}$ norm invariant, that is,
\begin{eqnarray*}
 \|u_\lambda(t)\|_{\dot H^{s_{c}}}=\|u(\lambda^2t)\|_{\dot H^{s_{c}}},
 \quad \|u(0)\|_{\dot H^{s_{c}}}=\|u_{\lambda}(0)\|_{\dot H^{s_{c}}},
\end{eqnarray*}
which is called \emph{critical regularity} $s_{c}$. It is also considered as the lowest regularity that problem  (\ref{eqs:NLS-cubic}) is well-posed for general $H^{s}(\R^d)$-data. Indeed, one can find some special initial datum belonging to $H^s(\R^d), s<s_c$ such that the problem   (\ref{eqs:NLS-cubic}) is ill-posed.

The $H^1$-solution of equation \eqref{eqs:NLS-cubic} also enjoys  the mass, momentum and energy
conservation laws, which read
\begin{equation}\label{eqs:energy-mass}
   \aligned
M(u(t))&:=\int |u(t,x)|^2\,dx=M(u_0),\\
P(u(t))&:=\textrm{Im}\int \overline{u(t,x)}\nabla u(t,x)\,dx=P(u_0),\\
E(u(t)) &:= \frac12\int |\nabla u(t,x)|^2\,dx + \frac{\mu}{p+2}\int
|u(t,x)|^{p+2} \,dx = E(u_0).
   \endaligned
\end{equation}

The well-posedness and scattering theory for Cauchy problem (\ref{eqs:NLS-cubic}) with initial data in $H^{s}(\R^d)$ were extensively studied, which we here  briefly review.
The local well-posedness theory follows from a standard fixed point argument, implying that  for all $u_{0}\in H^{s}(\R^d), s\ge s_c$, there exists $T_{0}>0$ such that its corresponding solution $u\in C([0,T_{0}),\ H^{s}(\R^d))$. In fact, the above $T_{0}$ depends on $\|u_{0}\|_{H^{s}(\R^d)}$ when $s>s_c$ and also the profile of $u_{0}$  when $s=s_c$. Some of the results can be found in Cazenave and Weissler \cite{CW1}.

The  fixed point argument used in local theory can be applied directly to prove the global well-posedness for solutions to equation (\ref{eqs:NLS-cubic})  with small initial data in $H^{s}(\R^d)$ with $s\geq s_c$.
 In the mass-supercritical, energy-subcritical cases, that is, $\frac4d<p<\frac4{d-2}$, if we consider the solution in energy space $H^1(\R^d)$, the local theory above together with  conservation laws  (\ref{eqs:energy-mass}), yields the global well-posedness for all initial data $u_0\in H^1(\R^d)$ in the defocusing case $\mu=1$, and for any initial data $u_0\in H^1(\R^d)$ with some restrictions in the focusing case.
Furthermore, the scattering under the same conditions were also obtained by Ginibre, Velo \cite{GiVe} in the defocusing case and Duyckaerts, Holmer and Roudenko \cite{DuHoRo} in the focusing case.

Recently, conditional global and scattering results for $s_c>1$ with the assumption of $u\in L^\infty_t(I,\dot H^{s_c}_x(\R^d))$ were considered by many authors, beginning with the major work \cite{KeMe-cubic-NLS-2010, KeMe-wave-2011-1}, and then developed by
\cite{Bu, DoMiMuZh-17,DKM, DuRo, MJJ, MWZ} and cited references.
 That is,  if the initial data $u_{0}\in \dot H^{s_c}(\R^d)$ and the solution has a priori estimate
\begin{eqnarray}
\sup_{0<t<T_{out}(u_{0})}\|u(t)\|_{\dot H^{s_c}_x(\R^d)}<+\infty, \label{uniformbound}
\end{eqnarray}
then $T_{out}(u_{0})=+\infty$ and the solution scatters in $\dot H^{s_c}(\R^d)$, here $[0,T_{out}(u_{0}))$ is the maximal interval for existence of the solution. Consequently, these results give blowup criterion in which the lifetime only depends on the critical norm $\|u\|_{L^\infty_t\dot H^{s_c}_x(\R^d)}$.  However,  no such large global
results are known for general initial data $u_{0}\in \dot H^{s_c}(\R^d)$.

In the case when $s_c<1$ the use of a priori estimates on the solution as a condition was also developed.
 The work of Bourgain \cite{Bou-1998} on the NLS, made assumptions on the space-time norm of the solution in space-time subsets, by deriving necessary conditions for blow-up. For example, Kenig and Merle \cite{KeMe-cubic-NLS-2010} proved that for the NLS in the intercritical case, in dimension 3 and cubic nonlinearity case, global existence and scattering hold under the condition
 $$
 \sup_{0<t<T_{out}(u_{0})} \|u(t)\|_{\dot H^{\frac12}(\R^d)} <\infty.
 $$
See also for examples some developments in \cite{KiMaMuVi-2018, XieFa-13}.

These conditional results, and other works, used the elaborate method introduced by Kenig-Merle, involving profile decomposition, concentration compactness  and the localized Strichartz estimates.
 In contrast, our results use explicit conditions on the \emph{initial data alone}. A further new consequence of our analysis is that supercritical/rough solutions exist, which are large and that in fact the standard assumption of initial data being in the space $H^{s_c}$ is not required.

We follow here a new paradigm, based on phase space analysis and propagation estimates, to give explicit conditions on the initial data, which implies global existence and scattering. This will allow us to give explicit conditions on large and rough initial data, for which global existence and scattering hold. Our method does not use the above techniques based on profile decomposition and concentration compactness.

When the initial data is rough, $s<s_c$, we do not have general global well-posedness. Yet, as we will show, generic conditions for global existence and scattering allow such initial data.
So, we conclude that rough initial data does not guarantee blowup.

 Rough data may appear naturally in some applications: initial data corrupted with noise (as in nonlinear optics applications), and in the construction of invariant measures for the NLS dynamics, see \cite{Bou-1994, Bou-1997} and further developments in \cite{CoOh-2012, Dodson-17,KiMuVi-2017, NaOhReSt-2015, OhOkPo-2017, Th-09}  and the cited references.
 In these last cases it is known that the relevant measure is supported on rough spaces.

 The phase-space analysis we use seeks to identify the initial data that can not move into the origin and blow-up.
 So, naturally, this analysis requires the distinction between outgoing waves under the free flow, which move away from the origin (for positive times), and the incoming waves which move towards the origin.

 In particular we show, that under some extra conditions, initial data which is outgoing, will lead to global solution, even if it is rough, that is, the data belongs to the space $H^s$, with $s<s_c,$ where $s_c$ is the critical Sobolev norm. While it looks unsurprising, it is in fact a subtle result!

 Under the free flow this initial data (radial) will move away from the origin, and therefore will get smaller in $L^{\infty}$, by the assumption of radial symmetry, and some $H^s$ regularity.  Moreover,  since the speed depends on the frequency, it gets smaller and faster for higher frequencies. Indeed, as an estimate essentially proved in Proposition \ref{prop:pre-outgoing} below, we show that
 \begin{align}
 \big\|\chi_{\le 1+Nt}e^{it\Delta}P_{\ge N}f_{out}\big\|_{H^1}\lesssim \|f\|_{L^2}.\label{est:smooth-effect}
\ \end{align}

  However, the property of being outgoing is not necessarily preserved by the nonlinearity. Furthermore, in a major difference from the wave equation case (see \cite{BS, BS-2}), part of the solution moves \emph{backward} towards the origin, even under the free flow.
   Since the data is rough, we do not have local existence (see the ill-posedness results in \cite{ChCoTa-Ill} and \cite{KePoVe-Duke-2001}), so, we can not move a short time forward, in the usual perturbative way e.g., as in \cite{S.Miao}.

 To counter these effects, we need optimal propagation estimates in space, frequency and time.
 In particular we need to use the extra smoothing effect for waves which move in the ``wrong'' direction, the classical forbidden regions, see \eqref{est:smooth-effect}.
 We also need gain of regularity in the Strichartz estimate for radial functions, which allows us to show that in fact the Duhamel part, contributed by the nonlinearity is essential in $H^1$, which is based on the following result.  A more general result will be stated in Section 4.
\begin{prop}\label{prop:out}
Let $d=3,4,5$.
For any  radial function $f\in L^2(\R^d)$ satisfying
\begin{align*}
\mbox{supp }f\subset \{x:|x|\ge 1\},
\end{align*}
there exists some suitable decomposition, says incoming and outgoing decomposition,
  $$
f=f_++f_-,
$$
such that the following supercritical space-time estimate holds for any $N\ge 1$,
\begin{align*}
\big\|e^{\pm it\Delta}P_Nf_\pm\big\|_{L^\infty_t H^1_x(\R^+\times\R^d)+L^2_tL^\infty_x(\R^+\times\R^d)}\lesssim\big\|P_Nf\big\|_{L^2}.
\end{align*}
\end{prop}

\begin{remark}
In general, the standard Strichartz estimates (see \cite{KeTa-Strichartz}) imply
\begin{align*}
\big\|e^{\pm it\Delta}P_Nf \big\|_{L^2_tL^\infty_x(\R^+\times\R^d)}\lesssim N^{\frac d2-1}\big\|P_Nf\big\|_{L^2},
\end{align*}
which is scaling invariant and thus the index $\frac d2-1$ is optimal. Hence, the decomposition above presents a way in which we are able to obtain a kind of   supercritical Strichartz estimates.

The proposition is based on the decomposition  
$$
L^2\cap \big\{f: \mbox{supp }f\subset \{|x|\ge 1\} \big\}=L^2_++L^2_-,
$$
for which  the spaces $L^2_\pm$ can be roughly described that if $f\in L^2_+$, then for any $t>0$, 
$$
e^{it\Delta}P_Nf=\chi_{\ge \delta\langle Nt\rangle } F(t)+O(N^{-a}\langle t \rangle^{-b}),\quad \mbox{for some} \quad F(t)\in L^2;
$$
and if $f\in L^2_-$, then the equality above holds for any $t<0$.
Here $\delta>0$ is a small constant,  $a>0$ and $b>0$. 

More precisely, if  $f$ is supported away from zero, then we can write $f=f_++f_-+f_0$ with $f_0\in H^1, f_+=P_+f \in L^2_+, f_-=P_-f \in L^2_-$, where $P_\pm$ are incoming/outgoing projection operators  defined in Definition \ref{def:outgong-incoming}. The crucial step in the construction above is the proof we give that the range of $P_+$ is almost invariant (in the sense above) under the free flow for $t>0,$ and similarly for $P_-, t<0.$   
Since the problem under study is energy subcritical, the solution is stable under $H^1$ perturbations. We emphasize that $P^\pm$ are not completely invariant under the free flow. There is always an $H^1$ correction. Moreover if it is true (invariance of the range) at one time, it will not be true at later times. These present obstacles in the study of the long time behavior of the solution. 
\end{remark}

These estimates together with the approximate energy identity are applied to the part $w$ of the solution.
$w$ is defined by the following decomposition for the nonlinear solution: $$w=u-e^{it\Delta}\big(P_{\ge N}f\big)_{+},$$ (with a slight modification).
Then, $w$ satisfies the following equation: $$i\partial_{t}w+\Delta w=|u|^pu.$$

Therefore, by showing that $w$ is in $H^1$, we conclude that all the singular part, is carried away as a free wave to infinity.

  Hence, since the problem is energy subcritical, we can control the effect of the nonlinearity, even though it is large.
 One should note however, that better, supercritical smoothing estimates only hold for outgoing waves.

 Our construction of the projections on incoming/outgoing waves follows a similar approach of  T. Tao \cite{Tao-2}. We make a different decomposition near the origin.

  A comparison to such in/out decompositions in scattering theory
 may be illuminating: the problem of global existence is the exact dual in the phase space, of the scattering problem. For global existence we need to control short time behavior, near the origin in space, at high frequency; the complete opposite of scattering. In scattering theory the decomposition into in and out waves, which was inspired by Enss method, leads to various definitions of such projections, beginning with the works of Mourre \cite{Mourre}.
 Mourre constructed the projection by taking the projection on the positive spectral part of the dilation generator $A$, defined as $x\cdot p+p\cdot x\equiv 2A$.
  Here $x$ is multiplication by $x$ in space and $p$, the momentum operator in Quantum Mechanics, is defined by $p\equiv -i\nabla_x.$ He defined the projections on the incoming/outgoing parts by $P^{\pm}(A)$,
 where $P^{\pm}(y)$ stand for the characteristic functions, in $y$, of the positive/negative real axis.

 Clearly $A$ is a nice PDO, but not function of $A.$

 Then, propagation estimates and other constructions, including estimates based on replacing $A$ by $\gamma\equiv f(x)x\cdot p+ p\cdot f(x)x$ were developed by Sigal and Soffer \cite{SiSo-87}. Here $f(x)\sim 1/|x|, |x|>1.$ In this construction, the ``out'' component of the solution still has some small portion moving inward, but this portion decays fast in time. See however \cite{Soffer-11}.
 Unlike the constructions we use, which are adapted to the radial nonlinear problem, the scattering theory constructions above apply to non-radial functions, and so are the known estimates.

 However, the propagation estimates obtained are not suitable, as they stand, for the dual space problem. They apply for general Hamiltonians, not just the Laplacian.  It may be used for short times as well, when the frequency is large; that is still not explicit in the literature. The problem comes from the fact that near zero in space, the frequency can be very large, and yet $A$ and $\gamma$ remain bounded.

 In the nonlinear context, this problem is far more severe, since the high frequency part is not stable under nonlinear perturbations.

   New constructions for in/out decomposition were first introduced by T. Tao \cite{Tao-2}. They are based on a decomposition in terms of spherical waves, of the form $e^{ikr}/r$, in three dimensions.

 Still the problem at zero remains, much work was done to deal with this part of the phase-space.
 Further works in this direction, and others, sharp propagation estimates were done in e.g.   Killip, Tao, Visan \cite{KiTaVi-2009}, and Li and Zhang \cite{KiLiViZh, LiZh-APDE, LiZh-SCM}.

 We follow, up to some modifications near the origin, similar constructions and estimates.

 We will use them in a different way, to localize and propagate rough initial data in particular, so as to get explicit conditions for global existence and scattering. Our strategy is to show that in some sense the phase-space localization of the initial data is stable, up to smooth corrections. So, we show that the solution is a sum of the linear, rough/supercritical part plus a correction coming from the Duhamel term (the nonlinear part).

 By using radial symmetry, we show there is a gain of regularity, and the Duhamel term contributes essentially in $H^1$ space, the correction $w.$
 It will also follow, that if the initial $w$ is small, then it can be controlled in the focusing case. For this we need only to have the initial data small in $H^s,$ for some $s<s_c.$

 Therefore, since the nonlinearity is inter-critical, and we have the Morawetz estimates and energy estimates in hand, we can then  get global existence theory as in the $H^1$ case, after using also frequency cut-off and a continuity argument.

  Similar improved smoothing for the Duhamel term was obtained before by Bourgain \cite{Bou-1999-book}, in cases where the data is \emph{subcritical} and below $H^1.$

 In the defocusing case, we can then cover outgoing initial data of arbitrary size, and in the focusing case small (but in the rough/supercritical norm!) data.
 Since the initial data is supercritical, in general, there is no well-posedness; to achieve that, we need the initial data to be supported at some positive distance away from the origin in space , together with the outgoing condition.

Our main result is  following, which we focus on the defocusing case.
\begin{thm}\label{thm:main1}
Let $\mu=1$, $d=3,4,5$.
Then there exist $s_0<1$ and $p_1(d)<\frac{4}{d-2}$, such that for any $p\in [p_1(d), \frac{4}{d-2})$, the following is true. Suppose that  $f$ is a radial function for which there exists $\varepsilon_0>0$ such that 
\begin{align}\label{condition:thm}
\chi_{\le \varepsilon_0}f\in H^1(\R^d),\quad \big(1-\chi_{\le \varepsilon_0}\big)f\in H^{s_0}(\R^d).
\end{align}
Then the solution $u$ to the equation \eqref{eqs:NLS-cubic} with the initial data
$$
u_0=f_{+} \qquad  (\mbox{or} \quad u_0=f_{-})
$$
exists globally forward (or backward) in time, and $u$ is unique in $C(\R^+;H^{s_0}(\R^d))\cap X_+$ (or $u\in C(\R^-;H^{s_0}(\R^d))\cap X_-$), in which $X_\pm$ are some auxiliary space-time space.
Moreover, the solution depends continuously on the initial data in $H^{s_0}(\R^d)$.
 Here $f_{+}$ and $f_{-}$ are the modified outgoing and  incoming components of $f$ respectively, which are  given in Proposition \ref{prop:out}.
Furthermore, there exists $u^+\in H^{s_0}(\R^d)$ (or $u^-\in  H^{s_0}(\R^d)$), such that when $t\to +\infty$ (or $t\to -\infty$),
\begin{align}
\lim\limits_{t\to +\infty}\|u(t)-e^{it\Delta}u^+\|_{H^1(\R^d)}= 0 \qquad
(\mbox{or } \lim\limits_{t\to -\infty}\|u(t)-e^{it\Delta}u^-\|_{H^1(\R^d)}= 0). \label{scattering1}
\end{align}
\end{thm}
\begin{remark}\label{rem:10}
We make several remarks regarding the above statements.

\begin{itemize}
  \item[(1).] Note that $s_0$ is independent of $p$, hence $s_0<s_c$ when $p$ is close to $\frac4{d-2}$. Moreover, since $f=f_++f_-$, if $f$ is not in $H^{s_c}(\R^d)$, at least one of $f_+$ and $f_-$ is not in $H^{s_c}(\R^d)$. Therefore, we obtain a class of the global solutions for the defocusing energy-subcritical nonlinear Sch\"odinger equation in the supercritical space $H^{s_0}(\R^d)$.

  \item[(2).]
  It is worth noting that there is no smallness restriction in Theorem \ref{thm:main1}.
 Here we are not going to pursue the sharp conditions on $s_0$ and $p_1(d)$ in this paper.

Moreover, the restriction on the dimensions $d=3,4,5$ is not essential, and the analogous results are valid in more general dimensions. However, for the sake of readability, we will not go into details.

 \item[(3).] The theorem implies that the incoming/outgoing solution has the ``smoothing effect''. Indeed, we can show that for initial data belonging to $L^2(\R^d)$, there exist some $s_*>0, r_0>2$, such that  the solution $u(t)$ of \eqref{eqs:NLS-cubic} corresponding to such initial data, is in $W^{s,r}(\R^d)$ for any $t>0$, any $r\in (2,r_0]$ and any $s\in [0,s_*]$, and moreover, the $L^r$-norm decays at infinite time in the sense that
$$
\|u(t)\|_{L^r(\R^d)}\to 0,\quad \mbox{ as }\quad  t\to +\infty.
$$

\item[(4).] By rescaling, we only need to prove the theorem when $\varepsilon_0=1$.
\end{itemize}
\end{remark}

\textbf{Organization of the paper.}  In Section 2, we give some preliminaries, which include
some basic lemmas, some estimates on the linear Schr\"odinger operator and the Fourier integral operators.
In Section 3, we give the definitions of the incoming/outgoing waves, their basic properties, the boundeness in $\dot H^{s_c}(\R^d)$. In Section 4, we give some supercritical spacetime estimates on the incoming/outgoing linear flow.
In Section 5, we give the proof of the main
theorem.

\section{Preliminary}

\subsection{Some notations}

We write $X \lesssim Y$ or $Y \gtrsim X$ to indicate $X \leq CY$ for some constant $C>0$. If $C$ depends upon some additional
parameters, we will indicate this with subscripts; for example, $X\lesssim_a Y$ denotes the
assertion that $X\le C(a)Y$ for some $C(a)$ depending on $a$. We use $O(Y)$ to denote any quantity $X$
such that $|X| \lesssim Y$.  We use the notation $X \sim Y$ whenever $X \lesssim Y \lesssim X$.

The notation $a+$ denotes $a+\epsilon$ for any small $\epsilon$, and $a-$ for $a-\epsilon$.
$|\nabla|^\alpha=(-\Delta)^{\alpha/2}$. $\langle \cdot\rangle=(1+|\cdot|^2)^\frac12$.
We denote $\mathcal S(\R^d)$ to be the Schwartz Space in $\R^d$, and $\mathcal S'(\R^d)$ to be the topological dual of $\mathcal S(\R^d)$.  Let $h\in \mathcal S'(\R^{d+1})$, we use
$\|h\|_{L^q_tL^p_x}$ to denote the mixed norm
$\Big(\displaystyle\int\|h(\cdot,t)\|_{L^p}^q\
dt\Big)^{\frac{1}{q}}$, and $\|h\|_{L^q_{xt}}:=\|h\|_{L^q_xL^q_t}$. Sometimes, we use the notation $q'=\frac{q}{q-1}$.

Throughout this paper, we use $\chi_{\le a}$ for $a\in \R^+$ to be the smooth function
\begin{align*}
\chi_{\le a}(x)=\left\{ \aligned
1, \ & |x|\le a,\\
0,    \ &|x|\ge \frac{101}{100} a.
\endaligned
  \right.
\end{align*}
Moreover, we denote $\chi_{\ge a}=1-\chi_{\le a}$, $\chi_{a\le \cdot\le b}=\chi_{\le b}-\chi_{\le a}$ and $\chi_{a}=\chi_{\le 2a}-\chi_{\le a}$ for short.

Also, we need some Fourier operators.
For each number $N > 0$, we define the Fourier multipliers $P_{\le N}, P_{\ge N}, P_N$ as
\begin{align*}
\widehat{P_{\leq N} f}(\xi) &:= \chi_{\leq N}(\xi) \hat f(\xi),\\
\widehat{P_{\ge N} f}(\xi) &:= \chi_{\ge N}(\xi) \hat f(\xi),\\
\widehat{P_N f}(\xi) &:= \chi_{N}(\xi) \hat
f(\xi).
\end{align*}
 We will usually use these multipliers when $N$ are \emph{dyadic numbers} (that is, of the form $2^k$
for some integer $k$).

\subsection{Some basic lemmas}
First, we need the following radial Sobolev embedding, see, e.g., \cite{TaViZh}.
\begin{lem}\label{lem:radial-Sob}
Let $\alpha,q,p,s$ be the parameters which satisfy
$$
\alpha>-\frac dq;\quad \frac1q\le \frac1p\le \frac1q+s;\quad 1\le p,q\le \infty; \quad 0<s<d
$$
with
$$
\alpha+s=d(\frac1p-\frac1q).
$$
Moreover, at most one of the equalities holds:
$$
p=1,\quad p=\infty,\quad q=1,\quad q=\infty,\quad \frac1p=\frac1q+s.
$$
Then for any radial function $u$ such that $ |\nabla|^s u\in L^p(\R^d)$,
\begin{align*}
\big\||x|^\alpha u\big\|_{L^q(\R^d)}\lesssim \big\||\nabla|^su\big\|_{L^p(\R^d)}.
\end{align*}
\end{lem}

The second is the following fractional Leibniz rule, see \cite{BoLi-KatoPonce, Li-KatoPonce} and the references therein.
\begin{lem}\label{lem:Frac_Leibniz}
Let $0<s<1$, $1<p\le \infty$, and $1<p_1,p_2,p_3, p_4 \le \infty$ with $\frac1p=\frac1{p_1}+\frac1{p_2}$, $\frac1p=\frac1{p_3}+\frac1{p_4}$, and let $f,g\in \mathcal S(\R^d)$,  then
\begin{align*}
\big\||\nabla|^s(fg)\big\|_{L^p}\lesssim \big\||\nabla|^sf\big\|_{L^{p_1}}\|g\|_{L^{p_2}}+ \big\||\nabla|^sg\big\|_{L^{p_3}}\|f\|_{L^{p_4}}.
\end{align*}
\end{lem}

A simple consequence is the  following elementary inequality, see \cite{BDSW} for the proof.
\begin{lem}\label{lem:frac_Hs}
For any $a>0, 1<p\le \infty, 0\le \gamma<\frac dp$, and $ |\nabla|^\gamma g\in L^p(\R^d)$,
\begin{align}
\big\||\nabla|^\gamma\big(\chi_{\le a}g\big)\big\|_{L^p(\R^d)}\lesssim \big\||\nabla|^\gamma g\big\|_{L^p(\R^d)}. \label{15.37}
\end{align}
Here the implicit constant is independent of $a$.
The same estimate holds for $\chi_{\ge a}g$.
\end{lem}

Moreover, we need the following mismatch result, which is  helpful in commuting the spatial and the frequency cutoffs.
\begin{lem}[Mismatch estimates, see \cite{LiZh-APDE}]\label{lem:mismatch}
Let $\phi_1$ and $\phi_2$ be smooth functions obeying
$$
|\phi_j| \leq 1 \quad \mbox{ and }\quad \mbox{dist}(\emph{supp}
\phi_1,\, \emph{supp} \phi_2 ) \geq A,
$$
for some large constant $A$.  Then for $\sigma\ge 0$, $M\ge 1$, $1\leq r\leq
q\leq \infty$ and for any $m\ge 0$,
\begin{align}
\bigl\| \phi_1 |\nabla|^\sigma P_{\geq M} (\phi_2 f)
\bigr\|_{L^q_x(\R^d)}
    &
    \lesssim M^{\sigma-m}A^{-m+\frac dq-\frac dr} \|\phi_2 f\|_{L^r_x(\R^d)}.\label{eqs:lem-mismath-1}
\end{align}
\end{lem}
\begin{remark}
In this paper, we will frequently used the following estimate from Lemma \ref{lem:mismatch},
\begin{align*}
\big\|\chi_{\le \frac14} P_{M}\big(\chi_{\ge 1}f\big)\big\|_{L^\infty(\R^d)}\lesssim M^{-m}\big\|\tilde P_{M}\big(\chi_{\ge 1}f\big)\big\|_{L^2(\R^d)}, \,\, \mbox{ for any } m\ge 0,
\end{align*}
where $\tilde P_N$ is defined by
$$
\widehat{\tilde P_N f}(\xi) := \chi_{2^{10}N\le\cdot \le  2^{10}N}(\xi) \hat
f(\xi).
$$
In the following, we shall slightly abuse notation and write $\tilde P_N$ by $P_N$.
\end{remark}

Furthermore, we need the following  elementary formulas, see \cite{BDSW}.
\begin{lem}\label{lem:muli-Lei-formula}
Let $f\in \mathcal S(\R^d)^d$ and  $g\in \mathcal S(\R^d)$, then for any integer $N$,
\begin{equation*}
\nabla_{\xi}\cdot \big(f\>\nabla_{\xi} \big)^{N-1}\cdot
(fg)=\sum\limits_{\begin{subarray}{c}
l_1,\cdots,l_N\in\N^d,l'\in\N^d;\\
|l_j|\le
j;|l_1|+\cdots+|l_N|+|l'|=N
\end{subarray}}
C_{l_1,\cdots,l_N,l'}\partial_\xi^{l_1}f\cdots
\partial_\xi^{l_N}f\>\partial_\xi^{l'}g,
\end{equation*}
where we have used the notations
$$
\nabla_{\xi}=\{\partial_{\xi_1},\cdots,\partial_{\xi_d}\};\quad
\partial_\xi^{l}f=\sum\limits_{j=1}^dC_j\partial_{\xi_1}^{l^1}\cdots\partial_{\xi_d}^{l^d}f_j, \mbox{ for any } l=\{l^1,\cdots,l^d\}\in \N^d.
$$
Here $C_{l_1,\cdots,l_N,l'}, C_j$ are some absolute constants only depending on $N$ and $d$. 
\end{lem}

\subsection{Linear Schr\"odinger operator}

Let the operator $S(t)=e^{it\Delta}$ be the linear Schr\"odinger flow, that is,
$$
(i\partial_t+\Delta)S(t)\equiv 0.
$$
The following are some fundamental properties of the operator $e^{it\Delta}$. The first is the explicit formula, see, e.g., Cazenave \cite{Cazenave-book}.
\begin{lem}\label{lem:formula-St}
For all $\phi\in \mathcal S(\R^d)$, $t\neq 0$,
$$
S(t)\phi(x)=\frac{1}{(4\pi it)^{\frac d2}}\int_{\R^d} e^{\frac{i|x-y|^2}{4t}}\phi(y)\,dy.
$$
Moreover, for any $r\ge2$,
$$
\|S(t)\phi\|_{L^r_x(\R^d)}\lesssim |t|^{-d(\frac12-\frac1r)}\|\phi\|_{L^{r'}(\R^d)}.
$$
\end{lem}

The following is the  standard Strichartz estimate, see for example \cite{KeTa-Strichartz}.
\begin{lem}\label{lem:strichartz}
Let $I$ be a time interval and let $u: I\times\R^d \to
\mathbb \R$ be a solution to the inhomogeneous Schr\"odinger equation
$$
iu_{t}- \Delta u + F = 0.
$$
Then for any $t_0\in I$,  any pairs $(q_j,r_j), j=1,2$ satisfying
$$
 q_j\ge 2, \,\, r_j\ge 2,\,\,  \frac2{q_j}+\frac d{r_j}=\frac d2,\,\, \mbox{and }\,\, (q_j,r_j,d)\ne (2,+\infty,2),
$$
the following estimate holds,
\begin{align*}
\bigl\|u\bigr\|_{C(I;L^2(\R^d))}+\big\|u\big\|_{L^{q_1}_tL^{r_1}_x(I\times\R^d)}
\lesssim \bigl\|u(t_0)\bigr\|_{L^2_x(\R^d)}+
\bigl\|F\bigr\|_{L^{q_2'}_tL^{r_2'}_x(I\times\R^d)}.
\end{align*}
\end{lem}

We also need the special Strichartz estimates for radial data, which were firstly proved by Shao \cite{Shao}, and then developed in \cite{CL,GW}.
\begin{lem}[Radial Strichartz estimates]\label{lem:radial-Str}
Let $g\in L^2(\R^d)$ be a radial function, $k$ be an integer, then for any triple $(q,r,\gamma)$ satisfying
\begin{align}
\gamma\in\R, \,\, q\ge 2, \,\, r> 2,\,\, \frac2q+\frac {2d-1}{r}<\frac{2d-1}2,\,\, \mbox{and }\,\, \frac2q+\frac d{r}=\frac d2+\gamma,\label{Str-conditions}
\end{align}
we have that
$$
\big\||\nabla|^{\gamma} e^{it\Delta}g\big\|_{L^q_{t}L^r_x(\R\times\R^d)}\lesssim  \big\|g\big\|_{L^2(\R^d)}.
$$
Furthermore, let $F\in L^{\tilde q'}_{t}L^{\tilde r'}_x(\R^{d+1})$ be a radial function in $x$, then
$$
\Big\|\int_0^t e^{i(t-s)\Delta}F(s)\,ds\Big\|_{L^q_{t}L^r_x(\R^{d+1})}+
\Big\||\nabla|^{-\gamma}\int_0^t e^{i(t-s)\Delta}F(s)\,ds\Big\|_{L^\infty_tL^2_{x}(\R^{d+1})}\lesssim \|F\|_{L^{\tilde q'}_{t}L^{\tilde r'}_x(\R^{d+1})},
$$
where the triples  $(q,r,\gamma)$,   $(\tilde q,\tilde r,-\gamma)$ satisfy \eqref{Str-conditions}.
\end{lem}

\subsection{Some lemmas about Fourier integral operators}
The following are some lemmas related to the estimate of the  Fourier integral operators, we refer to Stein \cite{Stein-book} for the proofs. The first one is from the application of the ``stationary phase'' theory.
\begin{lem}\label{lem:stationary-phase} Let $\phi, \psi$ be smooth functions defined in $\R$ and $\lambda\in\R^+$, and $\phi$ satisfies
$$
\phi(x_0)=\phi'(x_0)=0,\quad \mbox{ and } \quad \phi''(x_0)\ne 0.
$$
If $\psi$ is supported in  a sufficiently small neighborhood of $x_0$, then
\begin{align*}
\int_\R e^{i\lambda \phi(x)}\psi(x)\,dx=a_0\lambda^{-\frac12}+O(\lambda^{-\frac32}),\quad \mbox{ when } \lambda\to +\infty,
\end{align*}
with
$$
a_0=\psi(x_0)\Big(\frac{2\pi}{i\phi''(x_0)}\Big)^{\frac12}.
$$
\end{lem}

The second result is the estimate of the  Fourier integral operator, which can be regarded as an extension of the Plancherel identity. The following result can be found in \cite{FeSt}.
\begin{prop}\label{prop:L2-estimate-FIO} 
Let $T$ be the Fourier integral operator given by
\begin{align}
Tf(x)=\int_{\R^d} e^{ i x\cdot \xi} a(x,\xi)\chi_{\le 1}(\xi) \hat f(\xi)\,d\xi.\label{Four-In-Op-T}
\end{align}
Suppose that $a(x,\xi)\in C^\infty(\R^d\times\R^d)$ and satisfies that for all multi-indices $\alpha\in\R^d$,
$$
\sup\limits_{\xi\in\R^d} \langle \xi\rangle^{|\alpha|}\big\|\partial_\xi^\alpha a(\cdot,\xi)\big\|_{L^\infty_x(\R^d)}<+\infty.
$$
Then the operator $T$ defined in \eqref{Four-In-Op-T} satisfies that
$$
T: L^p(\R^d)\mapsto L^p(\R^d),\quad \mbox{ is bounded, for all } 1<p<\infty.
$$
\end{prop}

As an immediate consequence of the proposition above, we have the  estimates on the following   Fourier integral operators which are used in this paper. 
\begin{cor}\label{cor:L2-estimate-FIO} 
Let $\beta\in\R, N>0$,  and let $T$ be the Fourier integral operator given by
\begin{align}
Tf(x)=\int_{\R^d} e^{ i \beta x\cdot \xi} \chi_{\le 1}(x\cdot\xi)\chi_{\le N}(\xi)  f(\xi)\,d\xi.\label{Four-In-Op-T-2}
\end{align}
Then there exists some $C>0$ which is independent of $N, \beta$,  such that 
$$
\|Tf\|_{L^2}\le C|\beta|^{-\frac d2}\|\chi_{\le N}f\|_{L^2}.
$$
The same estimate holds if $\chi_{\le 1}$ is replaced by $\chi_{\ge 1}$, $\chi'_{\le 1}$, $\chi'_{\ge 1}$  or $\chi''_{\ge 1}$  in  \eqref{Four-In-Op-T-2}.
\end{cor}
\begin{proof}
 Since the operator $T$ is scaling invariant in $L^2$, we can use the rescaling argument. Indeed, the result is obtained by changing to the  new variables $\eta= N^{-1}\xi, y=\beta N x$, and  then applying Proposition \ref{prop:L2-estimate-FIO} and the Plancherel identity.  The cases of $\chi_{\le 1}, \chi'_{\le 1}, \chi'_{\ge 1}$  and  $\chi''_{\ge 1}$ can be treated in the same way. Since  $\chi_{\ge 1}(x\cdot\xi)= 1- \chi_{\le 1}(x\cdot\xi)$, combining with the  Plancherel identity,  the analogous estimate when $\chi_{\le 1}(x\cdot\xi)$ is replaced by $\chi_{\ge 1}(x\cdot\xi)$ in \eqref{Four-In-Op-T-2} is also proved. 
\end{proof}

\section{The incoming/outgoing waves}\label{sec:out/in}

\subsection{Definitions of the incoming/outgoing waves}
First of all, we give the definitions of the incoming/outgoing waves for the Schr\"odinger flow.

\subsubsection{The deformed Fourier transform}
We denote the standard Fourier transform  by $\hat f$ or $\mathscr{F} f$ as
\begin{align*}
\big( \,\mathscr F f(\xi)\,\,\mbox{or}\,\,\big)\hat f(\xi) &=\int_{\R^d} e^{-2\pi i x\cdot \xi}f(x)\,dx,
\end{align*}
and its inverse transform
\begin{align*}
\big( \,\mathscr F^{-1} f(x)\,\,\mbox{or}\,\,\big)\check f(x) &=\int_{\R^d} e^{2\pi i x\cdot \xi}f(\xi)\,d\xi.
\end{align*}
Now we  give a deformed Fourier transform, and its basic properties.
\begin{definition}
Let $\alpha\in \R,\beta\in\R$, and let $f\in \mathcal S(\R^d)$ with $|x|^\beta f\in L_{loc}^1(\R^d)$. We define
\begin{align}
\mathcal F f(\xi)=
|\xi|^\alpha\int_{\R^d}e^{-2\pi i x\cdot\xi}|x|^\beta f(x)\,dx. \label{def:deformed-Fourier}
\end{align}
\end{definition}
Then it is easy to see the following inverse transform, that is,
\begin{lem} \label{lem:inverse-Deformed-Four}
Let $f\in \mathcal S(\R^d)$,  $|x|^\beta f\in L_{loc}^1(\R^d)$ and $|\xi|^{-\alpha}\mathcal F f\in L_{loc}^1(\R^d)$, then for any $x\in \R^d\setminus \{0\}$,
\begin{align}
f(x)=
|x|^{-\beta}\int_{\R^d}e^{2\pi i x\cdot\xi}|\xi|^{-\alpha} \mathcal F f(\xi)\,d\xi. \label{inverse-Deformed-Four}
\end{align}
\end{lem}
\begin{proof}
From the definition,
\begin{align}
|\xi|^{-\alpha}\mathcal F f(\xi)=\mathscr F\Big(|x|^\beta f\Big)(\xi). \label{deformed-Fourier}
\end{align}
Hence, by the inverse Fourier transform, we have
\begin{align*}
|x|^\beta f(x)=\mathscr F^{-1}\Big(|\xi|^{-\alpha}\mathcal F f\Big)(x).
\end{align*}
This gives the formula in \eqref{inverse-Deformed-Four}.
\end{proof}
We give the following remark on the conditions in Lemma \ref{lem:inverse-Deformed-Four}.
\begin{remark}
If $f\in\mathcal S(\R^d)$ is radial, and $\beta>-d$, then $|x|^\beta f\in L_{loc}^1(\R^d)$. Indeed,
$$
\int_{|x|\le 1}|x|^\beta f(x)\,dx=c(d)\int_0^1 r^{\beta+d-1} f(r)\,dr\lesssim \|f\|_{L^\infty(\R^d)}.
$$
Similarly, if $\alpha<d$, we can prove that $|\xi|^{-\alpha}\mathcal F f\in L^1_{loc}(\R^d)$.
\end{remark}

We now give the radial version of the deformed Fourier transform and its inverse transform. We note that if $f$ is radial,  so is $\mathcal Ff$. Moreover,
\begin{align*}
\mathcal F f(\rho)&=
\rho^\alpha\int_0^{+\infty}\!\!\int_{|\theta|=1}e^{-2\pi i \rho r \omega\cdot \theta}\,d\theta  f(r)r^{\beta+d-1}  \,dr\\
&=\rho^\alpha\int_0^{+\infty}\!\! \widehat{d\omega}(\rho r \omega) r^{\beta+d-1} f(r)\,dr,
\end{align*}
where $\xi=\rho\omega$.
Note that $ \widehat{d\omega}$ has radial symmetry, and
$$
 \widehat{d\omega}(\rho r \omega)= \widehat{d\omega}(\rho r)=
 \int_{-\frac\pi2}^{\frac\pi2}e^{-2\pi i \rho r\sin \theta}\cos^{d-2} \theta\,d\theta.
$$
(It is equal to $2\pi (\rho r)^{-\frac{d-2}2} J_{\frac{d-2}2}(2\pi \rho r)$, where $J_{\frac{d-2}2}$ is a bessel function.)
Therefore,
\begin{align}
\mathcal F f(\rho)&=
\rho^\alpha\int_0^{+\infty}\!\!\int_{-\frac\pi2}^{\frac\pi2}e^{-2\pi i \rho r\sin \theta}\cos^{d-2} \theta \> r^{\beta+d-1} f(r) \,d\theta\,dr.\label{deformed-Fourier-radial}
\end{align}
Similarly, we have
\begin{align}
f(r)&=r^{-\beta}\int_0^{+\infty}\!\!\int_{-\frac\pi2}^{\frac\pi2}e^{2\pi i \rho r\sin \theta}\cos^{d-2} \theta\> \rho^{-\alpha+d-1} \mathcal F f(\rho)\,d\theta d\rho.\label{inverse-deformed-Fourier-radial}
\end{align}

\subsubsection{Definitions}
In this subsection, we define the incoming and outgoing components of functions and present their properties.
Our definition here is inspired by T. Tao \cite{Tao-2}. 
For convenience, we denote
\begin{align}\label{def-J}
J(r)=\int_0^{\frac\pi2}e^{2\pi i  r\sin \theta}\cos^{d-2} \theta\,d\theta.
\end{align}
Then we have
$$
J(-r)=\int_{-\frac\pi2}^0e^{2\pi i  r\sin \theta}\cos^{d-2} \theta\,d\theta.
$$
Let
$$
K( r)=\chi_{\ge 2}( r)\Big[-\frac{1}{2\pi i r}+\frac{d-3}{(2\pi i r)^3}\Big], \quad \mbox{for } d=3,4,5.
$$

\begin{definition}\label{def:outgong-incoming}
Let $\alpha<d, \beta>-d$, and the function $f\in L^1_{loc}(\R^d)$ be radial.
We define the incoming component of $f$ as
$$
f_{in}(r)=r^{-\beta}\int_0^{+\infty}\!\!\Big(J(-\rho r)+K(\rho r)\Big) \rho^{-\alpha+d-1} \mathcal F f(\rho)\, d\rho;
$$
the outgoing component of $f$ as
$$
f_{out}(r)=r^{-\beta}\int_0^{+\infty}\!\!\Big(J(\rho r)-K(\rho r)\Big) \rho^{-\alpha+d-1} \mathcal F f(\rho)\, d\rho.
$$
Moreover,  for any fixed  integer  $k$, we define the frequency restricted  incoming component of $f$ as
$$
f_{in,k}(r)=r^{-\beta}\int_0^{+\infty}\!\!\Big(J(-\rho r)+K(\rho r)\Big)\> \chi_{2^k}(\rho)\>\rho^{-\alpha+d-1} \mathcal F f(\rho)\, d\rho;
$$
correspondingly, the frequency restricted outgoing component of $f$ as
$$
f_{out,k}(r)=r^{-\beta}\int_0^{+\infty}\!\!\Big(J(\rho r)-K(\rho r)\Big) \> \chi_{2^k}(\rho)\>\rho^{-\alpha+d-1} \mathcal F f(\rho)\, d\rho.
$$
\end{definition}

\begin{remark}
If we consider the higher dimensional version, then $K(r)$ should be modified. Here we use an appropriately modified Fourier transform in order to cancel the singularity from $r^{-1}$ at origin and meanwhile to guarantee the boundedness of the incoming/outgoing projection operators on $L^2$. 
\end{remark} 

In the whole of the present paper, we set the numbers in Definition \ref{def:outgong-incoming} that
$$
\beta=\frac{d-1}{2}-2,\quad \mbox{ and }\quad  \alpha=0.
$$
(These numbers should be changed if one considers the cases when $d\ne 3,4,5$.)

From the definitions and \eqref{inverse-deformed-Fourier-radial}, we have
$$
f(r)=f_{out}(r)+f_{in}(r).
$$
Moreover,
$$
f_{out/in}(r)=\sum\limits_{k=-\infty}^{+\infty} f_{out/in,k}(r).
$$
Correspondingly, for $k_0\in \Z$, we denote
$$
f_{out/in,\ge k_0}(r)=\sum\limits_{k=k_0}^{+\infty} f_{out/in,k}(r).
$$

\subsection{Basic properties of the incoming/outgoing functions}
First, we give the estimates on the following oscillatory integrals. The first two can be regarded as some asymptotic behaviors of the restricted forms of $J(r)$.
\begin{lem}\label{lem:Jr} Let $d=3,4,5$.  There exists constant $c\in \C$ only depending on $d$, such that when $r\to +\infty$,
\begin{align}
\int_0^{\frac\pi2} e^{2\pi i  r\sin \theta}\chi_{\ge \frac\pi6}(\theta)\cos^{d-2} \theta  \,d\theta
=
c r^{-\frac{d-1}{2}} e^{2\pi i r} +O\big(r^{-\frac{d+1}{2}}\big);\label{1.04}
\end{align}
and
\begin{align}
\int_{-\frac\pi2}^{\frac\pi2}\!\!\!e^{-2\pi i  r\sin \theta}\chi_{\ge \frac\pi6}(\theta)\cos^{d-2} \theta  \,d\theta
&=
r^{-\frac{d-1}{2}}\Big(\bar c e^{-2\pi i r}+c  e^{2\pi ir}\Big) +O\big(r^{-\frac{d+1}{2}}\big).\label{1.05}
\end{align}
\end{lem}

\begin{remark}
The estimates presented in this lemma are not sharp. For example, the first estimate can be improved as the form of $r^{-\frac{d-1}{2}}\big(p_K(1/r)e^{2\pi ir}+O(r^{-K-1})\big)$ for any $K\in \Z^+$, where
$p_K$ is some polynomial function of  order  $K$. However, the version presented in this lemma is simpler and enough  in this paper.
\end{remark}
\begin{proof}[Proof of Lemma \ref{lem:Jr}]
We only prove the first estimate, since the second one is the sum of
$$
\int_0^{\frac\pi2} e^{2\pi i  r\sin \theta}\chi_{\ge \frac{\pi}{6}}(\theta)\cos^{d-2} \theta \,d\theta \quad \mbox{and}\quad
\int_0^{\frac\pi2} e^{-2\pi i  r\sin \theta}\chi_{\ge \frac{\pi}{6}}(\theta)\cos^{d-2} \theta \,d\theta.
$$

When $d=3$,
\begin{align*}
\int_0^{\frac\pi2} e^{2\pi i  r\sin \theta}&\chi_{\ge \frac{\pi}{6}}(\theta)\cos \theta \,d\theta
=\int_0^{1}e^{2\pi i  r s}\chi_{\ge \frac12}(s)\,d s\\
& =\frac{1}{2\pi i r}\int_0^{1}\partial_s\Big(e^{2\pi i  r s}\Big)\chi_{\ge \frac12}(s)\,d s
 =\frac{1}{2\pi i r} e^{2\pi ir}-\frac{1}{2\pi i r}\int_0^{1}e^{2\pi i  r s} \chi_{\ge \frac12}'(s)\,d s.
\end{align*}
By integration-by-parts $K$ times, we have
$$
\int_0^{1}e^{2\pi i  r s} \chi_{\ge \frac12}'(s)\,d s=O(r^{-K}).
$$
Hence we obtain that
\begin{align*}
\int_0^{\frac\pi2} e^{2\pi i  r\sin \theta}&\chi_{\ge \frac{\pi}{6}}(\theta)\cos \theta \,d\theta
=\frac{1}{2\pi i r} e^{2\pi ir}+O(r^{-K}).
\end{align*}

When $d=4$,  using the formula,
\begin{align}
2\pi i  r \cos \theta\cdot e^{2\pi i  r\sin \theta}=\partial_\theta e^{2\pi i  r\sin \theta}, \label{Formula:dev}
\end{align}
we have
\begin{align*}
\int_0^{\frac\pi2} e^{2\pi i  r\sin \theta}&\chi_{\ge \frac{\pi}{6}}(\theta)\cos^2 \theta \,d\theta
=\frac{1}{2\pi i r}\int_0^{\frac\pi2}\partial_\theta  e^{2\pi i  r\sin \theta}\chi_{\ge \frac{\pi}{6}}(\theta)\cos \theta \,d\theta\\
&=\frac{1}{2\pi i r}\int_0^{\frac\pi2}  e^{2\pi i  r\sin \theta}\Big(\chi_{\ge \frac{\pi}{6}}(\theta)\sin \theta
-\chi_{\ge \frac{\pi}{6}}'(\theta)\cos \theta\Big) \,d\theta.
\end{align*}
Note that
\begin{align*}
\int_0^{\frac\pi2}  e^{2\pi i  r\sin \theta}\chi_{\ge \frac{\pi}{6}}'(\theta)\cos \theta \,d\theta
=
\int_0^1  e^{2\pi i  rs}\chi_{\ge \frac12}'(s)  \,ds
=O(r^{-K}).
\end{align*}
Moreover, using Lemma \ref{lem:stationary-phase}, 
\begin{align*}
\int_0^{\frac\pi2}  e^{2\pi i  r\sin \theta}\chi_{\ge \frac{\pi}{6}}(\theta)\sin \theta
 \,d\theta
&=\frac12\int_0^{\pi}  e^{2\pi i  r\sin \theta}\chi_{\frac{\pi}{6}\le \cdot\le \frac{5\pi}{6} }(\theta)\sin \theta \,d\theta\\
&=
\Big(\frac{1}{2\pi ir}\Big)^\frac12 e^{2\pi i  r}+O(r^{-\frac32}).
\end{align*}
This gives the required estimate in dimension four.

When $d=5$, we have
\begin{align*}
\int_0^{\frac\pi2} e^{2\pi i  r\sin \theta}&\chi_{\ge \frac{\pi}{6}}(\theta)\cos^3 \theta \,d\theta
=\int_0^{1}e^{2\pi i  r s}\chi_{\ge \frac12}(s)(1-s^2)\,d s\\
&=\frac{1}{2\pi i r}\int_0^{1}\partial_s\Big(e^{2\pi i  r s}\Big)\chi_{\ge \frac12}(s)(1-s^2)\,d s.
\end{align*}
By integration-by-parts once and noting the zero boundary values, it is equal to
$$
\frac{1}{2\pi i r}\int_0^{1}e^{2\pi i  r s}\Big(-\chi_{\ge \frac12}'(s)(1-s^2)+2\chi_{\ge \frac12}(s)s\Big)\,d s.
$$
Now arguing similarly as above and integration-by-parts many times, we have
\begin{align*}
\frac{1}{2\pi i r}\int_0^{1}e^{2\pi i  r s}\chi_{\ge \frac12}'(s)(1-s^2)\,d s
&=O(r^{-K}),
\end{align*}
and
\begin{align*}
\frac{1}{2\pi i r}\int_0^{1}e^{2\pi i  r s}\chi_{\ge \frac12}(s)s\,d s
&=\frac{1}{(2\pi ir)^2}\Big(1-\frac{1}{2\pi ir}\Big)e^{2\pi i  r}+O(r^{-K}).
\end{align*}
This obeys the form as claimed. Hence we finish the proof of the lemma.
\end{proof}

Second, we have the following estimates.
\begin{lem} \label{lem:ocs2} When $r\to +\infty$,
\begin{align*}
\int_0^{\frac\pi2} e^{2\pi i  r\sin \theta}\chi_{\le \frac\pi6}(\theta)\cos^{d-2} \theta  \,d\theta &= -\frac{1}{2\pi i r}+\frac{d-3}{(2\pi i r)^3}+O\big(r^{-5}\big).
\end{align*}
Moreover,
\begin{align*}
\int_{-\frac\pi2}^{\frac\pi2} e^{2\pi i  r\sin \theta}\chi_{\le \frac\pi6}(\theta)\cos^{d-2} \theta  \,d\theta &=O\big(r^{-10}\big).
\end{align*}
\end{lem}
\begin{proof}
Using the formula \eqref{Formula:dev} and integration-by-parts,
\begin{align*}
\int_0^{\frac\pi2} &e^{2\pi i  r\sin \theta}\chi_{\le \frac\pi6}(\theta)\cos^{d-2} \theta  \,d\theta \\
&= \frac{1}{2\pi i r}\int_0^{\frac\pi2} \partial_\theta e^{2\pi i  r\sin \theta}\chi_{\le \frac\pi6}(\theta)\cos^{d-3} \theta  \,d\theta\\
&= -\frac{1}{2\pi i r}-\frac{1}{2\pi i r}\int_0^{\frac\pi2} e^{2\pi i  r\sin \theta}\eta_1(\theta) \,d\theta,
\end{align*}
where
$$
\eta_1(\theta) =\partial_\theta \Big(\chi_{\le \frac\pi6}(\theta)\cos^{d-3} \theta \Big).
$$
Note that $\eta_1(0)=\eta_1(\frac\pi2)=0$,  then using the formula \eqref{Formula:dev} and integration-by-parts again, we have
\begin{align*}
\int_0^{\frac\pi2} &e^{2\pi i  r\sin \theta}\chi_{\le \frac\pi6}(\theta)\cos^{d-2} \theta  \,d\theta \\
&= -\frac{1}{2\pi i r}+\frac{1}{(2\pi i r)^2}\int_0^{\frac\pi2} e^{2\pi i  r\sin \theta}\eta_2(\theta) \,d\theta,
\end{align*}
where
$$
\eta_2(\theta) =\partial_\theta \Big(\frac{\eta_1(\theta)}{\cos\theta} \Big).
$$
Note that $|\eta_2(\theta)|\lesssim 1$. Moreover, by a direct computation, we have
$$
\eta_2(0)=-(d-3), \quad \eta_2(\frac\pi2)=0.
$$
Then repeating the process above again, we obtain
\begin{align*}
\int_0^{\frac\pi2} &e^{2\pi i  r\sin \theta}\chi_{\le \frac\pi6}(\theta)\cos^{d-2} \theta  \,d\theta \notag\\
&= -\frac{1}{2\pi i r}+\frac{d-3}{(2\pi i r)^3}+\frac{1}{(2\pi i r)^3}\int_0^{\frac\pi2} e^{2\pi i  r\sin \theta}\eta_3(\theta) \,d\theta.
\end{align*}
Here
$$
\eta_3(\theta) =\partial_\theta \Big(\frac{\eta_2(\theta)}{\cos\theta} \Big).
$$
Moreover, we note that $\eta_3(0)=\eta_3(\frac\pi2)=0$. Then the same process
gives
\begin{align}
\int_0^{\frac\pi2} &e^{2\pi i  r\sin \theta}\chi_{\le \frac\pi6}(\theta)\cos^{d-2} \theta  \,d\theta \notag\\
&= -\frac{1}{2\pi i r}+\frac{d-3}{(2\pi i r)^3}+\frac{1}{(2\pi i r)^5}\int_0^{\frac\pi2} e^{2\pi i  r\sin \theta}\eta_4(\theta) \,d\theta, \label{12.29}
\end{align}
where the function $\eta_4$ satisfies $|\eta_4(\theta)|\lesssim 1$.
Then we obtain the desired result in the first estimate.

The second one, because of the zero boundary values in every step,   can be obtained by integration-by-parts 10 times.
\end{proof}

As consequences of the lemmas above, we have the following variant forms of $J(r)-K(r)$. The first one is
\begin{cor}\label{cor:Jr-osc1} Let $d=3,4,5$, then
\begin{align}
\chi_{\le 1}(r)\Big(J(r)&-K(r)\Big)=\chi_{\le 1}(r)\int_0^{\frac\pi2} e^{2\pi i  r\sin \theta} \cos^{d-2} \theta  \,d\theta;\label{20.14I}
\end{align}
and
\begin{align}
\chi_{\ge 1}&(r)\Big(J(r)-K(r)\Big)=\chi_{\ge 1}(r)\int_0^{\frac\pi2} e^{2\pi i  r\sin \theta}\chi_{\ge \frac\pi6}(\theta)\cos^{d-2} \theta  \,d\theta\notag\\
&+ \chi_{ 1\le \cdot \le 2}(r)\int_0^{\frac\pi2} e^{2\pi i  r\sin \theta}\chi_{\le \frac\pi6}(\theta) \cos^{d-2}\theta \,d\theta+\chi_{\ge 2}(r)\frac{1}{(2\pi ir)^5}\int_0^{\frac\pi2} e^{2\pi i r\sin \theta}\tilde\eta_d(\theta) \,d\theta,\label{20.14III}
\end{align}
where $\tilde \eta_d, d=3,4,5$ are uniformly bounded functions.
\end{cor}
\begin{proof}
From the definition of $K(r)$, we have $K(r)=0$ in the part of $\chi_{\le 1}(r)$. Hence, by the definition of $J(r)$ directly, we have \eqref{20.14I}.

For  \eqref{20.14III}, we split $J(r)$ into the following two parts,
$$
\int_0^{\frac\pi2} e^{2\pi i  r\sin \theta}\chi_{\ge \frac\pi6}(\theta) \cos^{d-2}\theta \,d\theta \quad
\mbox{ and }
\int_0^{\frac\pi2} e^{2\pi i  r\sin \theta}\chi_{\le \frac\pi6}(\theta) \cos^{d-2}\theta \,d\theta.
$$
For the latter part,  we use the identity \eqref{12.29} to obtain that
\begin{align*}
\chi_{\ge 2}(r)\int_0^{\frac\pi2} &e^{2\pi i  r\sin \theta}\chi_{\le \frac\pi6}(\theta)\cos^{d-2} \theta  \,d\theta 
=K(r)+\chi_{\ge 2}(r)\frac{1}{(2\pi i r)^5}\int_0^{\frac\pi2} e^{2\pi i  r\sin \theta}\eta_4(\theta) \,d\theta.
\end{align*}
Hence, setting $\tilde\eta_d=\eta_4$, we get \eqref{20.14III}.
Note that from the definitions in the proof of Lemma \ref{lem:ocs2}, $\eta_4$ are uniformly bounded functions whenever $d=3,4,5$.
This proves the lemma.
\end{proof}

The following result shows the asymptotic behaviors of $J(r)-K(r)$ and $J(r)+J(-r)$.
\begin{cor}\label{cor:Jr-osc2} Let $d=3,4,5$, then
\begin{align}
J(r)-K(r)&=\int_0^{\frac\pi2} e^{2\pi i  r\sin \theta}\chi_{\ge \frac\pi6}(\theta)\cos^{d-2} \theta  \,d\theta+O\big(\langle r\rangle^{-5}\big)\label{1.02}\\
&=O\big(\langle r\rangle^{-\frac{d-1}{2}}\big) ;\label{1.02-2}
\end{align}
and
\begin{align}
J(r)+J(-r)&=\int_{-\frac\pi2}^{\frac\pi2}e^{-2\pi i  r\sin \theta}\chi_{\ge \frac\pi6}(\theta)\cos^{d-2} \theta  \,d\theta+O\big(\langle r\rangle^{-10}\big).\label{1.03}
\end{align}
\end{cor}
\begin{proof}
Note that when $r\lesssim  1$, the following four functions: $J(r)-K(r)$, $J(r)+J(-r)$,
$$
\int_0^{\frac\pi2} e^{2\pi i  r\sin \theta}\chi_{\ge \frac\pi6}(\theta)\cos^{d-2} \theta  \,d\theta
\quad \mbox{and}\quad
\int_{-\frac\pi2}^{\frac\pi2}e^{-2\pi i  r\sin \theta}\chi_{\ge \frac\pi6}(\theta)\cos^{d-2} \theta  \,d\theta
$$
are uniformly bounded. Hence,  the estimates of \eqref{1.02} and \eqref{1.03} hold when $r\lesssim  1$.

When $r\gtrsim 1$, from \eqref{20.14III} , 
\begin{align*}
J(r)-K(r)&=\int_0^{\frac\pi2} e^{2\pi i  r\sin \theta}\chi_{\ge \frac\pi6}(\theta)\cos^{d-2} \theta  \,d\theta+O\big( r^{-5}\big);
\end{align*}
and from Lemma \ref{lem:ocs2}, 
\begin{align*}
J(r)+J(-r)&=\int_{-\frac\pi2}^{\frac\pi2}e^{-2\pi i  r\sin \theta}\chi_{\ge \frac\pi6}(\theta)\cos^{d-2} \theta  \,d\theta+O\big( r^{-10}\big).
\end{align*}
Hence, we obtain \eqref{1.02} and \eqref{1.03}. Combining with \eqref{1.02} and \eqref{1.04}, we obtain \eqref{1.02-2}. This finishes the proof of the corollary.
\end{proof}

\subsubsection{Matching estimates in frequency and physical spaces}
In this subsection, we will show that the incoming/outgoing projection would almost preserve the localization in both space and frequency. The first result below is related to the preservation of frequency. Specifically, if a function $f$ has high frequency $f=P_{2^k} f$, then its incoming/outgoing component will  have almost the same frequency plus a smooth perturbation.
\begin{prop}\label{prop:f++-highfreq}
Let $k_0\ge 0$.
Suppose that $f\in L^2(\R^d)$ with supp$f\subset \{x:|x|\ge 1\}$, then
$$
\big(P_{\ge 2^{k_0}}f\big)_{out/in}=\big(P_{\ge 2^{k_0}}f\big)_{out/in,\ge k_0-1}+h
$$
with
\begin{align}
\|h\|_{H^{\mu(d)}(\R^d)}\lesssim 2^{-k_0}\|f\|_{H^{-10}(\R^d)}.\label{est:h}
\end{align}
Here $\mu(d)=2$ if $d=3,4$, and $\mu(5)=3$. Similarly,
$$
\big(P_{\le 2^{k_0}}f\big)_{out/in}=\big(P_{\le 2^{k_0}}f\big)_{out/in,\le k_0+1}+\tilde h
$$
with $\tilde h$ verifying the same estimate to \eqref{est:h}.
\end{prop}
\begin{proof}
We only give the estimate of the first part, since the second one can be treated in the same manner.
Moreover, considering the support of $f$, for simplifying the notations, we write
$$
f=\chi_{\ge 1}f.
$$
To give the desired estimate,
we use the Littlewood-Paley decomposition and write
$$
P_{\ge 2^{k_0}}f=\sum\limits_{k=k_0}^{+\infty} P_{2^k}f.
$$
From \eqref{deformed-Fourier}, we note that
$$
\mathcal F (P_{2^{k}}f)(\rho)= \mathscr F\Big( |x|^\beta P_{2^{k}}f\Big)(\rho).
$$
Accordingly,  we decompose $\mathcal F (P_{2^{k}}f)(\rho)$ into the following two parts,
\begin{align*}
 \chi_{\ge  2^{k_0-1}} &(\rho)\>\mathscr F\Big(|x|^\beta P_{2^{k}}f\Big)(\rho)
+ \>\chi_{\le  2^{k_0-1}} (\rho) \mathscr F\Big( |x|^\beta P_{2^{k}}f\Big)(\rho).
\end{align*}
Now due to the support of $f$, we further decompose the expression above as
\begin{align*}
 \chi_{\ge  2^{k_0-1}} &(\rho)\>\mathscr F\Big(|x|^\beta P_{2^{k}}f\Big)(\rho)
+ \>\chi_{\le  2^{k_0-1}} (\rho) \mathscr F\Big(\chi_{\ge  \frac12} |x|^\beta \cdot P_{2^{k}}\chi_{\ge  1}f\Big)(\rho)\\
&\qquad + \>\chi_{\le  2^{k_0-1}} (\rho)\mathscr F\Big( \chi_{\le \frac12}(x)\,|x|^\beta\cdot P_{2^{k}}\chi_{\ge 1}f\Big)(\rho).
\end{align*}
We denote
\begin{align*}
g_1(\rho)&= \>\chi_{\le  2^{k_0-1}} (\rho) \mathscr F\Big(\chi_{\ge  \frac12}(x)\, |x|^\beta\cdot  P_{2^{k}}\chi_{\ge  1}f\Big)(\rho);\\
g_2(\rho)&= \>\chi_{\le  2^{k_0-1}} (\rho)\mathscr F\Big( \chi_{\le \frac12}(x)\,|x|^\beta\cdot P_{2^{k}}\chi_{\ge 1}f\Big)(\rho).
\end{align*}
Then we have
\begin{align}
\mathcal F (P_{2^{k}}f)(\rho)
=\chi_{\ge  2^{k_0-1}}(\rho)\>\mathcal F (P_{2^{k}}f)(\rho)+g_1+g_2.\label{13.03}
\end{align}
From Definition \ref{def:outgong-incoming}, we have
\begin{align*}
\big(P_{2^{k}}f\big)_{out}&=r^{-\beta}\int_0^{+\infty}\!\!\Big(J(\rho r)-K(\rho r)\Big) \rho^{d-1} \mathcal F (P_{2^{k}}f)(\rho)\,d\rho,
\end{align*}
thus using  \eqref{13.03}, we write
\begin{align*}
\big(P_{2^{k}}f\big)_{out}=&r^{-\beta}\int_0^{+\infty}\!\!\Big(J(\rho r)-K(\rho r)\Big) \rho^{d-1} \chi_{\ge  2^{k_0-1}}(\rho)\>\mathcal F (P_{2^{k}}f)(\rho)\,d\rho\\
&\qquad+r^{-\beta}\int_0^{+\infty}\!\!\Big(J(\rho r)-K(\rho r)\Big) \rho^{d-1}\>g_1(\rho)\,d\rho\\
&\qquad+r^{-\beta}\int_0^{+\infty}\!\!\Big(J(\rho r)-K(\rho r)\Big) \rho^{d-1}\>g_2(\rho)\,d\rho.
\end{align*}
Denote
\begin{align}
h_k(r)=&r^{-\beta}\int_0^{+\infty}\!\!\Big(J(\rho r)-K(\rho r)\Big) \rho^{d-1} \>g_1(\rho)\,d\rho\notag\\
&+r^{-\beta}\int_0^{+\infty}\!\!\Big(J(\rho r)-K(\rho r)\Big) \rho^{d-1} \>g_2(\rho)\,d\rho,\label{def:hk}
\end{align}
then
$$
\big(P_{2^{k}}f\big)_{out}=\big(P_{2^{k}}f\big)_{out,\ge k_0-1}+h_k.
$$
Moreover, let
\begin{align*}
h(r)=\sum\limits_{k=k_0}^{+\infty} h_k(r),
\end{align*}
then
$$
\big(P_{\ge 2^{k_0}}f\big)_{out/in}=\big(P_{\ge 2^{k_0}}f\big)_{out/in,\ge k_0-1}+h.
$$
Now we give the control of $h_k$. To this end, we first claim that for any $M\in \Z^+$,
\begin{align}
\|g_1\|_{L^\infty_\rho}+\|g_2\|_{L^\infty_\rho}
\lesssim_M 2^{-Mk}\big\|P_{2^k}f\big\|_{L^2(\R^d)}.\label{est:g1g2}
\end{align}
To prove this claim, we note that
\begin{align*}
g_1(\rho)&= \>\chi_{\le  2^{k_0-1}} (\rho) \mathscr F\Big(P_{\ge 2^{k-2}}\big(\chi_{\ge  \frac12} |x|^\beta\big)\cdot  P_{2^{k}} f\Big)(\rho).
\end{align*}
Here we have used the fact $k\ge k_0$. Then by Hausdorff-Young's inequality, 
\begin{align*}
\|g_1\|_{L^\infty_\rho}
&\lesssim
\Big\|\mathscr F\Big(P_{\ge 2^{k-2}}\big(\chi_{\ge  \frac12} |x|^\beta\big)\cdot  P_{2^{k}} f\Big)\Big\|_{L^\infty_\rho}\\
&\lesssim
\big\|P_{\ge 2^{k-2}}\big(\chi_{\ge  \frac12} |x|^\beta\big)\cdot  P_{2^{k}} f\big\|_{L^1_x(\R^d)}\\
&\lesssim
\big\|P_{\ge 2^{k-2}}\big(\chi_{\ge  \frac12} |x|^\beta\big)\big\|_{L^2_x(\R^d)} \big\| P_{2^{k}} f\big\|_{L^2_x(\R^d)}.
\end{align*}
Note that
$$
\big\|P_{\ge 2^{k-2}}\big(\chi_{\ge  \frac12} |x|^\beta\big)\big\|_{L^2_x(\R^d)}
\lesssim 2^{-Mk}\big\|(-\Delta)^\frac {M}2 P_{\ge 2^{k-2}}\big(\chi_{\ge  \frac12} |x|^\beta\big)\big\|_{L^2_x(\R^d)}
\lesssim 2^{-Mk},
$$
thus we obtain that
\begin{align*}
\|g_1\|_{L^\infty_\rho}
&\lesssim   2^{-Mk}\big\|P_{2^k}f\big\|_{L^2(\R^d)}.
\end{align*}

For $g_2$, we rewrite it as
$$
g_2(\rho)= \>\chi_{\le  2^{k_0-1}} (\rho)\mathscr F\Big( \chi_{\le \frac12}(x)\,|x|^\beta\cdot \chi_{\le \frac34}(x)\,P_{2^{k}}\chi_{\ge 1}f\Big)(\rho).
$$
Then
\begin{align*}
\|g_2\|_{L^\infty_\rho}
&\lesssim
\Big\|\mathscr F\Big(\chi_{\le \frac12}(x)\,|x|^\beta\cdot \chi_{\le \frac34}(x)\,P_{2^{k}}\chi_{\ge 1}f\Big)\Big\|_{L^\infty_\rho}\\
&\lesssim
\big\|\chi_{\le \frac12}(x)\,|x|^\beta\cdot \chi_{\le \frac34}(x)\,P_{2^{k}}\chi_{\ge 1}f\big\|_{L^1_x(\R^d)}\\
&\lesssim
\big\|\chi_{\le \frac12}(x)\,|x|^\beta\big\|_{L^2_x(\R^d)} \big\|  \chi_{\le \frac34}\,P_{2^{k}}\chi_{\ge 1}f\big\|_{L^2_x(\R^d)}.
\end{align*}
Since $\beta=\frac{d-5}{2}>-\frac d2$, we have 
$$
\big\|\chi_{\le \frac12}\,|x|^\beta\big\|_{L^2_x(\R^d)}\lesssim 1.
$$
Moreover, by the mismatch estimate given in Lemma \ref{lem:mismatch}, we have
\begin{align*}
\big\|  \chi_{\le \frac34}\,P_{2^{k}}\chi_{\ge 1}f\big\|_{L^2_x(\R^d)}
\lesssim
2^{-Mk}\big\|P_{2^k}f\big\|_{L^2(\R^d)}.
\end{align*}
Hence, we obtain that
\begin{align*}
\|g_2\|_{L^\infty_\rho}
&\lesssim   2^{-Mk}\big\|P_{2^k}f\big\|_{L^2(\R^d)}.
\end{align*}
This together with the estimate on $g_1$, gives the claim \eqref{est:g1g2}.

Based on \eqref{est:g1g2}, we consider $\|h_k\|_{L^2(\R^d)}$. Note that
\begin{align*}
\|h_k\|_{L^2(\R^d)}=&c\big\|r^{\frac{d-1}{2}}h_k(r)\big\|_{L^2_r}.
\end{align*}
Further, 
\begin{align*}
r^{\frac{d-1}{2}}h_k(r)
=&r^2\int_0^{+\infty}\!\!\Big(J(\rho r)-K(\rho r)\Big) \rho^{d-1} \>g_1(\rho)\,d\rho\\
&+r^2\int_0^{+\infty}\!\!\Big(J(\rho r)-K(\rho r)\Big) \rho^{d-1} \>g_2(\rho)\,d\rho.
\end{align*}
We only consider the first term. Since  the estimates for $g_1,g_2$ are the same,  see \eqref{est:g1g2},  they can be treated in the same way.
To do this, we split it into the following two parts again,
\begin{subequations}\label{h-12}
\begin{align}
r^2\int_0^{+\infty} J(\rho r)\chi_{\le 1}(\rho r)\rho^{d-1} \>g_1(\rho)\,d\rho,\label{h-Smaller1}
\end{align}
and
\begin{align}
r^2\int_0^{+\infty} \big(J(\rho r)-K(\rho r)\big)\chi_{\ge 1}(\rho r)\rho^{d-1} \>g_1(\rho)\,d\rho.\label{h-Larger1}
\end{align}
\end{subequations}
Since $|J(\rho r)|\lesssim 1$, \eqref{h-Smaller1} can be controlled by
\begin{align*}
\int_0^{+\infty} \chi_{\le 1}(\rho r)\rho^{d-3} \>\big|g_1(\rho)\big|\,d\rho.
\end{align*}
Since $d\ge 3$,
\begin{align*}
\|\eqref{h-Smaller1}\|_{L^2_r}
\lesssim &
\Big\|\int_0^{+\infty} \chi_{\le 1}(\rho r)\rho^{d-3}\chi_{\le  2^{k_0-1}} (\rho)\>\big|g_1(\rho)\big|\,d\rho\Big\|_{L^2_r}\\
\lesssim &
\int_0^{+\infty} \rho^{d-\frac72}\chi_{\le  2^{k_0-1}} (\rho)\>\big|g_1(\rho)\big|\,d\rho\\
\lesssim &
2^{(d-\frac52)k_0-Mk}\big\|P_{2^k}f\big\|_{L^2(\R^d)}.
\end{align*}
Then, choosing $M$ large enough, we have 
\begin{align}
\|\eqref{h-Smaller1}\|_{L^2_r}
\lesssim
2^{-\frac12 Mk}\big\|P_{2^k}f\big\|_{L^2(\R^d)}. \label{est:h-Smaller1}
\end{align}

For \eqref{h-Larger1}, according to \eqref{1.02}, we split it into the following two parts,
\begin{subequations}\label{h-Larger}
\renewcommand{\theequation}
{\theparentequation-\arabic{equation}}
\begin{align}
r^2\int_0^{+\infty}\!\!\!\int_0^{\frac\pi2} e^{2\pi i  \rho r\sin \theta}\chi_{\ge \frac\pi6}(\theta)\cos^{d-2} \theta  \,d\theta \chi_{\ge 1}(\rho r)\rho^{d-1} \>g_1(\rho)\,d\rho,\label{h-Larger1-1}
\end{align}
and
\begin{align}
r^2\int_0^{+\infty}O\big(\langle \rho r\rangle^{-5}\big) \chi_{\ge 1}(\rho r)\rho^{d-1} \>g_1(\rho)\,d\rho.\label{h-Larger1-2}
\end{align}
\end{subequations}
To estimate \eqref{h-Larger1-1}, we use the formula
$$
 e^{2\pi i  \rho r\sin \theta}=\frac1{(2\pi i r\sin \theta)^2}  \partial_\rho^2 e^{2\pi i  \rho r\sin \theta}
$$
to reduce it as
\begin{align*}
\int_0^{+\infty}\!\!\!\int_0^{\frac\pi2} \frac1{(2\pi i  \sin \theta)^2}  \partial_\rho^2 e^{2\pi i  \rho r\sin \theta}\chi_{\ge \frac\pi6}(\theta)\cos^{d-2} \theta  \,d\theta \chi_{\ge 1}(\rho r)\rho^{d-1} \>g_1(\rho)\,d\rho.
\end{align*}
By integration-by-parts, it is equal to
\begin{align*}
\int_0^{+\infty}\!\!\!\int_0^{\frac\pi2} \frac1{(2\pi i  \sin \theta)^2} e^{2\pi i  \rho r\sin \theta}\chi_{\ge \frac\pi6}(\theta)\cos^{d-2} \theta  \,d\theta  \>\partial_\rho^2\Big[\chi_{\ge 1}(\rho r)\rho^{d-1} \>g_1(\rho)\Big]\,d\rho.
\end{align*}
Then by Corollary \ref{cor:L2-estimate-FIO}, we have
\begin{align*}
\big\|\eqref{h-Larger1-1}\big\|_{L^2_r}
\lesssim &
\int_0^{\frac\pi2} \frac1{(2\pi  \sin \theta)^\frac52} \chi_{\ge \frac\pi6}(\theta)\cos^{d-2} \theta  \,d\theta\\
&\quad  \cdot
 \Big(\big\|\rho^{d-3} \>g_1(\rho)\big\|_{L^2_\rho}+\big\|\rho^{d-2} \>\partial_\rho g_1(\rho)\big\|_{L^2_\rho}+\big\|\rho^{d-1} \>\partial_\rho^2 g_1(\rho)\big\|_{L^2_\rho}\Big).
\end{align*}
Note that
\begin{align*}
\big\|\rho^{d-3} \>g_1(\rho)\big\|_{L^2_\rho}
\lesssim &
\big\|\chi_{\le  2^{k_0-1}} (\rho)\rho^{d-3}\big\|_{L^2_\rho}\|g_1\|_{L^\infty_\rho}\\
\lesssim &
2^{(d-\frac52)k_0-Mk}\big\|P_{2^k}f\big\|_{L^2(\R^d)}
\lesssim
2^{-\frac12 Mk}\big\|P_{2^k}f\big\|_{L^2(\R^d)};
\end{align*}
and furthermore, by the Plancherel identity, 
\begin{align*}
\big\|\rho^{d-2} \>\partial_\rho g_1(\rho)\big\|_{L^2_\rho}
\lesssim &
2^{\frac{d-3}2k_0}\Big\|\rho^{\frac{d-1}2} \chi_{\le  2^{k_0-1}} (\rho) \>\nabla_\xi \mathscr F\Big(\big(\chi_{\ge  \frac12} |x|^\beta\big)\cdot  P_{2^{k}} f\Big)(\rho)\Big\|_{L^2_\rho}\\
\lesssim &
2^{\frac{d-3}2k_0}\Big\|\chi_{\le  2^{k_0-1}} (\xi)\>\nabla \mathscr F\Big(P_{\ge 2^{k-2}}\big(\chi_{\ge  \frac12} |x|^\beta \big)\cdot  P_{2^{k}} f\Big)(\xi)\Big\|_{L^2_\xi(\R^d)}\\
\lesssim &
2^{\frac{d-3}2k_0}\big\|P_{\ge 2^{k-2}}\big(\chi_{\ge  \frac12} |x|^\beta x\big)\cdot  P_{2^{k}} f\big\|_{L^2(\R^d)}\\
\lesssim &
2^{\frac{d-3}2k_0}\big\|P_{\ge 2^{k-2}}\big(\chi_{\ge  \frac12} |x|^\beta x\big)\big\|_{L^\infty(\R^d)} \big\|P_{2^{k}} f\big\|_{L^2(\R^d)}\\
\lesssim &
2^{-\frac12 Mk}\big\|P_{2^{k}} f\big\|_{L^2(\R^d)}.
\end{align*}
Moreover,  using the following estimate instead, 
$$
\big\|P_{\ge 2^{k-2}}\big(\chi_{\ge  \frac12} |x|^{\beta+2}\big)\big\|_{L^\infty(\R^d)}\lesssim 2^{-Mk},
$$
we also have 
\begin{align*}
\big\|\rho^{d-1} \>\Delta g_1\big\|_{L^2_\rho}
\lesssim &
2^{-\frac12 Mk}\big\|P_{2^k}f\big\|_{L^2(\R^d)}.
\end{align*}
Since $\partial_{\rho}^2=\Delta - \frac{d-1}{\rho}\partial_\rho$, by the estimates above, we get 
$$
\big\|\rho^{d-1} \>\partial_\rho^2 g_1(\rho)\big\|_{L^2_\rho}\lesssim 2^{-\frac12 Mk}\big\|P_{2^k}f\big\|_{L^2(\R^d)}.
$$
Therefore,
\begin{align*}
\big\|\eqref{h-Larger1-1}\big\|_{L^2_r}
\lesssim 2^{-\frac12 Mk}\big\|P_{2^k}f\big\|_{L^2(\R^d)}.
\end{align*}
For \eqref{h-Larger1-2}, it is bounded by
\begin{align*}
r^{-1}\int_0^{+\infty} \chi_{\ge 1}(\rho r)\rho^{d-4} \>|g_1(\rho)|\,d\rho.
\end{align*}
Similar as \eqref{h-Smaller1},
\begin{align*}
\|\eqref{h-Larger1-2}\|_{L^2_r}
\lesssim
2^{-\frac12 Mk}\big\|P_{2^k}f\big\|_{L^2(\R^d)}.
\end{align*}
Combining the above two estimates for \eqref{h-Larger}, we obtain
\begin{align*}
\|\eqref{h-Larger1}\|_{L^2_r}
\lesssim
2^{-\frac12 Mk}\big\|P_{2^k}f\big\|_{L^2(\R^d)}.
\end{align*}
Then collecting the above two findings for \eqref{h-12}, we have
$$
\|h_k\|_{L^2(\R^d)}\lesssim 2^{-\frac12 Mk}\big\|P_{2^k}f\big\|_{L^2(\R^d)}.
$$

Now we consider the estimates of $h_k$ with the high-order derivatives. For $\|\Delta h_k\|_{L^2(\R^d)}$, it is equivalent to
$$
\big\|r^{\frac{d-1}2}\partial_{rr}h_k\big\|_{L^2_r}+\big\|r^{\frac{d-1}2-1}\partial_{r}h_k\big\|_{L^2_r}.
$$
Note that from the definition of $h_k$ in \eqref{def:hk}, both of $r^{\frac{d-1}2}\partial_{rr}h_k$ and $r^{\frac{d-1}2-1}\partial_{r}h_k$ are the combination of the following three parts,
\begin{align}
\int_0^{+\infty}\!\!\Big(J(\rho r)-K(\rho r)\Big) \rho^{d-1} \>g_j(\rho)\,d\rho;\notag\\
r\int_0^{+\infty}\!\!\partial_r\Big(J(\rho r)-K(\rho r)\Big) \rho^{d-1} \>g_j(\rho)\,d\rho;\label{17.07}
\end{align}
and
\begin{align*}
r^2\int_0^{+\infty}\!\!\partial_r^2\Big(J(\rho r)-K(\rho r)\Big) \rho^{d-1} \>g_j(\rho)\,d\rho,
\end{align*}
where $j=1,2$. When $d=5$, we also need to estimate $\|\nabla \Delta h_k\|_{L^2(\R^d)}$, which is equivalent to
$$
\big\|r^{\frac{d-1}2}\partial_{rrr}h_k\big\|_{L^2_r}+\big\|r^{\frac{d-1}2-1}\partial_{rr}h_k\big\|_{L^2_r}+\big\|r^{\frac{d-1}2-2}\partial_{r}h_k\big\|_{L^2_r}.
$$
Note that  $\beta=0$ in this case, hence all of $r^{\frac{d-1}2}\partial_{rrr}h_k$, $r^{\frac{d-1}2-1}\partial_{rr}h_k$ and $r^{\frac{d-1}2-2}\partial_{r}h_k$ are  the combination of the following three parts,
\begin{align*}
\int_0^{+\infty}\!\!\partial_r\Big(J(\rho r)-K(\rho r)\Big) \rho^{d-1} \>g_j(\rho)\,d\rho;\\
r\int_0^{+\infty}\!\!\partial_r^2\Big(J(\rho r)-K(\rho r)\Big) \rho^{d-1} \>g_j(\rho)\,d\rho;
\end{align*}
and
\begin{align*}
r^2\int_0^{+\infty}\!\!\partial_r^3\Big(J(\rho r)-K(\rho r)\Big) \rho^{d-1} \>g_j(\rho)\,d\rho.
\end{align*}
(Note that there is no singularity in $r$, due to $\beta=0$.)

By the formulas given in Corollary \ref{cor:Jr-osc1}, we have the explicit form of $\partial_r^l\Big(J(\rho r)-K(\rho r)\Big)$ for $l=1,2,3$, then similar argument can be used to estimate all the terms which can be dominated by $2^{-\frac12 Mk}\big\|P_{2^k}f\big\|_{L^2(\R^d)}$. For the sake of completeness, we take the term \eqref{17.07}  for example and give the estimation.
To do this, we split \eqref{17.07} into the following two parts,
\begin{subequations}\label{20.12-13}
\begin{align}
r\int_0^{+\infty}\!\!\partial_r\Big(\big(J(\rho r)-K(\rho r)\big)\chi_{\le 1}(\rho r)\Big) \rho^{d-1} \>g_j(\rho)\,d\rho; \label{20.12}
\end{align}
and
\begin{align}
r\int_0^{+\infty}\!\!\partial_r\Big(\big(J(\rho r)-K(\rho r)\big)\chi_{\ge 1}(\rho r)\Big) \rho^{d-1} \>g_j(\rho)\,d\rho.\label{20.13}
\end{align}
\end{subequations}
For \eqref{20.12}, we use the following estimate which is from \eqref{20.14I},
\begin{align*}
\Big|\partial_r\Big(\big(J(\rho r)-K(\rho r)\big)\chi_{\le 1}(\rho r)\Big)\Big|\lesssim \rho \chi_{\le 1}(\rho r).
\end{align*}
Then \eqref{20.12} can be dominated by
\begin{align*}
\int_0^{+\infty}\chi_{\le 1}(\rho r) \rho^{d-1} \>\big|g_j(\rho)\big|\,d\rho.
\end{align*}
Hence,
\begin{align*}
\| \eqref{20.12} \|_{L^2_r}
\lesssim &
\int_0^{+\infty}  \rho^{d-\frac32} \>\big|g_j(\rho)\big|\,d\rho\\
\lesssim & 2^{(d-\frac12)k_0}\|g_j\|_{L^\infty_\rho}
\lesssim 2^{-\frac12 Mk}\big\|P_{2^k}f\big\|_{L^2(\R^d)}.
\end{align*}
For \eqref{20.13}, we use the following estimate which is from \eqref{20.14III},
\begin{align*}
\partial_r\Big(\chi_{\ge 1}(\rho r)\big(J(\rho r)-K(\rho r)\big)\Big)=2\pi i\rho\> \chi_{\ge 1}(\rho r) &\int_0^{\frac\pi2} e^{2\pi i  \rho r\sin \theta}\chi_{\ge \frac\pi6}(\theta)\sin \theta\cos^{d-2} \theta  \,d\theta\\
&+\chi_{\gtrsim 1}(\rho r)\cdot O\big(\rho^{-4}r^{-5}\big).
\end{align*}
Accordingly, \eqref{20.13} can be split into the following two subparts again,
\begin{subequations}
\renewcommand{\theequation}
{\theparentequation-\arabic{equation}}
\begin{align}
2\pi i r\int_0^{+\infty}\!\!\int_0^{\frac\pi2} e^{2\pi i \rho r\sin \theta}\chi_{\ge \frac\pi6}(\theta)\sin \theta\cos^{d-2} \theta  \,d\theta\chi_{\ge 1}(\rho r) \rho^{d} \>g_j(\rho)\,d\rho;\label{20.13-I}
\end{align}
and
\begin{align}
r\int_0^{+\infty}\!\!O\big(\rho^{-4}r^{-5}\big)\chi_{\gtrsim 1}(\rho r) \rho^{d-1} \>g_j(\rho)\,d\rho.\label{20.13-II}
\end{align}
\end{subequations}
Then the part \eqref{20.13-I} can be treated similarly as \eqref{h-Larger1-1}; the part \eqref{20.13-II} can be treated similarly as \eqref{h-Larger1-2},
and thus we have 
\begin{align*}
\|\eqref{20.13}\|_{L^2_r}\le  \| \eqref{20.13-I} \|_{L^2_r}+\| \eqref{20.13-II} \|_{L^2_r}
\lesssim   2^{-\frac12 Mk}\big\|P_{2^k}f\big\|_{L^2(\R^d)}.
\end{align*}
Therefore, from the estimates on \eqref{20.12-13}, we get
\begin{align*}
\| \eqref{17.07} \|_{L^2_r}
\lesssim   2^{-\frac12 Mk}\big\|P_{2^k}f\big\|_{L^2(\R^d)}.
\end{align*}
Combining the estimates above, we establish that
$$
\|h_k\|_{H^{\mu(d)}(\R^d)}\lesssim 2^{-\frac12 Mk}\big\|P_{2^k}f\big\|_{L^2(\R^d)}.
$$
Therefore, by summation in $k$ and choosing $M$ suitably large, we obtain that
$$
\|h\|_{H^{\mu(d)}(\R^d)}\lesssim 2^{-k_0}\big\| f\big\|_{H^{-10}(\R^d)}.
$$
This finishes the proof of the proposition.
 \end{proof}

The following result shows that if $f$ is supported outside of a ball, then $f_{out/in}$ is also almost supported outside of the ball.
\begin{lem}\label{lem:supportf+}
Let $\mu(d)$ be defined in Proposition \ref{prop:f++-highfreq}. Suppose that supp$f\subset \{x:|x|>1\}$, then
\begin{align*}
\big\|\chi_{\le \frac 14} (P_{\ge 1}f)_{out/in}\big\|_{H^{\mu(d)}(\R^d)}
\lesssim \|f\|_{H^{-1}(\R^d)},
\end{align*}
and for any $k\in \Z^+$,
\begin{align*}
\big\|\chi_{\le \frac 14} (P_{2^k}f)_{out/in, \ge k-1}\big\|_{H^{\mu(d)}(\R^d)}
\lesssim 2^{-2k}\|P_{2^k}f\|_{L^2(\R^d)}.
\end{align*}
\end{lem}
\begin{proof}
We only consider the estimates on the outgoing part $f_{out}$, because the ones on the incoming part can be proved in a similar way.
Using the Littlewood-Paley decomposition, we write
$$
P_{\ge 1}f=\sum\limits_{k=0}^{+\infty} P_{2^k}f.
$$
Then
\begin{align*}
\big\|\chi_{\le \frac 14} \big(P_{\ge 1}f\big)_{out/in}\big\|_{H^{\mu(d)}(\R^d)}
\lesssim
\sum\limits_{k=0}^{+\infty}\big\|\chi_{\le \frac 14} \big(P_{2^k}f\big)_{out}\big\|_{H^{\mu(d)}(\R^d)}.
\end{align*}
Using Proposition \ref{prop:f++-highfreq}, we have 
$$
\big(P_{2^k}f\big)_{out}= \big(P_{2^k}f\big)_{out,\ge k-1}+h, \quad \mbox{with }\quad \|h\|_{H^{\mu(d)}(\R^d)}\lesssim 2^{-k}\big\|P_{2^k}f\big\|_{H^{-10}(\R^d)}.
$$
Hence, 
\begin{align*}
\big\|\chi_{\le \frac 14} \big(P_{2^k}f\big)_{out}\big\|_{H^{\mu(d)}(\R^d)}
\lesssim
\big\|\chi_{\le \frac 14} \big(P_{2^k}f\big)_{out,\ge k-1}\big\|_{H^{\mu(d)}(\R^d)}+
\big\|P_{2^k}f\big\|_{H^{-10}(\R^d)}.
\end{align*}
Therefore, we only need to consider $\big\|\chi_{\le \frac 14} \big(P_{2^k}f\big)_{out,\ge k-1}\big\|_{H^{\mu(d)}(\R^d)}$. We first consider the $L^2$-norm,
which is equal to
$$
\big\|r^{\frac{d-1}{2}}\chi_{\le \frac14}(r)\big(P_{2^k}f\big)_{out,\ge k-1}\big\|_{L^2_r}.
$$
From the definition, we have
\begin{align*}
r^{\frac{d-1}{2}}\big(P_{2^k}f\big)_{out,\ge k-1}
=& r^2\int_0^{+\infty}\big(J(\rho r)-K(\rho r)\big)\chi_{\ge 2^{k-1}}(\rho) \rho^{d-1}\mathcal F(P_{2^k}f)(\rho)\,d\rho.
\end{align*}
As in the proof of the previous proposition, we split it into the following two parts,
\begin{subequations}
\begin{align}
r^2\int_0^{+\infty}\big(J(\rho r)-K(\rho r)\big)\chi_{\le 1}(\rho r)\chi_{\ge 2^{k-1}}(\rho) \rho^{d-1}\mathcal F(P_{2^k}f)(\rho)\,d\rho;\label{0.21}
\end{align}
and
\begin{align}
r^2\int_0^{+\infty}\big(J(\rho r)-K(\rho r)\big)\chi_{\ge 1}(\rho r)\chi_{\ge 2^{k-1}}(\rho) \rho^{d-1}\mathcal F(P_{2^k}f)(\rho)\,d\rho.\label{0.22}
\end{align}
\end{subequations}

For \eqref{0.21}, using the definition of $\mathcal F(P_{2^k}f)$, we rewrite it as
\begin{align*}
r^2\int_0^{+\infty} J(\rho r)\chi_{\le 1}(\rho r)\chi_{\ge 2^{k-1}}(\rho) \rho^{d-1}\int_0^{+\infty}\!\!\!\Big(J(\rho s)+J(-\rho s)\Big)s^{\beta+d-1}P_{2^k}f(s)\,ds d\rho.
\end{align*}
Using \eqref{1.03} and \eqref{def-J}, we split \eqref{0.21} into the following two subparts again,
\begin{subequations}\label{0.44-45}
\renewcommand{\theequation}
{\theparentequation-\arabic{equation}}
\begin{align}
r^2\int_0^{+\infty}\!\!\!& \int_0^{\frac\pi2}e^{2\pi i \rho r\sin \theta}\cos^{d-2} \theta  \,d\theta\,\chi_{\le 1}(\rho r)\chi_{\ge 2^{k-1}}(\rho) \rho^{d-1}\notag\\
&\cdot\int_0^{+\infty}\!\!\!\int_{-\frac\pi2}^{\frac\pi2}e^{-2\pi i \rho s\sin \theta'}\chi_{\ge \frac\pi6}(\theta')\cos^{d-2} \theta'  \,d\theta' s^{\beta+d-1}P_{2^k}f(s)\,ds d\rho,\label{0.44}
\end{align}
and
\begin{align}
r^2\int_0^{+\infty}& J(\rho r)\chi_{\le 1}(\rho r)\chi_{\ge 2^{k-1}}(\rho) \rho^{d-1}
\int_0^{+\infty}\!\!\!O\big( \langle\rho s\rangle^{-10}\big) s^{\beta+d-1}P_{2^k}f(s)\,ds d\rho.\label{0.45}
\end{align}
\end{subequations}
The term \eqref{0.44} can be rewritten as
\begin{align*}
r^2\int_0^{\frac\pi2}\int_{-\frac\pi2}^{\frac\pi2}
\int_0^{+\infty}\!\!\!\int_0^{+\infty}\!\!\!& e^{2\pi i \rho (r\sin \theta-s\sin \theta')}\chi_{\le 1}(\rho r)\chi_{\ge 2^{k-1}}(\rho) \rho^{d-1}s^{\beta+d-1}P_{2^k}f(s)\,ds d\rho\notag\\
&\cdot\chi_{\ge \frac\pi6}(\theta')  \cos^{d-2} \theta \cos^{d-2} \theta' \,d\theta' d\theta.
\end{align*}
When $r\le \frac13,s\ge \frac{9}{10}$, $|\sin \theta'|\ge \frac25$, we have
\begin{align}\label{7.12-11.32}
|r\sin \theta-s\sin \theta'|\gtrsim  r+s.
\end{align}
Then using the formula,
\begin{align}\label{phase-def}
e^{2\pi i  \rho (r\sin \theta-s\sin \theta')}=\frac{1}{2\pi i (r\sin \theta-s\sin \theta')}\partial_\rho\Big(e^{2\pi i  \rho (r\sin \theta-s\sin \theta')}\Big),
\end{align}
and  integration-by-parts 5 times, \eqref{0.44} turns to
\begin{align}
r^2& \int_0^{\frac\pi2}\!\!\!\int_{-\frac\pi2}^{\frac\pi2}\!\!\int_0^{+\infty}\!\!\! \int_0^{+\infty}\!\!\! \frac{e^{2\pi i  \rho (r\sin \theta-s\sin \theta')}}{\left[2\pi i (r\sin \theta-s\sin \theta')\right]^{5}}\partial_\rho^{5}\Big[\chi_{\le 1}(\rho r)\chi_{\ge 2^{k-1}}(\rho) \rho^{d-1}\Big]
\notag\\
&\quad\cdot   s^{\beta+d-1}P_{2^k}f(s)\,d\rho ds \>\chi_{\ge \frac\pi6}(\theta')\> \cos^{d-2} \theta\cos^{d-2} \theta'  d\theta' d\theta.
\label{7.12-11.33}
\end{align}
Note that $r\chi_{\le 1}'(\rho r) \lesssim \rho^{-1}$, we have that 
$$
\Big|\partial_\rho^{5}\Big[\chi_{\le 1}(\rho r)\chi_{\ge 2^{k-1}}(\rho) \rho^{d-1}\Big]\Big|
\lesssim
\chi_{\le 1}(\rho r)\chi_{\ge 2^{k-1}}(\rho) \rho^{d-6}.
$$
Here and in the following, for simplicity, we regard $\chi_{\ge 1}$ and its derivatives as the same.
Then this last estimate combining with \eqref{7.12-11.32} and \eqref{7.12-11.33},  gives the bound of \eqref{0.44} as
\begin{align*}
& \int_0^{+\infty}\!\!\! \int_0^{+\infty}\!\!\! \chi_{\le 1}(\rho r)\chi_{\ge 2^{k-1}}(\rho) \rho^{d-8}s^{\beta+d-6}|P_{2^k}f(s)|\,d\rho ds.
\end{align*}
Therefore,
$$
\|\eqref{0.44}\|_{L^2_r(\{r\le \frac14\})}\lesssim 2^{(d-\frac{15}{2})k}\|s^{\frac{d-1}{2}}P_{2^k}f\|_{L^2_s}
\lesssim
2^{(d-7)k}\big\|P_{2^k}f\big\|_{L^2(\R^d)}.
$$
Note that $|J(\rho r)|\lesssim 1$, thus \eqref{0.45} is also bounded by
\begin{align*}
& \int_0^{+\infty}\!\!\! \int_0^{+\infty}\!\!\! \chi_{\le 1}(\rho r)\chi_{\ge 2^{k-1}}(\rho) \langle\rho s\rangle^{-10}\rho^{d-3}s^{\beta+d-1}|P_{2^k}f(s)|\,d\rho ds.
\end{align*}
Hence, we get
$$
\|\eqref{0.45}\|_{L^2_r(\{r\le \frac14\})}\lesssim
2^{(d-12)k}\big\|P_{2^k}f\big\|_{L^2(\R^d)}.
$$
Combining these two estimates on \eqref{0.44-45}, we obtain that
\begin{align}
\|\eqref{0.21}\|_{L^2_r(\{r\le \frac14\})}\lesssim
2^{-2k}\big\|P_{2^k}f\big\|_{L^2(\R^d)}.\label{est:0.21}
\end{align}

Now we consider \eqref{0.22}, which is equal to
\begin{align*}
r^2&\int_0^{+\infty}  \big(J(\rho r)-K(\rho r)\big)\chi_{\ge 1}(\rho r)\chi_{\ge 2^{k-1}}(\rho) \rho^{d-1}\\
&\quad\cdot\int_0^{+\infty}\big(J(\rho s)+J(-\rho s)\big)s^{\beta+d-1}P_{2^k}f(s)\,ds d\rho .
\end{align*}
Using \eqref{1.03} again, we split \eqref{0.22} into the following two subparts again,
\begin{subequations}\label{0.46-47}
\renewcommand{\theequation}
{\theparentequation-\arabic{equation}}
\begin{align}
r^2\int_0^{+\infty}\!\!\!& \big(J(\rho r)-K(\rho r)\big)\chi_{\ge 1}(\rho r)\chi_{\ge 2^{k-1}}(\rho) \rho^{d-1}\notag\\
&\cdot\int_0^{+\infty}\!\!\!\int_{-\frac\pi2}^{\frac\pi2}e^{-2\pi i \rho s\sin \theta'}\chi_{\ge \frac\pi6}(\theta')\cos^{d-2} \theta'  \,d\theta' s^{\beta+d-1}P_{2^k}f(s)\,ds d\rho,\label{0.46}
\end{align}
and
\begin{align}
r^2\int_0^{+\infty}\!\!\!& \big(J(\rho r)-K(\rho r)\big)\chi_{\ge 1}(\rho r)\chi_{\ge 2^{k-1}}(\rho) \rho^{d-1}
\int_0^{+\infty}\!\!\!O\big( \langle\rho s\rangle^{-10}\big) s^{\beta+d-1}P_{2^k}f(s)\,ds d\rho.\label{0.47}
\end{align}
\end{subequations}
For \eqref{0.46},  using the formula in \eqref{20.14III}, we can  write
\begin{align}
\chi_{\ge 1}(\rho r)\Big(J(\rho r)&-K(\rho r)\Big)=\chi_{\ge 1}(\rho r)\int_0^{\frac\pi2} e^{2\pi i  \rho r\sin \theta}\eta(\theta,\rho r) \,d\theta, \label{20.14IV}
\end{align}
where the function $\eta(\theta,\rho r)$ is defined as
\begin{align*}
\eta(\theta,\rho r)=\chi_{\ge \frac\pi6}(\theta)\cos^{d-2} \theta+ \chi_{ 1\le \cdot \le 2}(\rho r) \chi_{\le \frac\pi6}(\theta) \cos^{d-2}\theta +\frac{\chi_{\ge 2}(\rho r)}{(2\pi i\rho r)^5}\tilde\eta_d(\theta).
\end{align*}
Then from \eqref{20.14IV},  the term \eqref{0.46} can be rewritten as
\begin{align*}
r^2\int_0^{\frac\pi2}\int_{-\frac\pi2}^{\frac\pi2}
\int_0^{+\infty}\!\!\!\int_0^{+\infty}\!\!\!& e^{2\pi i \rho (r\sin \theta-s\sin \theta')}\chi_{\ge 1}(\rho r)\chi_{\ge 2^{k-1}}(\rho)\eta(\theta,\rho r) \notag\\
&\cdot\rho^{d-1}s^{\beta+d-1}P_{2^k}f(s)\,ds d\rho\>\chi_{\ge \frac\pi6}(\theta')  \cos^{d-2} \theta' \,d\theta' d\theta.
\end{align*}
Now using \eqref{7.12-11.32} and treating similarly as \eqref{0.44}, we obtain that
$$
\|\eqref{0.46}\|_{L^2_r(\{r\le \frac14\})}\lesssim
2^{(d-12)k}\big\|P_{2^k}f\big\|_{L^2(\R^d)}.
$$
For \eqref{0.47}, applying \eqref{1.02-2}, we have that
\begin{align}
\chi_{\ge 1}(\rho r)\Big(J(\rho r)&-K(\rho r)\Big)=O\big((\rho r )^{-\frac{d-1}2}\big).
\end{align}
Hence,
\begin{align*}
|\eqref{0.47}|\lesssim  r^{2-\frac{d-1}2}\int_0^{+\infty}\!\!\! \int_0^{+\infty}\!\!\!  \chi_{\ge 2^{k-1}}(\rho) \langle\rho s\rangle^{-10}\rho^{\frac{d-1}2}s^{\beta+d-1}|P_{2^k}f(s)|\,d\rho ds.
\end{align*}
This gives that 
$$
\|\eqref{0.47}\|_{L^2_r(\{r\le \frac14\})}\lesssim
2^{\frac{d-19}{2}k}\big\|P_{2^k}f\big\|_{L^2(\R^d)}.
$$
Therefore, we obtain that
$$
\|\eqref{0.22}\|_{L^2_r(\{r\le \frac14\})}\lesssim
2^{\frac{d-19}{2}k}\big\|P_{2^k}f\big\|_{L^2(\R^d)}.
$$
Combining the estimates above, we obtain that
\begin{align}
\big\|\chi_{\le \frac 14} \big(P_{2^k}f\big)_{out,\ge k-1}\big\|_{L^2(\R^d)}
\lesssim
2^{-2k}\big\|P_{2^k}f\big\|_{L^2(\R^d)}. \label{est:supportf+}
\end{align}
Now similar argument (as in the proof of Proposition \ref{prop:f++-highfreq}) can be used to treat the functions in the high-order derivatives. Thus we obtain the desired estimate in  $H^{\mu(d)}(\R^d)$ space, and prove the lemma.
\end{proof}


\subsection{Boundedness of the incoming/outgoing projection operators}

The main results in this section are the boundedness of incoming/outgoing projection operator in $L^2(\R^d)$.
\begin{prop}
\label{prop:bound-fpm-L2} Suppose that $f\in L^2(\R^d)$, then for any $k\in \Z^+$,
\begin{align*}
\|f_{out/in,k}\|_{L^2(\R^d)}&\lesssim \|f\|_{L^2(\R^d)}.
\end{align*}
Here the implicit constant is independent of $k$.
\end{prop}

\begin{remark}
We remark here that we have a slightly stronger estimate read as
\begin{align*}
\|f_{out/in}\|_{L^2(\R^d)}&\lesssim \|f\|_{L^2(\R^d)}.
\end{align*}
The proof follows from the similar argument as in the proof of Proposition \ref{prop:bound-fpm-L2} below, with a further spatial cut-off on the function $f$.
However, the $\log$-loss is not essential in this paper, so to simplify the proof we only present a slightly weaker estimate  as in Proposition \ref{prop:bound-fpm-L2}.
\end{remark}

\begin{proof}[Proof of Proposition \ref{prop:bound-fpm-L2}]
We only consider the estimates on $f_{out,k}$. Since $f_{out,k}$ is radial, we have
\begin{align*}
\|f_{out,k}\|_{L^2(\R^d)}=&c\big\|r^{\frac{d-1}{2}}f_{out,k}(r)\big\|_{L^2_r}.
\end{align*}
By Definition \ref{def:outgong-incoming}, we write
\begin{subequations}\label{small-large}
\begin{align}
r^{\frac{d-1}{2}}f_{out,k}(r)
=&r^2\int_0^{+\infty} \big(J(\rho r)-K(\rho r)\big) \chi_{2^k}(\rho) \rho^{d-1}\mathcal Ff(\rho)\,d\rho\notag\\
=& r^2\int_0^{+\infty} J(\rho r)\chi_{\le 1}(\rho r) \chi_{2^k}(\rho) \rho^{d-1}\mathcal Ff(\rho)\,d\rho\label{Smaller1}\\
&+
r^2\int_0^{+\infty} \big(J(\rho r)-K(\rho r)\big)\chi_{\ge 1}(\rho r) \chi_{2^k}(\rho) \rho^{d-1}\mathcal Ff(\rho)\,d\rho.\label{Larger1}
\end{align}
\end{subequations}

\noindent{\bf Estimates on \eqref{Smaller1}.}  Note that 
\begin{align*}
\eqref{Smaller1}=r^2\int_0^{+\infty} J(\rho r)\chi_{\le 1}(\rho r) \chi_{2^k}(\rho) \rho^{d-1}\int_0^{+\infty}\Big(J(\rho s)+J(-\rho s)\Big)s^{\beta+d-1}f(s)\,ds d\rho .
\end{align*}
Accordingly, we further split it into the following two subparts again,
\begin{subequations}
\renewcommand{\theequation}
{\theparentequation-\arabic{equation}}
\begin{align}
r^2\int_0^{+\infty}\!\! J(\rho r)\chi_{\le 1}(\rho r) \chi_{2^k}(\rho) \rho^{d-1}\int_0^{+\infty}\!\!\Big(J(\rho s)+J(-\rho s)\Big)\chi_{\le 10}(\rho s)s^{\beta+d-1}f(s)\,ds d\rho;\label{S1}
\end{align}
and
\begin{align}
r^2\int_0^{+\infty}\!\! J(\rho r)\chi_{\le 1}(\rho r) \chi_{2^k}(\rho) \rho^{d-1}\int_0^{+\infty}\!\!\Big(J(\rho s)+J(-\rho s)\Big)\chi_{\ge 10}(\rho s)s^{\beta+d-1}f(s)\,ds d\rho .\label{S2}
\end{align}
\end{subequations}
For \eqref{S1}, noting that $|J(\rho r)|, |J(\rho s)|\lesssim 1$, we have that
$$
|\eqref{S1}|
\lesssim 
r^2\int_0^{+\infty}\!\!\int_0^{+\infty}\!\! \chi_{\le 1}(\rho r) \chi_{2^k}(\rho) \rho^{d-1}\chi_{\le 10}(\rho s)s^{\beta+d-1}\big|f(s)\big|\,ds d\rho. 
$$
Thus, 
\begin{align}
\|\eqref{S1}\|_{L^2_r}
&\lesssim
\int_0^{+\infty}\!\!\int_0^{+\infty}\!\! \big\|r^2 \chi_{\le 1}(\rho r)\big\|_{L^2_r} \chi_{2^k}(\rho)\rho^{d-1}\chi_{\le 10}(\rho s)s^{\beta+d-1}|f(s)|\,ds d\rho\notag\\
&\lesssim
\int_0^{+\infty}\!\! \chi_{2^k}(\rho) \rho^{d-\frac72}\big\|\chi_{\le 10}(\rho s)s^{d-3}\big\|_{L^2_s}\big\|s^{\frac{d-1}{2}}f(s)\big\|_{L^2_s} d\rho\notag\\
&\lesssim
\int_0^{+\infty}\!\!\rho^{-1} \chi_{2^k}(\rho) \,d\rho \>\big\|s^{\frac{d-1}{2}}f(s)\big\|_{L^2_s}
\lesssim \|f\|_{L^2(\R^d)}.\label{est:S1}
\end{align}
Now  we estimate \eqref{S2}.
From \eqref{1.03}, we write 
\begin{subequations}
\renewcommand{\theequation}
{\theparentequation-\roman{equation}}
\begin{align}
 \eqref{S2}=r^2\int_0^{+\infty} &\!\! J(\rho r)\chi_{\le 1}(\rho r) \chi_{2^k}(\rho) \rho^{d-1}\int_0^{+\infty}\!\!\int_{-\frac\pi2}^{\frac\pi2}e^{-2\pi i  \rho s\sin \theta}\notag\\
&\cdot\chi_{\ge \frac\pi6}(\theta)\cos^{d-2} \theta  \,d\theta\chi_{\ge 10}(\rho s)s^{\beta+d-1}f(s)\,ds d\rho\label{S21}\\
+r^2\int_0^{+\infty}\!\! & J(\rho r)\chi_{\le 1}(\rho r) \chi_{2^k}(\rho) \rho^{d-1}\int_0^{+\infty}\!\!O\big((\rho s)^{-10}\big)\chi_{\ge 10}(\rho s)s^{\beta+d-1}f(s)\,ds d\rho.\label{S23}
\end{align}
\end{subequations}
The terms \eqref{S21} is equal to
\begin{align}
r^2\int_0^\frac\pi2\int_{-\frac\pi2}^{\frac\pi2}&\int_0^{+\infty}\!\!\int_0^{+\infty}\!\!e^{2\pi i\rho(-s\sin \theta'+r\sin \theta)}\chi_{\le 1}(\rho r) \chi_{2^k}(\rho)\chi_{\ge 10}(\rho s)\notag\\
&\cdot\rho^{d-1}s^{\beta+d-1}f(s) \,d\rho ds\> \chi_{\ge \frac\pi6}(\theta')\cos^{d-2} \theta' \cos^{d-2}\theta \,d\theta' d\theta.
\end{align}
Note that the power $\beta+d-1$ is too large to be integrable in $s$, 
but we can use integration-by-parts to decrease the power of $s$, due to the non-resonance of the phase.
Indeed, when $|\theta'|\ge \frac\pi7$, $s\ge 5r$, we have \eqref{7.12-11.32}. Thus using the formula \eqref{phase-def}
and integration-by-parts 10 times,  \eqref{S21} can be controlled by
$$
r^2\int_0^{+\infty}\!\!\int_0^{+\infty}\!\!\chi_{\lesssim 1}(\rho r) \chi_{2^k}(\rho)\chi_{\gtrsim 1}(\rho s)\rho^{d-11}s^{\beta+d-11}|f(s)| \,d\rho ds.
$$
Therefore, we obtain
$$
\|\eqref{S21}\|_{L^2_r}\lesssim \|f\|_{L^2(\R^d)}.
$$
The term \eqref{S23} can be controlled by
\begin{align*}
\int_0^{+\infty}\!\! \int_0^{+\infty}\!\!\chi_{\le 1}(\rho r)\chi_{\ge 10}(\rho s) \chi_{2^k}(\rho) \rho^{d-13}s^{d-13}\cdot s^{\frac{d-1}{2}}|f(s)|\,ds d\rho.
\end{align*}
Hence, similar as the estimate on \eqref{S1}, we have
$$
\|\eqref{S23}\|_{L^2_r}\lesssim \|f\|_{L^2(\R^d)}.
$$
Combining this last estimate  with the estimate on \eqref{S21}, we get
$$
\|\eqref{S2}\|_{L^2_r}\lesssim \|f\|_{L^2(\R^d)}.
$$
This together with \eqref{est:S1}, gives
\begin{align}
\|\eqref{Smaller1}\|_{L^2_r}\lesssim \|f\|_{L^2(\R^d)}.\label{est:Small1}
\end{align}

\noindent {\bf Estimates on  \eqref{Larger1}.}  Note that 
\begin{align*}
 \eqref{Larger1}=r^2\int_0^{+\infty} &  \big(J(\rho r)-K(\rho r)\big)\chi_{\ge 1}(\rho r) \chi_{2^k}(\rho) \\
& \cdot\rho^{d-1}\int_0^{+\infty}\Big(J(\rho s)+J(-\rho s)\Big)s^{\beta+d-1}f(s)\,ds d\rho .
\end{align*}
Note that for $J(\rho r)-K(\rho r)$, we have the equality given in \eqref{1.02}; and for $J(\rho s)+J(-\rho s)$ we have the equality given in  \eqref{1.03}.
Then we  divide \eqref{Larger1} into the following four parts. The first part is
\begin{subequations}
\renewcommand{\theequation}
{\theparentequation-\arabic{equation}}
\begin{align}
r^2&\int_0^{+\infty}\!\!\!  \int_0^{\frac\pi2} e^{2\pi i  \rho r\sin \theta}\chi_{\ge \frac\pi6}(\theta)\cos^{d-2} \theta  \,d\theta\>
\chi_{\ge 1}(\rho r) \chi_{2^k}(\rho) \rho^{d-1}\notag\\
&\cdot\int_0^{+\infty}\!\!\! \int_{-\frac\pi2}^{\frac\pi2}e^{-2\pi i  \rho s\sin \theta'}\chi_{\ge \frac\pi6}(\theta')\cos^{d-2} \theta'  \,d\theta' s^{\beta+d-1}f(s)\,ds d\rho;\label{22.08}
\end{align}
the second part is
\begin{align}
r^2 &\int_0^{+\infty}\!\!\!  \int_0^{\frac\pi2} e^{2\pi i  \rho r\sin \theta}\chi_{\ge \frac\pi6}(\theta)\cos^{d-2} \theta  \,d\theta\>\chi_{\ge 1}(\rho r) \chi_{2^k}(\rho) \rho^{d-1}\notag\\
&\cdot\int_0^{+\infty}O\big( \langle\rho s\rangle^{-10}\big) s^{\beta+d-1}f(s)\,ds d\rho;\label{22.09}
\end{align}
the third part is
\begin{align}
 r^2 &\int_0^{+\infty} O\big( (\rho r)^{-5}\big)\chi_{\ge 1}(\rho r) \chi_{2^k}(\rho) \rho^{d-1}\notag\\
&\cdot\int_0^{+\infty}\!\!\! \int_{-\frac\pi2}^{\frac\pi2}e^{-2\pi i  \rho s\sin \theta'}\chi_{\ge \frac\pi6}(\theta')\cos^{d-2} \theta'  \,d\theta's^{\beta+d-1}f(s)\,ds d\rho;\label{22.10}
\end{align}
and the fourth part is
\begin{align}
 r^2 &\int_0^{+\infty}O\big( (\rho r)^{-5}\big)\chi_{\ge 1}(\rho r) \chi_{2^k}(\rho) \rho^{d-1}\int_0^{+\infty}O\big(\langle\rho s\rangle^{-10}\big)s^{\beta+d-1}f(s)\,ds d\rho.\label{22.11}
\end{align}
\end{subequations}

\noindent Estimate on \eqref{22.08}.  We split  \eqref{22.08} into the following two subparts again,
\begin{subequations}
\renewcommand{\theequation}
{\theparentequation-\roman{equation}}
\begin{align}
r^2&\int_0^{+\infty}\!\!\!  \int_0^{\frac\pi2} e^{2\pi i  \rho r\sin \theta}\chi_{\ge \frac\pi6}(\theta)\cos^{d-2} \theta  \,d\theta\>
\chi_{\ge 1}(\rho r) \chi_{2^k}(\rho) \rho^{d-1}\notag\\
&\cdot\int_0^{+\infty}\!\!\! \int_{-\frac\pi2}^{\frac\pi2}e^{-2\pi i  \rho s\sin \theta'}\chi_{\ge \frac\pi6}(\theta')\cos^{d-2} \theta'  \,d\theta' \big(\chi_{\le \frac{1}{10}r}(s)+\chi_{\ge 10r}(s)\big) s^{\beta+d-1}f(s)\,ds d\rho; \label{22.08I}
\end{align}
and
\begin{align}
r^2&\int_0^{+\infty}\!\!\!  \int_0^{\frac\pi2} e^{2\pi i  \rho r\sin \theta}\chi_{\ge \frac\pi6}(\theta)\cos^{d-2} \theta  \,d\theta\>
\chi_{\ge 1}(\rho r) \chi_{2^k}(\rho) \rho^{d-1}\notag\\
&\cdot\int_0^{+\infty}\!\!\! \int_{-\frac\pi2}^{\frac\pi2}e^{-2\pi i  \rho s\sin \theta'}\chi_{\ge \frac\pi6}(\theta')\cos^{d-2} \theta'  \,d\theta'\chi_{\frac1{10}r\le \cdot\le 10r}(s)s^{\beta+d-1}f(s)\,ds d\rho.\label{22.08II}
\end{align}
\end{subequations}
For \eqref{22.08I}, we rewrite it as
\begin{align*}
r^2& \int_0^{\frac\pi2}\!\!\!\int_{-\frac\pi2}^{\frac\pi2}\!\!\int_0^{+\infty}\!\!\! \int_0^{+\infty}\!\!\! e^{2\pi i  \rho (r\sin \theta-s\sin \theta')}\chi_{\ge 1}(\rho r) \chi_{2^k}(\rho) \rho^{d-1}\>\chi_{\ge \frac\pi6}(\theta)
\notag\\
&\cdot\cos^{d-2} \theta\chi_{\ge \frac\pi6}(\theta')\cos^{d-2} \theta'  \big(\chi_{\le \frac{1}{10}r}(s)+\chi_{\ge 10r}(s)\big) s^{\beta+d-1}f(s)\, d\rho ds d\theta' d\theta.
\end{align*}
Note that we have \eqref{7.12-11.32} in this situation. Thus treated similarly as the term \eqref{S21},  using the formula \eqref{phase-def}
and integration-by-parts $10$ times, 
\eqref{22.08I} has the bound of
\begin{align*}
r^2\chi_{\gtrsim 2^{-k}}(r)\int_0^{+\infty}\!\!\! \int_0^{+\infty}\!\!\!  (r+s)^{-10}\rho^{-10}\chi_{2^k}(\rho) \rho^{d-1} s^{\beta+d-1}\big|f(s)\big|\,d\rho ds.
\end{align*}
Hence,
\begin{align}
\|\eqref{22.08I}\|_{L^2_r}\lesssim   \|f\|_{L^2(\R^d)}. \label{est: 22.08I}
\end{align}
For \eqref{22.08II}, by dyadic decomposition, it is equal to
\begin{align*}
\sum\limits_{j=-k}^\infty &\chi_{2^j}(r)\cdot\eqref{22.08II}.
\end{align*}
Note that $\chi_{2^j}(r)\cdot\eqref{22.08II}$ can be written by
\begin{align}
r^2&\chi_{2^j}(r)\int_0^{+\infty}\!\!\!  \int_0^{\frac\pi2} e^{2\pi i  \rho r\sin \theta}\chi_{\ge \frac\pi6}(\theta)\cos^{d-2} \theta  \,d\theta\>
\chi_{\ge 1}(\rho r) \chi_{2^k}(\rho) \rho^{d-1}\notag\\
&\cdot\int_0^{+\infty}\!\!\! \int_{-\frac\pi2}^{\frac\pi2}\!\!\!e^{-2\pi i  \rho s\sin \theta'}\chi_{\ge \frac\pi6}(\theta')\cos^{d-2} \theta'  \,d\theta'\chi_{2^{j-2}\le \cdot\le 2^{j+2}}(s)s^{\beta+d-1}f(s)\,ds d\rho\label{23.02}\\
&\quad +\mbox{other terms}.\notag
\end{align}
Here the ``\textit{other terms}'' can be treated in the same manner as \eqref{22.08I} and thus we ignore them.
For \eqref{23.02}, from Lemma \ref{lem:Jr}, we have the following two formulas,
\begin{align*}
\int_0^{\frac\pi2} e^{2\pi i  \rho r\sin \theta}\chi_{\ge \frac\pi6}(\theta)\cos^{d-2} \theta  \,d\theta
=
c (\rho r)^{-\frac{d-1}{2}} e^{2\pi i \rho r} +O\big((\rho r)^{-\frac{d+1}{2}}\big);
\end{align*}
and
\begin{align*}
\int_{-\frac\pi2}^{\frac\pi2}\!\!\!e^{-2\pi i  \rho s\sin \theta'}\chi_{\ge \frac\pi6}(\theta')\cos^{d-2} \theta'  \,d\theta'
&=
(\rho s)^{-\frac{d-1}{2}}\Big(\bar c e^{-2\pi i \rho s}+c  e^{2\pi i\rho s}\Big) +O\big((\rho s)^{-\frac{d+1}{2}}\big).
\end{align*}
Accordingly, we divide \eqref{23.02} into four parts as \eqref{22.08}--\eqref{22.11}. The first part is
\begin{subequations}
\renewcommand{\theequation}
{\theparentequation-ii\alph{equation}}
\begin{align}
r^2&\chi_{2^j}(r)\int_0^{+\infty}\!\!\!  (\rho r)^{-\frac{d-1}{2}} e^{2\pi i \rho r}
\chi_{\ge 1}(\rho r) \chi_{2^k}(\rho) \rho^{d-1}\notag\\
&\cdot\int_0^{+\infty}\!\!\! (\rho s)^{-\frac{d-1}{2}}\Big(\bar c e^{-2\pi i \rho s}+c  e^{2\pi i\rho s}\Big)\chi_{2^{j-2}\le \cdot\le 2^{j+2}}(s)s^{\beta+d-1}f(s)\,ds d\rho;\label{23.02I}
\end{align}
the second part is
\begin{align}
r^2&\chi_{2^j}(r)\int_0^{+\infty}\!\!\!  (\rho r)^{-\frac{d-1}{2}} e^{2\pi i \rho r}
\chi_{\ge 1}(\rho r) \chi_{2^k}(\rho) \rho^{d-1}\notag\\
&\cdot\int_0^{+\infty}\!\!\! O\big((\rho s)^{-\frac{d+1}{2}}\big)\chi_{2^{j-2}\le \cdot\le 2^{j+2}}(s)s^{\beta+d-1}f(s)\,ds d\rho;\label{23.02II}
\end{align}
the third part is
\begin{align}
 r^2&\chi_{2^j}(r)\int_0^{+\infty}\!\!\!  O\big((\rho r)^{-\frac{d+1}{2}}\big)
\chi_{\ge 1}(\rho r) \chi_{2^k}(\rho) \rho^{d-1}\notag\\
&\cdot\int_0^{+\infty}\!\!\! (\rho s)^{-\frac{d-1}{2}}\Big(\bar c e^{-2\pi i \rho s}+c  e^{2\pi i\rho s}\Big)\chi_{2^{j-2}\le \cdot\le 2^{j+2}}(s)s^{\beta+d-1}f(s)\,ds d\rho;\label{23.02III}
\end{align}
and the fourth part is
\begin{align}
 r^2&\chi_{2^j}(r)\int_0^{+\infty}\!\!\!  O\big((\rho r)^{-\frac{d+1}{2}}\big)
\chi_{\ge 1}(\rho r) \chi_{2^k}(\rho) \rho^{d-1}\notag\\
&\cdot\int_0^{+\infty}\!\!\! O\big((\rho s)^{-\frac{d+1}{2}}\big)\chi_{2^{j-2}\le \cdot\le 2^{j+2}}(s)s^{\beta+d-1}f(s)\,ds d\rho.\label{23.02IV}
\end{align}
\end{subequations}
First, we consider \eqref{23.02I}, which is equivalent to
\begin{align*}
2^{(2-\frac{d-1}{2})j}&\chi_{2^j}(r)\int_0^{+\infty}\!\!\!  e^{2\pi i \rho r}
\chi_{\ge 1}(\rho r) \chi_{2^k}(\rho) \notag\\
&\cdot\int_0^{+\infty}\!\!\! \Big(\bar c e^{-2\pi i \rho s}+c  e^{2\pi i\rho s}\Big)\chi_{2^{j-2}\le \cdot\le 2^{j+2}}(s)s^{\frac{d-5}{2}}s^{\frac{d-1}{2}}f(s)\,ds d\rho.
\end{align*}
Then using Corollary \ref{cor:L2-estimate-FIO} and Plancherel's identity, we have
\begin{align}
\|\eqref{23.02I}\|_{L^2_r}
\lesssim &
2^{\frac{5-d}{2}j}\Big\|\int_0^{+\infty}\!\!\! \Big(\bar c e^{-2\pi i \rho s}+c  e^{2\pi i\rho s}\Big)\chi_{2^{j-2}\le \cdot\le 2^{j+2}}(s)s^{\frac{d-5}{2}}s^{\frac{d-1}{2}}f(s)\,ds\Big\|_{L^2_\rho}\notag\\
\lesssim &
2^{\frac{5-d}{2}j}\big\|\chi_{2^{j-2}\le \cdot\le 2^{j+2}}(s)s^{\frac{d-5}{2}}s^{\frac{d-1}{2}}f(s) \big\|_{L^2_s}
\lesssim \big\|f\big\|_{L^2(\{|x|\sim 2^{j}\})}.\label{est:23.02I}
\end{align}
Using the relationship $r\sim s$, the term \eqref{23.02II}--\eqref{23.02IV} can be treated in the similar way as \eqref{22.09}--\eqref{22.11}. We just take the term
\eqref{23.02II} for example. Indeed, 
by  Corollary \ref{cor:L2-estimate-FIO}, we have
\begin{align}
\|\eqref{23.02II}\|_{L^2_r}
\lesssim &
2^{(2-\frac{d-1}{2})j}\Big\|\rho ^{-\frac{d-1}{2}}
 \chi_{2^k}(\rho) \rho^{d-1}\int_0^{+\infty}\!\!\! O\big((\rho s)^{-\frac{d+1}{2}}\big)\chi_{2^{j-2}\le \cdot\le 2^{j+2}}(s)s^{\beta+d-1}f(s)\,ds \Big\|_{L^2_\rho}\notag\\
\lesssim &
\big\|\chi_{2^{j-2}\le \cdot\le 2^{j+2}}(s)s^{\frac{d-1}{2}}f(s) \big\|_{L^2_s}
\lesssim \big\|f\big\|_{L^2(\{|x|\sim 2^{j}\})}.\label{est:23.02II}
\end{align}
Similar as \eqref{22.10}, \eqref{22.11}, we have
\begin{align}
\|\eqref{23.02III}\|_{L^2_r}+\|\eqref{23.02IV}\|_{L^2_r}
\lesssim \big\|f\big\|_{L^2(\{|x|\sim 2^{j}\})}.\label{est:23.02III}
\end{align}
Combining with \eqref{est:23.02II}, \eqref{est:23.02III}  and \eqref{est:23.02I}, we obtain that
\begin{align*}
\|\eqref{23.02}\|_{L^2_r}
\lesssim
\big\|f\big\|_{L^2(\{|x|\sim 2^{j}\})}.
\end{align*}
Hence, we have
\begin{align*}
\|\chi_{2^j}(r)\cdot\eqref{22.08II}\|_{L^2_r}
\lesssim
\big\|f\big\|_{L^2(\{|x|\sim 2^{j}\})}.
\end{align*}
Therefore,
\begin{align*}
\|\eqref{22.08II}\|_{L^2_r}^2
\lesssim
\sum\limits_{j=-k}^{+\infty}\|\chi_{2^j}(r)\cdot\eqref{22.08II}\|_{L^2_r}^2
\lesssim
\sum\limits_{j=-k}^{+\infty}\big\|f\big\|_{L^2(\{|x|\sim 2^{j}\})}^2
\lesssim \big\|f\big\|_{L^2(\R^d)}^2.
\end{align*}
This last estimate together with \eqref{est: 22.08I}, yields that
\begin{align}\label{est:22.08}
\|\eqref{22.08}\|_{L^2_r}
\lesssim \big\|f\big\|_{L^2(\R^d)}.
\end{align}

Now we consider the second part \eqref{22.09}.  Using the formula,
$$
\partial_\rho e^{2\pi i  \rho r\sin \theta}=2\pi i  r\sin \theta\cdot  e^{2\pi i  \rho r\sin \theta},
$$
and integration-by-parts, we reduce \eqref{22.09} to
\begin{align*}
\int_0^{+\infty}\!\!\!  \int_0^{\frac\pi2} &e^{2\pi i  \rho r\sin \theta}\frac{\chi_{\ge \frac\pi6}(\theta)\cos^{d-2} \theta }{(2\pi i \sin\theta)^2} \,d\theta\>\partial_\rho^2\Big[\chi_{\ge 1}(\rho r) \chi_{2^k}(\rho) \rho^{d-1}\notag\\
&\cdot\int_0^{+\infty}O\big( \langle\rho s\rangle^{-10}\big) s^{\beta+d-1}f(s)\,ds\Big] d\rho.
\end{align*}
Then by Corollary \ref{cor:L2-estimate-FIO}, we have
\begin{align*}
\|\eqref{22.09}\|_{L^2_r}
\lesssim & \int_0^{\frac\pi2}\Big\|\partial_\rho^2\Big[ \chi_{2^k}(\rho) \rho^{d-1} \int_0^{+\infty}O\big( \langle\rho s\rangle^{-10}\big) s^{\beta+d-1}f(s)\,ds\Big]\Big\|_{L^2_\rho}
\chi_{\ge \frac\pi6}(\theta)\frac{\cos^{d-2} \theta}{\sin^\frac52 \theta}  \,d\theta\\
\lesssim & \Big\|\partial_\rho^2\Big[ \chi_{2^k}(\rho) \rho^{d-1} \int_0^{+\infty}O\big( \langle\rho s\rangle^{-10}\big) s^{\beta+d-1}f(s)\,ds\Big]\Big\|_{L^2_\rho}.
\end{align*}
Here we regard $\chi_1, \chi_1', \chi_1''$ as the same because they have the same properties we need. 
From the explicit formula in \eqref{12.29}, we note that 
$$
\partial_\rho^j\left[O\big( \langle\rho s\rangle^{-10}\big)\right]=O\big(s^j  \langle\rho s\rangle^{-10}\big), \quad 
\mbox{for } j=0,1,2.
$$
Hence,  we get that 
\begin{align*}
&\Big|\partial_\rho^2\Big[ \chi_{2^k}(\rho) \rho^{d-1} \int_0^{+\infty}O\big( \langle\rho s\rangle^{-10}\big) s^{\beta+d-1}f(s)\,ds\Big]\Big|\\
\lesssim &
\int_0^{+\infty}\chi_{2^k}(\rho)\rho^{d-1}\langle\rho s\rangle^{-10}s^2 s^{\beta+d-1}|f(s)|\,ds
+\int_0^{+\infty}\chi_{2^k}(\rho)\rho^{d-3}\langle\rho s\rangle^{-10} s^{\beta+d-1}|f(s)|\,ds\\
\lesssim &
\int_0^{+\infty}\chi_{2^k}(\rho)\>\left[(\rho s)^{d-1}+(\rho s)^{d-3}\right]\langle\rho s\rangle^{-10}s^{\frac{d-1}2}|f(s)|\,ds.
\end{align*}
Therefore, by H\"older's inequality, we have
\begin{align}
\|\eqref{22.09}\|_{L^2_r}\lesssim \|s^{\frac{d-1}2}f\|_{L^2_s}\lesssim \|f\|_{L^2(\R^d)}. \label{est: 22.09}
\end{align}

For the term \eqref{22.10}, we split it into the following two subsubparts again,
\begin{subequations}
\renewcommand{\theequation}
{\theparentequation-\roman{equation}}
\begin{align}
 r^2 &\int_0^{+\infty} O\big( (\rho r)^{-5}\big)\chi_{\ge 1}(\rho r) \chi_{2^k}(\rho) \rho^{d-1}\notag\\
&\cdot\int_0^{+\infty}\!\!\! \int_{-\frac\pi2}^{\frac\pi2}e^{-2\pi i  \rho s\sin \theta'}\chi_{\ge \frac\pi6}(\theta')\cos^{d-2} \theta'  \,d\theta' \chi_{\le 2^{-k}}(s)s^{\beta+d-1}f(s)\,ds d\rho; \label{22.10I}
\end{align}
and
\begin{align}
 r^2 &\int_0^{+\infty} O\big( (\rho r)^{-5}\big)\chi_{\ge 1}(\rho r) \chi_{2^k}(\rho) \rho^{d-1}\notag\\
&\cdot\int_0^{+\infty}\!\!\! \int_{-\frac\pi2}^{\frac\pi2}e^{-2\pi i  \rho s\sin \theta'}\chi_{\ge \frac\pi6}(\theta')\cos^{d-2} \theta'  \,d\theta' \chi_{\ge 2^{-k}}(s)s^{\beta+d-1}f(s)\,ds d\rho.\label{22.10II}
\end{align}
\end{subequations}
By H\"older's inequality,  we have 
\begin{align*}
  |\eqref{22.10I}|\lesssim  r^{-3}&\Big\|\chi_{\ge 1}(\rho r) \chi_{2^k}(\rho) \rho^{d-6}\Big\|_{L^2_\rho}\notag\\
&\cdot\int_{-\frac\pi2}^{\frac\pi2}\!\!\Big\|\int_0^{+\infty}\!\!\!e^{-2\pi i  \rho s\sin \theta'} \chi_{\le 2^{-k}}(s)s^{\beta+d-1}f(s)\,ds \Big\|_{L^2_\rho}\chi_{\ge \frac\pi6}(\theta')\cos^{d-2} \theta' \,d\theta',
\end{align*}
then by Plancherel's identity, it is further bounded by
\begin{align*}
 r^{-3}\chi_{\ge 2^{-k-1}}(r)2^{(d-\frac12-5)k} \Big\| \chi_{\le 2^{-k}}(s)s^{d-3}s^{\frac{d-1}2}f(s) \Big\|_{L^2_s}.
\end{align*}
Hence,  since $d\ge3$, we have
\begin{align}
\|\eqref{22.10I}\|_{L^2_r}\lesssim \|s^{\frac{d-1}2}f\|_{L^2_s}\lesssim \|f\|_{L^2(\R^d)}. \label{est: 22.10I}
\end{align}

Now we consider \eqref{22.10II}. From Lemma \ref{lem:Jr}, we have
$$
\int_{-\frac\pi2}^{\frac\pi2}e^{-2\pi i  \rho s\sin \theta'}\chi_{\ge \frac\pi6}(\theta')\cos^{d-2} \theta'  \,d\theta'
=(\rho s)^{-\frac{d-1}{2}}\Big(\bar c e^{-2\pi i\rho s}+c  e^{2\pi i\rho s}\Big) +O\big((\rho s)^{-\frac{d+1}{2}}\big).
$$
Accordingly, \eqref{22.10II} can be further divided into the following two sub$^3$-parts again,
\begin{subequations}
\renewcommand{\theequation}
{\theparentequation-ii\alph{equation}}
\begin{align}
 r^2 &\int_0^{+\infty} O\big( (\rho r)^{-5}\big)\chi_{\ge 1}(\rho r) \chi_{2^k}(\rho) \rho^{d-1}\notag\\
&\cdot\int_0^{+\infty}\!\!\!(\rho s)^{-\frac{d-1}{2}}\Big(\bar c e^{2\pi i\rho s}+c  e^{-2\pi i\rho s}\Big) \chi_{\ge 2^{-k}}(s)s^{\beta+d-1}f(s)\,ds d\rho;\label{22.10II-1}
\end{align}
and
\begin{align}
 r^2 &\int_0^{+\infty} O\big( (\rho r)^{-5}\big)\chi_{\ge 1}(\rho r) \chi_{2^k}(\rho) \rho^{d-1}\notag\\
&\cdot\int_0^{+\infty}\!\!\!O\big((\rho s)^{-\frac{d+1}{2}}\big) \chi_{\ge 2^{-k}}(s)s^{\beta+d-1}f(s)\,ds d\rho.\label{22.10II-2}
\end{align}
\end{subequations}
Thanks to the power of $(\rho s)^{-\frac{d-1}{2}}$ in the integral, we can deal with \eqref{22.10II-1}  as in \eqref{22.10I}, which implies that
\begin{align*}
|\eqref{22.10II-1}|\lesssim  r^{-3}\chi_{\ge 2^{-k-1}}(r)2^{(\frac d2-5)k} \Big\| \chi_{\ge 2^{-k}}(s)s^{\frac{d-5}2}s^{\frac{d-1}2}f(s) \Big\|_{L^2_s}.
\end{align*}
Since $d\le 5$, we have
\begin{align}
\|\eqref{22.10II-1}\|_{L^2_r}\lesssim   \|f\|_{L^2(\R^d)}. \label{est: 22.10II-1}
\end{align}
The term \eqref{22.10II-2} is controlled by
$$
r^{-3}\int_0^{+\infty}\!\!\!\int_0^{+\infty}\!\!\!  \chi_{\ge 1}(\rho r)\chi_{2^k}(\rho) \chi_{\ge 2^{-k}}(s)\rho^{\frac{d-13}{2}}s^{\frac{d-7}{2}}s^{\frac{d-1}2}|f(s)|\,ds d\rho.
$$
Hence,  by H\"older's inequality,  we obtain that
\begin{align}
\|\eqref{22.10II-2}\|_{L^2_r}\lesssim   \|f\|_{L^2(\R^d)}. \label{est: 22.10II-2}
\end{align}
Therefore, together with \eqref{est: 22.10II-1} and \eqref{est: 22.10II-2}, we get
\begin{align*}
\|\eqref{22.10II}\|_{L^2_r}\lesssim   \|f\|_{L^2(\R^d)}.
\end{align*}
Hence, combining this last estimate with \eqref{est: 22.10I}, we get
\begin{align}
\|\eqref{22.10}\|_{L^2_r}\lesssim   \|f\|_{L^2(\R^d)}. \label{est: 22.10}
\end{align}

For the term \eqref{22.11},
it can be controlled by
$$
r^{-3}\int_0^{+\infty}\!\!\!\int_0^{+\infty}\!\!\!  \chi_{\ge 1}(\rho r)\chi_{2^k}(\rho) \langle\rho s\rangle^{-10}\rho^{d-6}s^{\beta+d-1}f(s)\,ds d\rho.
$$
Similar as \eqref{22.10II-2}, we have that
\begin{align}
\|\eqref{22.11}\|_{L^2_r}\lesssim   \|f\|_{L^2(\R^d)}. \label{est: 22.11}
\end{align}
This together with \eqref{est: 22.09}, \eqref{est: 22.10} and \eqref{est: 22.11}, gives
\begin{align*}
\|\eqref{Larger1}\|_{L^2_r}
\lesssim \big\|f\big\|_{L^2(\R^d)}.
\end{align*}
Hence, this last estimate above combined with \eqref{est:Small1}, yields the desired estimate.
\end{proof}

Based on Proposition \ref{prop:bound-fpm-L2}, we have the following estimate which can be regarded as the extension of the Bernstein estimate.
\begin{cor}\label{cor:f+-Hs}
For any $s\in [0,4]$, and any integer $k\ge 0$, suppose that $f\in \dot H^s(\R^d)$, then
\begin{align*}
\Big\|\chi_{\ge \frac 14}\left(P_{2^k}\big(\chi_{\ge 1}f\big)\right)_{out/in,k-1\le \cdot \le k+1}\Big\|_{\dot H^s(\R^d)}&\lesssim 2^{sk}\big\|P_{2^k}\big(\chi_{\ge 1}f\big)\big\|_{L^2(\R^d)}.
\end{align*}
Here the implicit constant is independent of $k$.
\end{cor}
\begin{proof}
We only consider the case when $s=4$, since the other cases can be obtained by the interpolation.

First, we give some reductions.
Note that for the radial function $g$, we have the following formula, for any integer $s\ge0$,
$$
(-\Delta)^{\frac s2}g(x)=\sum\limits_{j=0}^s c_j r^{j-s}\partial_r^jg(r),
$$
where $c_j$ are the constants dependent on $d$ and $j$, and $r=|x|$.
Therefore,
\begin{align*}
\Big\|\chi_{\ge \frac 14}&\left(P_{2^k}\big(\chi_{\ge 1}f\big)\right)_{out/in,k-1\le \cdot \le k+1}\Big\|_{\dot H^4(\R^d)}\\
&\lesssim \sum\limits_{j=0}^4c_j\big\|\chi_{\gtrsim 1}(r)r^{\frac{d-1}{2}}\partial_r^j\left(P_{2^k}\big(\chi_{\ge 1}f\big)\right)_{out/in,k-1\le \cdot \le k+1}(r)\big\|_{L^2_r}.
\end{align*}
Hence, we only need to show that for any integer $j\le 4$,
\begin{align}
\big\|\chi_{\gtrsim 1}(r)r^{\frac{d-1}{2}}\partial_r^j\left(P_{2^k}\big(\chi_{\ge 1}f\big)\right)_{out/in,k-1\le \cdot \le k+1}(r)\big\|_{L^2_r}
&\lesssim  2^{jk}\big\|P_{2^k}\big(\chi_{\ge 1}f\big)\big\|_{L^2(\R^d)}.\label{23.06}
\end{align}
Furthermore, from Definition \ref{def:outgong-incoming},
\begin{align}
\partial_r^j\big(P_{2^k}&\big(\chi_{\ge 1}f\big)\big)_{out/in,k-1\le \cdot \le k+1}(r)= \sum\limits_{j_1=0}^j c_{j_1} r^{-\beta-j+j_1}\notag\\
&\cdot\int_0^{+\infty}\!\!\partial_r^{j_1}\Big(J(\rho r)-K(\rho r)\Big) \> \chi_{2^k}(\rho)\>\rho^{d-1} \mathcal F \big(P_{2^k}(\chi_{\ge 1}f)\big)(\rho)\, d\rho,
\label{23.17}
\end{align}
where we have simplified the notation and denote
$$
\chi_{2^{k}}(\rho)=\chi_{2^{k-1}\le\cdot \le 2^{k+1}}(\rho),
$$
and the  constants $c_{j}$ may change line to line in the following. Moreover, we note that
$$
\partial_r^{j_1}\big(J(\rho r)-K(\rho r)\big)
=\Big(\frac{\rho}{r}\Big)^{j_1}\partial_\rho^{j_1}\big(J(\rho r)-K(\rho r)\big).
$$
Hence, \eqref{23.17} reduces to
\begin{align*}
\sum\limits_{j_1=0}^j c_{j_1} r^{-\beta-j}\int_0^{+\infty}\!\!\partial_\rho^{j_1}\Big(J(\rho r)-K(\rho r)\Big) \> \chi_{2^k}(\rho)\>\rho^{d-1+j_1} \mathcal F \big(P_{2^k}(\chi_{\ge 1}f)\big)(\rho)\, d\rho.
\end{align*}
Further, by integration-by-parts, it turns to
\begin{align*}
\sum\limits_{j_1=0}^j c_{j_1} r^{-\beta-j}\int_0^{+\infty}\!\!\Big(J(\rho r)-K(\rho r)\Big) \> \partial_\rho^{j_1}\Big[\chi_{2^k}(\rho)\>\rho^{d-1+j_1} \mathcal F \big(P_{2^k}(\chi_{\ge 1}f)\big)(\rho)\Big]\, d\rho.
\end{align*}
Therefore, we obtain that
\begin{align*}
\partial_r^j\big(P_{2^k}&\big(\chi_{\ge 1}f\big)\big)_{out/in,k-1\le \cdot \le k+1}(r)\\
=&
\sum\limits_{j_1=0}^j c_{j_1} r^{-\beta-j}\!\!\int_0^{+\infty}\!\!\!\Big(J(\rho r)-K(\rho r)\Big) \> \partial_\rho^{j_1}\Big[\chi_{2^k}(\rho)\>\rho^{j_1}\rho^{d-1} \Big]\mathcal F \big(P_{2^k}(\chi_{\ge 1}f)\big)(\rho)\Big]\, d\rho\\
&\quad+\sum\limits_{j_1=0}^j c_{j_1} r^{-\beta-j}\!\!\int_0^{+\infty}\!\!\!\Big(J(\rho r)-K(\rho r)\Big) \>\chi_{2^k}(\rho)\rho^{j_1}\rho^{d-1}\partial_\rho^{j_1}\Big[ \mathcal F \big(P_{2^k}\big(\chi_{\ge 1}f\big)\big)(\rho)\Big]\, d\rho.
\end{align*}
Therefore, to prove \eqref{23.06}, we only need to show the following two estimates: For any $j_1\le j$, 
\begin{subequations}
\begin{align}
\Big\|\chi_{\gtrsim 1}(r)r^{\frac{d-1}{2}}r^{-\beta-j}&\!\!\int_0^{+\infty}\!\!\!\Big(J(\rho r)-K(\rho r)\Big) \> \partial_\rho^{j_1}\Big[\chi_{2^k}(\rho)\>\rho^{j_1}\rho^{d-1}\Big] \notag\\
&\cdot \mathcal F \big(P_{2^k}\big(\chi_{\ge 1}f\big)\big)(\rho)\, d\rho\Big\|_{L^2_r}
\lesssim 2^{jk}\big\|P_{2^k}\big(\chi_{\ge 1}f\big)\big\|_{L^2(\R^d)};\label{0.34-I}
\end{align}
and
\begin{align}
\Big\|\chi_{\gtrsim 1}(r)r^{\frac{d-1}{2}}r^{-\beta-j}&\!\!\int_0^{+\infty}\!\!\!\Big(J(\rho r)-K(\rho r)\Big) \>\chi_{2^k}(\rho)\rho^{j_1}\rho^{d-1}\notag\\
&\cdot \partial_\rho^{j_1}\Big[ \mathcal F \big(P_{2^k}\big(\chi_{\ge 1}f\big)\big)(\rho)\Big]\, d\rho\Big\|_{L^2_r}
\lesssim
2^{jk}\big\|P_{2^k}\big(\chi_{\ge 1}f\big)\big\|_{L^2(\R^d)}.\label{0.34-II}
\end{align}
\end{subequations}
For \eqref{0.34-I}, note that $r\gtrsim 1, k\ge 0$, it thus reduces to show
\begin{align*}
\Big\|r^{-\beta}\!\!\int_0^{+\infty}\!\!\!\Big(J(\rho r)-K(\rho r)\Big) \> &\partial_\rho^{j_1}\Big[\chi_{2^k}(\rho)\>\rho^{j_1}\rho^{d-1}\Big] \mathcal F \big(P_{2^k}\big(\chi_{\ge 1}f\big)\big)(\rho)\, d\rho\Big\|_{L^2(\R^d)}\notag\\
&\lesssim \big\|P_{2^k}\big(\chi_{\ge 1}f\big)\big\|_{L^2(\R^d)}.
\end{align*}
Then we can prove it in the same way as Proposition \ref{prop:bound-fpm-L2}. Indeed,  we can regard
$$
\partial_\rho^{j_1}\Big[\chi_{2^k}(\rho)\>\rho^{j_1}\rho^{d-1}\Big] \quad \mbox{ as } \quad\chi_{2^k}(\rho)\rho^{d-1},
$$
because they have the same properties used in the proof of Proposition \ref{prop:bound-fpm-L2}.

For \eqref{0.34-II}, we first consider
$$
\partial_\rho^{j_1}\Big[ \mathcal F \big(P_{2^k}\big(\chi_{\ge 1}f\big)\big)(\rho)\Big],
$$
which is equal to
\begin{align*}
\int_0^{+\infty}\!\!\partial_\rho^{j_1}\Big(J(\rho s)+J(-\rho s)\Big)s^{\beta+d-1}P_{2^k}\big(\chi_{\ge 1}f\big)(s)\,ds.
\end{align*}
Using the relationship
$$
\partial_\rho^{j_1}\Big(J(\rho s)+J(-\rho s)\Big)=\Big(\frac{s}{\rho}\Big)^{j_1}\partial_s^{j_1}\Big(J(\rho s)+J(-\rho s)\Big),
$$
we have
\begin{align*}
\partial_\rho^{j_1}\Big[ \mathcal F \big(P_{2^k}\big(\chi_{\ge 1}f\big)\big)(\rho)\Big]
=\frac1{\rho^{j_1}}\int_0^{+\infty}\!\!\partial_s^{j_1}\Big(J(\rho s)+J(-\rho s)\Big)s^{\beta+d-1+j_1}P_{2^k}\big(\chi_{\ge 1}f\big)(s)\,ds.
\end{align*}
Then by integration-by-parts, to prove \eqref{0.34-II}, it suffices to show the following two estimates. The first one is
\begin{subequations}
\renewcommand{\theequation}
{\theparentequation-\arabic{equation}}
\begin{align}
\Big\|\chi_{\gtrsim 1}&(r)r^{\frac{d-1}{2}}r^{-\beta-j}\!\!\int_0^{+\infty}\!\!\!\Big(J(\rho r)-K(\rho r)\Big) \>\chi_{2^k}(\rho)\rho^{d-1}\notag\\
&\cdot \int_0^{+\infty}\!\!\Big(J(\rho s)+J(-\rho s)\Big)s^{\beta+d-1}P_{2^k}\big(\chi_{\ge 1}f\big)(s)\,ds d\rho\Big\|_{L^2_r}\lesssim
\big\|P_{2^k}\big(\chi_{\ge 1}f\big)\big\|_{L^2(\R^d)};\label{0.34-II-1}
\end{align}
and the second one is
\begin{align}
\Big\|\chi_{\gtrsim 1}(r)&r^{\frac{d-1}{2}}r^{-\beta-j}\!\!\int_0^{+\infty}\!\!\!\Big(J(\rho r)-K(\rho r)\Big) \>\chi_{2^k}(\rho)\rho^{d-1}\notag\\
&\cdot \int_0^{+\infty}\!\!\Big(J(\rho s)+J(-\rho s)\Big)s^{\beta+d-1+j_1}\partial_s^{j_1}\Big[P_{2^k}\big(\chi_{\ge 1}f\big)\Big]\,ds d\rho\Big\|_{L^2_r}\notag\\
&\qquad \qquad \lesssim
2^{jk}\big\|P_{2^k}\big(\chi_{\ge 1}f\big)\big\|_{L^2(\R^d)}.\label{0.34-II-2}
\end{align}
\end{subequations}
For \eqref{0.34-II-1}, we drop $\chi_{\gtrsim 1}(r)r^{-j}$ and reduce the left-hand side of  \eqref{0.34-II-1}  to 
\begin{align*}
\Big\|&r^{\frac{d-1}{2}-\beta}\!\!\int_0^{+\infty}\!\!\!\Big(J(\rho r)-K(\rho r)\Big) \>\chi_{2^k}(\rho)\rho^{d-1}\notag\\
&\cdot \int_0^{+\infty}\!\!\!\Big(J(\rho s)+J(-\rho s)\Big)s^{\beta+d-1}P_{2^k}\big(\chi_{\ge 1}f\big)(s)\,ds d\rho\Big\|_{L^2_r}.
\end{align*}
Note that by Definition \ref{def:outgong-incoming}, it  is equal to
$$
\big\|\left(P_{2^k}\big(\chi_{\ge 1}f\big)\right)_{out,k}\big\|_{L^2(\R^d)}.
$$
Thus the estimate  \eqref{0.34-II-1} follows from Proposition \ref{prop:bound-fpm-L2}.

For \eqref{0.34-II-2}, we split its left-hand side term into the following two subparts again:
\begin{subequations}
\renewcommand{\theequation}
{\theparentequation-\roman{equation}}
\begin{align}
\Big\|\chi_{\gtrsim 1}(r)&r^{\frac{d-1}{2}}r^{-\beta-j}\!\!\int_0^{+\infty}\!\!\!\Big(J(\rho r)-K(\rho r)\Big) \>\chi_{2^k}(\rho)\rho^{d-1}\notag\\
&\cdot \int_0^{+\infty}\!\!\Big(J(\rho s)+J(-\rho s)\Big)s^{\beta+d-1+j_1}\chi_{\le 1}(s)\partial_s^{j_1}\Big[P_{2^k}\big(\chi_{\ge 1}f\big)\Big]\,ds d\rho\Big\|_{L^2_r};
\label{20.54-I}
\end{align}
and
\begin{align}
\Big\|\chi_{\gtrsim 1}&(r)r^{\frac{d-1}{2}}r^{-\beta-j}\!\!\int_0^{+\infty}\!\!\!\Big(J(\rho r)-K(\rho r)\Big) \>\chi_{2^k}(\rho)\rho^{d-1}\notag\\
&\cdot
\int_0^{+\infty}\!\!\Big(J(\rho s)+J(-\rho s)\Big) s^{\beta+d-1+j_1}\chi_{\ge 1}(s)\partial_s^{j_1}\Big[P_{2^k}\big(\chi_{\ge 1}f\big)\Big]\,ds d\rho\Big\|_{L^2_r}.\label{20.54-II}
\end{align}
\end{subequations}

By Definition \ref{def:outgong-incoming} again, \eqref{20.54-I} is controlled by
$$
\Big\|\left(\chi_{\le 1}(r)r^{j_1}\partial_r^{j_1}\Big[P_{2^k}\big(\chi_{\ge 1}f\big)\Big]\right)_{out,k}\Big\|_{L^2(\R^d)}.
$$
Hence by Proposition \ref{prop:bound-fpm-L2}, it is further bounded by
\begin{align}
\Big\|\chi_{\le 1}(r)r^{j_1}\partial_r^{j_1}\Big[P_{2^k}\big(\chi_{\ge 1}f\big)\Big]\Big\|_{L^2(\R^d)}. \label{21.37}
\end{align}
Note that for any radial function $g$,
\begin{align}
\partial_r^{j_1}g(r)= \Big(\frac x{|x|}\cdot \nabla\Big)^{j_1}g(x).\label{22.12}
\end{align}
Therefore, using \eqref{21.37} and the formula \eqref{22.12},   we have that 
\begin{align*}
 \eqref{20.54-I} \lesssim 2^{j_1k}\big\|P_{2^k}\big(\chi_{\ge 1}f\big)\big\|_{L^2(\R^d)}.
\end{align*}

For \eqref{20.54-II}, we note that
\begin{align*}
\eqref{20.54-II}^2=\sum\limits_{h=0}^{+\infty}&2^{-jh}\Big\|\chi_{2^h}(r)r^{\frac{d-1}{2}}r^{-\beta}\!\!\int_0^{+\infty}\!\!\!\Big(J(\rho r)-K(\rho r)\Big) \>\chi_{2^k}(\rho)\rho^{d-1}\notag\\
&\cdot \int_0^{+\infty}\!\!\Big(J(\rho s)+J(-\rho s)\Big) s^{\beta+d-1+j_1}\chi_{\ge 1}(s)\partial_s^{j_1}\Big[P_{2^k}\big(\chi_{\ge 1}f\big)\Big]\,ds d\rho\Big\|_{L^2_r}^2.
\end{align*}
Due to this,  we may split $\eqref{20.54-II}^2$ into the following three subsubparts again. The first one is
\begin{subequations}
\renewcommand{\theequation}
{\theparentequation-ii-\arabic{equation}}
\begin{align}
\sum\limits_{h=0}^{+\infty}2^{-jh}\Big\|&\chi_{2^h}(r)r^{\frac{d-1}{2}}r^{-\beta}\!\!\int_0^{+\infty}\!\!\!\Big(J(\rho r)-K(\rho r)\Big) \>\chi_{2^k}(\rho)\rho^{d-1}\int_0^{+\infty}\!\!\Big(J(\rho s)+J(-\rho s)\Big)\notag\\
&\cdot s^{\beta+d-1+j_1}\chi_{\ge 1}(s) \chi_{2^{h-2}\le \cdot \le 2^{h+2}}(s)\partial_s^{j_1}\Big[P_{2^k}\big(\chi_{\ge 1}f\big)\Big]\,ds d\rho\Big\|_{L^2_r}^2;\label{20.54-II-1}
\end{align}
the second one is
\begin{align}
\sum\limits_{h=0}^{+\infty}2^{-jh}\Big\|&\chi_{2^h}(r)r^{\frac{d-1}{2}}r^{-\beta}\!\!\int_0^{+\infty}\!\!\!\Big(J(\rho r)-K(\rho r)\Big) \>\chi_{2^k}(\rho)\rho^{d-1}\int_0^{+\infty}\!\!\Big(J(\rho s)+J(-\rho s)\Big)\notag\\
&\cdot s^{\beta+d-1+j_1}\chi_{\ge 1}(s) \chi_{\ge 2^{h+2}}(s)\partial_s^{j_1}\Big[P_{2^k}\big(\chi_{\ge 1}f\big)\Big]\,ds d\rho\Big\|_{L^2_r}^2;\label{20.54-II-2}
\end{align}
and the third one is 
\begin{align}
\sum\limits_{h=0}^{+\infty}2^{-jh}\Big\|&\chi_{2^h}(r)r^{\frac{d-1}{2}}r^{-\beta}\!\!\int_0^{+\infty}\!\!\!\Big(J(\rho r)-K(\rho r)\Big) \>\chi_{2^k}(\rho)\rho^{d-1}\int_0^{+\infty}\!\!\Big(J(\rho s)+J(-\rho s)\Big)\notag\\
&\cdot s^{\beta+d-1+j_1}\chi_{\ge 1}(s) \chi_{\le 2^{h-2}}(s)\partial_s^{j_1}\Big[P_{2^k}\big(\chi_{\ge 1}f\big)\Big]\,ds d\rho\Big\|_{L^2_r}^2.\label{20.54-II-3}
\end{align}
\end{subequations}
Similar as \eqref{20.54-I},  from Definition \ref{def:outgong-incoming}  and Proposition \ref{prop:bound-fpm-L2}, we have
\begin{align*}
\eqref{20.54-II-1} \lesssim \sum\limits_{h=0}^{+\infty}2^{-jh}\Big\|r^{j_1}\chi_{2^{h-2}\le \cdot \le 2^{h+2}}(r)\chi_{\ge 1}(r)\partial_r^{j_1}\Big[P_{2^k}\big(\chi_{\ge 1}f\big)\Big]\Big\|_{L^2(\R^d)}^2.
\end{align*}
Since $j_1\le j$, it is further dominated by
\begin{align*}
\sum\limits_{h=0}^{+\infty}\Big\|\chi_{2^{h-2}\le \cdot \le 2^{h+2}}(r)\chi_{\ge 1}(r)\partial_r^{j_1}\Big[P_{2^k}\big(\chi_{\ge 1}f\big)\Big]\Big\|_{L^2(\R^d)}^2
\lesssim \Big\|\chi_{\ge 1}(r)\partial_r^{j_1}\Big[P_{2^k}\big(\chi_{\ge 1}f\big)\Big]\Big\|_{L^2(\R^d)}^2.
\end{align*}
Using \eqref{22.12} again, it is bounded by $2^{2j_1k}\big\|P_{2^k}\big(\chi_{\ge 1}f\big)\big\|_{L^2(\R^d)}^2$. 

For \eqref{20.54-II-2} and \eqref{20.54-II-3}, we only consider the former, since they can be treated in the same manner. From Corollary \ref{cor:Jr-osc2}, we split \eqref{20.54-II-2} into the following two terms again,
\begin{subequations}
\renewcommand{\theequation}
{\theparentequation-ii-2\alph{equation}}
\begin{align}
\sum\limits_{h=0}^{+\infty}2^{-jh}&\Big\|\chi_{2^h}(r)r^{\frac{d-1}{2}}r^{-\beta}\!\!\int_0^{+\infty}\!\!\!\Big(J(\rho r)-K(\rho r)\Big) \>\chi_{2^k}(\rho)\rho^{d-1}\int_0^{+\infty}\!\!\int_{-\frac\pi2}^{\frac\pi2}\!\!\!e^{-2\pi i \rho s\sin \theta}\notag\\
&\cdot\chi_{\ge \frac\pi6}(\theta)\cos^{d-2} \theta  \,d\theta s^{\beta+d-1+j_1}\chi_{\ge 1}(s)\chi_{\ge 2^{h+2}}(s)\partial_s^{j_1}\Big[P_{2^k}\big(\chi_{\ge 1}f\big)\Big]\,ds d\rho\Big\|_{L^2_r}^2;\label{20.54-II-3-1}
\end{align}
and
\begin{align}
\sum\limits_{h=0}^{+\infty}2^{-jh}\Big\|&\chi_{2^h}(r)r^{-\beta}\!\!\int_0^{+\infty}\!\!\!\Big(J(\rho r)-K(\rho r)\Big) \>\chi_{2^k}(\rho)\rho^{d-1}\int_0^{+\infty}\!\!O\big(( \rho s)^{-10}\big)\notag\\
&\cdot\chi_{\ge 1}(s)\chi_{\ge 2^{h+2}}(s)s^{\beta+d-1+j_1}\partial_s^{j_1}\Big[P_{2^k}\big(\chi_{\ge 1}f\big)\Big]\,ds d\rho\Big\|_{L^2(\R^d)}^2.\label{20.54-II-3-2}
\end{align}
\end{subequations}

For \eqref{20.54-II-3-1}, due to the non-resonance of the phase, it can be treated similarly as \eqref{22.08I}.  Indeed,  from Corollary \ref{cor:Jr-osc1}, we  write
\begin{align*}
J(\rho r)-K(\rho r)=\int_0^{\frac\pi2} e^{2\pi i  \rho r\sin \theta}\eta(\theta,\rho r) \,d\theta,
\end{align*}
where $\eta(\theta,\rho r)$ is defined as
\begin{align*}
\eta(\theta,\rho r)=&\chi_{\le 1}(\rho r)\cos^{d-2} \theta+
\chi_{\ge 1}(\rho r)\chi_{\ge \frac\pi6}(\theta)\cos^{d-2} \theta\\
&+\chi_{1\le \cdot \le 2}(\rho r)\chi_{\le \frac\pi6}(\theta)\cos^{d-2} \theta+\chi_{\ge 2}(\rho r)\frac{1}{(2\pi i\rho r)^5}\tilde\eta_d(\theta).
\end{align*}
Therefore, \eqref{20.54-II-3-1} can be rewritten as
\begin{align*}
\sum\limits_{h=0}^{+\infty}2^{-jh}&\Big\|\chi_{2^h}(r)r^{\frac{d-1}{2}}r^{-\beta}\!\!\int_0^{\frac\pi2}\!\!\!\int_{-\frac\pi2}^{\frac\pi2}\!
\int_0^{+\infty}\!\!\int_0^{+\infty}\!\!e^{2\pi i  \rho  (r\sin \theta-s\sin \theta')}\rho^{d-1}\eta(\theta,\rho r) \>\chi_{2^k}(\rho)\notag\\
&\cdot s^{\beta+d-1+j_1}\chi_{\ge 1}(s)\chi_{\ge 2^{h+2}}(s)\partial_s^{j_1}\Big[P_{2^k}\big(\chi_{\ge 1}f\big)\Big]\chi_{\ge \frac\pi6}(\theta')\cos^{d-2} \theta'  \,d\rho ds  d\theta' d\theta\Big\|_{L^2_r}^2.
\end{align*}
Note that we have \eqref{7.12-11.32}. Thus using the formula \eqref{phase-def} and  integration-by-parts 10 times,  we further control \eqref{20.54-II-3-1} as
\begin{align*}
\sum\limits_{h=0}^{+\infty}&\Big\|\chi_{2^h}(r)\!\!
\int_0^{+\infty}\!\!\!\int_0^{+\infty}\!\!\!\rho^{d-11}\>\chi_{2^k}(\rho)s^{d-11}\chi_{\ge 1}(s)\chi_{\ge 2^{h+2}}(s)\left|\partial_s^{j_1}\Big[P_{2^k}\big(\chi_{\ge 1}f\big)\Big]\right|\,ds d\rho \Big\|_{L^2(\R^d)}^2.
\end{align*}
Hence, by the H\"older inequality, we get 
$$
\eqref{20.54-II-3-1}\lesssim \Big\|\chi_{\ge 1}(r)\partial_r^{j_1}\Big[P_{2^k}\big(\chi_{\ge 1}f\big)\Big]\Big\|_{L^2(\R^d)}^2\lesssim
2^{2j_1k}\big\|P_{2^k}\big(\chi_{\ge 1}f\big)\big\|_{L^2(\R^d)}^2.
$$

For \eqref{20.54-II-3-2}, thanks to the high-order decay of $\rho$ and $s$, by H\"older's inequality, it is controlled by
\begin{align*}
\sum\limits_{h=0}^{+\infty}2^{-jh}2^{(\frac{d}{2}-\beta)h}2^{-9k}
\Big\|\chi_{\ge 2^{h+2}}(s)s^{\beta+d-10+j_1}\partial_s^{j_1}\Big[P_{2^k}\big(\chi_{\ge 1}f\big)\Big]\Big\|_{L^2_s}^2.
\end{align*}
Hence, it is also bounded by
$$
2^{2j_1k}\big\|P_{2^k}\big(\chi_{\ge 1}f\big)\big\|_{L^2(\R^d)}^2.
$$
Therefore, we obtain that
$$
\eqref{20.54-II-2}\lesssim 2^{2j_1k}\big\|P_{2^k}\big(\chi_{\ge 1}f\big)\big\|_{L^2(\R^d)}^2.
$$
Collecting the estimates above, we finish the proof of the corollary.
\end{proof}

\vskip0.3cm

\section{Estimates on the incoming/outgoing linear flow}\label{sec:LE-Out-in}
In this subsection, we present some important properties on the incoming/outgoing linear flow.

\subsection{``Incoming/outgoing" decomposition of the linear flow}\label{sec:In/out-linearflow}
Let $N\ge 1$. 
 First of all, we show that the linear flow $e^{it\Delta}(P_{\ge N}\chi_{\ge 1}f)_{out}$ is almost "outgoing" with the frequency dependent velocity.
To do this, we use the Littlewood-Paley decomposition. Without loss of generality, we may assume that $N=2^{k_0}$ for some $k_0\in \N$, then
$$
P_{\ge N}\chi_{\ge 1}f=\sum\limits_{k=k_0}^\infty  P_{2^k}\big(\chi_{\ge 1}f\big),
$$
and thus
$$
\big(P_{\ge N}\chi_{\ge 1}f\big)_{out/in}=\sum\limits_{k=k_0}^\infty \left(P_{2^k}\big(\chi_{\ge 1}f\big)\right)_{out/in}.
$$
Now we consider $\left(P_{2^k}\big(\chi_{\ge 1}f\big)\right)_{out/in}$, for which we only consider the outgoing part. First, from Proposition \ref{prop:f++-highfreq}, we have
\begin{align}
\left(P_{2^k}\big(\chi_{\ge 1}f\big)\right)_{out}=\left(P_{2^k}\big(\chi_{\ge 1}f\big)\right)_{out,k-1\le \cdot \le k+1}+h_k, \label{remainer1}
\end{align}
where $h_k$ satisfies the following estimate,
\begin{align}
\|h_k\|_{H^{\mu(d)}(\R^d)}\lesssim 2^{-10k}\big\| P_{2^{k}}\big(\chi_{\ge 1}f\big)\big\|_{L^2(\R^d)}.\label{est:remainer1}
\end{align}
Due to this estimate,  we only need to consider $\left(P_{2^k}\big(\chi_{\ge 1}f\big)\right)_{out,k-1\le \cdot \le k+1}$. Further, we write
\begin{align*}
\big(P_{2^k}&\big(\chi_{\ge 1}f\big)\big)_{out,k-1\le \cdot \le k+1}\\
&=\chi_{\le \frac14}\left(P_{2^k}\big(\chi_{\ge 1}f\big)\right)_{out,k-1\le \cdot \le k+1}+
\chi_{\ge \frac14}\left(P_{2^k}\big(\chi_{\ge 1}f\big)\right)_{out,k-1\le \cdot \le k+1}.
\end{align*}
From Lemma \ref{lem:supportf+}, we have
\begin{align}
\big\|\chi_{\le \frac 14} \big(P_{2^k}\big(\chi_{\ge 1}f\big)\big)_{out/in,k-1\le \cdot \le k+1}\big\|_{H^{\mu(d)}(\R^d)}
\lesssim
2^{-2k}\big\|P_{2^k}\big(\chi_{\ge 1}f\big)\big\|_{L^2(\R^d)}. \label{est:remainer2}
\end{align}
So it is left to consider $\chi_{\ge \frac14}\left(P_{2^k}\big(\chi_{\ge 1}f\big)\right)_{out,k-1\le \cdot \le k+1}$.  The first estimate is
\begin{prop}\label{prop:pre-outgoing}
Let $k\ge 0$ be an integer. Then there exists $\delta>0$, such that for any triple $(\gamma,q,r)$ satisfying 
that 
\begin{align}
q\ge 2,\,\, r> 2,\,\, 0\le \gamma\le 1,\,\, \frac2q+\frac {2d-1}r< \frac{2d-1}{2}, \label{paremater-condition}
\end{align}
the following estimate holds,
\begin{align*}
\Big\||\nabla|^\gamma\Big[&\chi_{\le \delta(1+2^k t)}e^{it\Delta}\Big(\chi_{\ge \frac14}\left(P_{2^k}\big(\chi_{\ge 1}f\big)\right)_{out,k-1\le \cdot \le k+1}\Big)\Big]\Big\|_{L^q_tL^r_x(\R^+\times\R^d)}\\
&\lesssim
2^{-\big(5-d-(\gamma-\frac 2q-\frac dr)\big)k}\big\|P_{2^k}\big(\chi_{\ge 1}f\big)\big\|_{L^2(\R^d)}.
\end{align*}
The same estimate holds when $e^{it\Delta}$ and $ _{out}$ are replaced by $e^{-it\Delta}$ and $ _{in}$, respectively.
\end{prop}
\begin{proof}
We only consider the estimates on the ``outgoing" part, since the ``incoming" part can be treated in the same way.

First, from the definition,
\begin{align*}
\big(P_{2^k}&\big(\chi_{\ge 1}f\big)\big)_{out,k-1\le \cdot \le k+1}(r)\\
&=r^{-\beta}\int_0^{+\infty}\!\!\! \big(J(\rho r)-K(\rho r)\big)\chi_{2^{k}}(\rho) \rho^{d-1}\mathcal F \left(P_{2^k}\big(\chi_{\ge 1}f\big)\right)(\rho)\,d\rho.
\end{align*}
Here we use the notation $
\chi_{2^{k}}(\rho)=\chi_{2^{k-1}\le\cdot \le 2^{k+1}}(\rho)
$
again. 
We denote $\tilde \eta(\theta, r)$ as
\begin{align*}
\tilde \eta(\theta,  r)=\chi_{\le 2}( r)\chi_{\le \frac\pi6}(\theta) \cos^{d-2} \theta+\chi_{\ge 2}( r)\frac{1}{(2\pi i r)^5}\tilde\eta_d(\theta).
\end{align*}
Then
\begin{align}
\tilde \eta(\theta, r)=O(\langle r\rangle^{-5}).\label{est:tilde-eta}
\end{align}
Moreover, from Corollary \ref{cor:Jr-osc1},
\begin{align*}
J(r)&-K(r)=\int_0^{\frac\pi2} e^{2\pi i  r\sin \theta}\chi_{\ge \frac\pi6}(\theta)\cos^{d-2} \theta  \,d\theta
+\int_0^{\frac\pi2} e^{2\pi i r\sin \theta}\tilde \eta(\theta,  r) \,d\theta.
\end{align*}
Accordingly, we split $\left(P_{2^k}\big(\chi_{\ge 1}f\big)\right)_{out,k-1\le \cdot \le k+1}(r)$ into two parts, which are denoted by $F_1$ and $F_2$, as follows:
\begin{align*}
F_1(r)=r^{-\beta}\int_0^{+\infty}\!\!\! \int_0^{\frac\pi2}\!\! e^{2\pi i\rho r\sin \theta}\tilde \eta(\theta, \rho r) \,d\theta\>
\chi_{2^{k}}(\rho) \rho^{d-1}\mathcal F \left(P_{2^k}\big(\chi_{\ge 1}f\big)\right)(\rho)\,d\rho;
\end{align*}
and
\begin{align*}
F_2(r)=r^{-\beta}\int_0^{+\infty}\!\!\! \int_0^{\frac\pi2}\!\! e^{2\pi i \rho r\sin \theta}\chi_{\ge \frac\pi6}(\theta)\cos^{d-2} \theta  \,d\theta\>\chi_{2^{k}}(\rho) \rho^{d-1}\mathcal F \left(P_{2^k}\big(\chi_{\ge 1}f\big)\right)(\rho)\,d\rho.
\end{align*}

We consider $F_1$ first, and claim that for $s\in [0,3]$,
\begin{align}
\|\chi_{\ge \frac14}F_1\|_{\dot H^s(\R^d)}\lesssim 2^{-(5-\frac{d}{2}-s)k}\big\|P_{2^k}\big(\chi_{\ge 1}f\big)\big\|_{L^2(\R^d)}.\label{est:F2}
\end{align}
To prove \eqref{est:F2}, we need the following lemma.
\begin{lem}\label{lem:Ff} Let $d=3,4,5$, and $f\in L^2(\R^d)$, then
\begin{align*}
\big\|\chi_{\gtrsim 2^k}(\rho)\mathcal F \left(P_{2^k}\big(\chi_{\ge 1}f\big)\right)\big\|_{L^2_\rho}
\lesssim
2^{-\frac{d-1}{2}k}\big\|P_{2^k}\big(\chi_{\ge 1}f\big)\big\|_{L^2(\R^d)}.
\end{align*}
\end{lem}
\begin{proof}
Note that 
\begin{align}
\big\|\chi_{\gtrsim 2^k}(\rho)\mathcal F \left(P_{2^k}\big(\chi_{\ge 1}f\big)\right)\big\|_{L^2_\rho}
=&\big\||\xi|^{-\frac{d-1}{2}}\chi_{\gtrsim 2^k}(\xi)\mathcal F \left(P_{2^k}\big(\chi_{\ge 1}f\big)\right)(\xi)\big\|_{L^2_\xi(\R^d)}\notag\\
\lesssim& 
2^{-\frac{d-1}{2}k}\big\|\mathcal F \left(P_{2^k}\big(\chi_{\ge 1}f\big)\right)\big\|_{L^2_\xi(\R^d)}.\label{3.38-0801}
\end{align}
From \eqref{deformed-Fourier}, 
$$\mathcal F \left(P_{2^k}\big(\chi_{\ge 1}f\big)\right)
=
\mathscr F \left(|x|^\beta P_{2^k}\big(\chi_{\ge 1}f\big)\right).
$$
Then by the Plancherel identity, we obtain  
\begin{align*}
 \eqref{3.38-0801}
 =&
  2^{-\frac{d-1}{2}k}\big\| |x|^\beta P_{2^k}\big(\chi_{\ge 1}f\big)\big\|_{L^2_x(\R^d)}\\
 \le &
 2^{-\frac{d-1}{2}k}\big\| |x|^\beta \chi_{\le \frac14} P_{2^k}\big(\chi_{\ge 1}f\big)\big\|_{L^2_x(\R^d)}+2^{-\frac{d-1}{2}k}\big\| |x|^\beta \chi_{\ge \frac14} P_{2^k}\big(\chi_{\ge 1}f\big)\big\|_{L^2_x(\R^d)}.
\end{align*}
Note that $3\le d\le 5$, we have that $-1\le \beta\le 0$. Then  it is further bounded by 
$$
2^{-\frac{d-1}{2}k}\big\|\chi_{\le \frac14} P_{2^k}\big(\chi_{\ge 1}f\big)\big\|_{L^\infty(\R^d)}+2^{-\frac{d-1}{2}k}\big\|P_{2^k}\big(\chi_{\ge 1}f\big)\big\|_{L^2(\R^d)}.
$$  
From Lemma \ref{lem:mismatch},  it is controlled by  $2^{-\frac{d-1}{2}k}\big\|P_{2^k}\big(\chi_{\ge 1}f\big)\big\|_{L^2(\R^d)}$. 
This  gives the desired estimate.
\end{proof}

By \eqref{est:tilde-eta}, H\"older's inequality and  Lemma \ref{lem:Ff}, we get
\begin{align*}
\chi_{\ge \frac14}(r)|F_1(r)|
&\lesssim
r^{-\beta-5}\Big\|
\chi_{2^{k}}(\rho) \rho^{d-6}\Big\|_{L^2_\rho} \cdot \big\|\chi_{\gtrsim 2^k}(\rho)\mathcal F \left(P_{2^k}\big(\chi_{\ge 1}f\big)\right)\big\|_{L^2_\rho}\\
&\lesssim
r^{-\beta-5}2^{(d-\frac{11}{2})k}2^{-\frac{d-1}{2}k}\big\|P_{2^k}\big(\chi_{\ge 1}f\big)\big\|_{L^2(\R^d)}\\
&\lesssim
r^{-\beta-5}2^{(\frac{d}{2}-5)k}\big\|P_{2^k}\big(\chi_{\ge 1}f\big)\big\|_{L^2(\R^d)}.
\end{align*}
Therefore, we obtain
\begin{align*}
\|\chi_{\ge \frac14}F_1\|_{L^2(\R^d)}\lesssim 2^{(\frac{d}{2}-5)k}\big\|P_{2^k}\big(\chi_{\ge 1}f\big)\big\|_{L^2(\R^d)}.
\end{align*}
For the estimate on the high-order derivatives, we note that for $j=1,2,3$, 
$$
\partial_r^j\int_0^{+\infty}\!\!\! \int_0^{\frac\pi2}\!\! e^{2\pi i\rho r\sin \theta}\tilde \eta(\theta, \rho r) \,d\theta=  O( r^{-5}\rho^{-5+j}).
$$
Hence arguing similarly as above, we get that for $s=1,2,3$,
\begin{align*}
\|\chi_{\ge \frac14}F_1\|_{\dot H^s(\R^d)}\lesssim 2^{-(5-\frac{d}{2}-s)k}\big\|P_{2^k}\big(\chi_{\ge 1}f\big)\big\|_{L^2(\R^d)}.
\end{align*}
For general $s\in[0,3]$,  we obtain it by the interpolation.
This proves \eqref{est:F2}.

Using Lemma \ref{lem:frac_Hs}, we obtain that
\begin{align*}
\Big\||\nabla|^\gamma&\Big[\chi_{\le \delta(1+2^k t)}e^{it\Delta}\big(\chi_{\ge \frac14}F_1\big)\Big]\Big\|_{L^{q}_tL^r_x(\R^+\times\R^d)}
\lesssim
\Big\|\langle\nabla\rangle^\gamma e^{ it\Delta}\big(\chi_{\ge \frac14}F_1\big)\Big\|_{L^{q}_tL^r_x(\R^+\times\R^d)}.
\end{align*}
Then from   Lemma \ref{lem:radial-Str}, it can be controlled by
\begin{align*}
\|\chi_{\ge \frac14}F_1\|_{ H^s(\R^d)},\quad \mbox{ with }\quad s=\frac d2+\gamma-\frac 2q-\frac dr.
\end{align*}
Now applying \eqref{est:F2}, it gives
\begin{align}
\Big\||\nabla|^\gamma&\Big[\chi_{\le \delta(1+2^k t)}e^{it\Delta}\big(\chi_{\ge \frac14}F_1\big)\Big]\Big\|_{L^{q}_tL^r_x(\R^+\times\R^d)}\lesssim
2^{-\big(5-d-(\gamma-\frac 2q-\frac dr)\big)k}\big\|P_{2^k}\big(\chi_{\ge 1}f\big)\big\|_{L^2(\R^d)}.\label{est:F-1}
\end{align}

Now we consider $e^{it\Delta}\big(\chi_{\ge \frac14}F_2\big)$, which is equal to
\begin{align}
\int_0^{+\infty}\!\!\! &\int_0^{\frac\pi2}\!\! e^{it\Delta}\Big(r^{-\beta}\chi_{\ge \frac14}(r) e^{2\pi i \rho r\sin \theta}\Big)\chi_{\ge \frac\pi6}(\theta)\cos^{d-2} \theta  \,d\theta\notag\\
&\cdot\chi_{2^{k}}(\rho) \rho^{d-1}\mathcal F \left(P_{2^k}\big(\chi_{\ge 1}f\big)\right)(\rho)\,d\rho.\label{22.50}
\end{align}
To give its estimate, we consider the following term,
$$
e^{it\Delta}\Big(r^{-\beta}\chi_{\ge \frac14}(r) e^{2\pi i \rho r\sin \theta}\Big).
$$
Using the formula in Lemma \ref{lem:formula-St}, we have
\begin{align*}
e^{it\Delta}\Big(r^{-\beta}\chi_{\ge \frac14}(r) e^{2\pi i \rho r\sin \theta}\Big)
=\frac{c}{t^{\frac d2}}\int_{\R^d}e^{i\frac{|x-y|^2}{4t}+2\pi i \rho|y|\sin \theta}|y|^{-\beta}\chi_{\ge \frac14}(y) \,dy.
\end{align*}
We expand the phase and rewrite it as
\begin{align}
\frac{c}{t^{\frac d2}}e^{i\frac{|x|^2}{4t}}& \int_{\R^d}e^{-i\frac{x\cdot y}{2t}+i\frac{|y|^2}{4t}+2\pi i \rho|y|\sin \theta}|y|^{-\beta}\chi_{\ge \frac14}(y)\,dy\notag\\
=& \frac{c}{t^{\frac d2}}e^{i\frac{|x|^2}{4t}}\int_{|\omega|=1}\int_0^{+\infty} e^{i\phi(r)}r^{d-1-\beta}\chi_{\ge \frac14}(r)\,dr d\omega,\label{20.46}
\end{align}
where $\phi(r)=-\frac{x\cdot \omega}{2t}r+\frac{r^2}{4t}+2\pi \rho r\sin \theta$. Then
$$
\phi'(r)=-\frac{x\cdot \omega}{2t}+\frac{r}{2t}+2\pi\rho\sin \theta.
$$
Note that $\rho\sim  2^k$, $r\sim  1$ and $\sin \theta\ge\frac14$, so choosing $\delta$ small enough, we obtain that  when $|x|\le \delta(1+2^kt)$,
\begin{align}
\phi'(r)\ge  \frac14\Big(\frac{r}{t}+\pi\rho\Big). \label{est:psi'}
\end{align}
Moreover,
\begin{align}
\phi''(r)=\frac1{2t},\quad \mbox{and }\quad \phi^{(j)}(r)=0,\quad \mbox{for } j=3,4,\cdots. \label{est:psi''}
\end{align}
Hence we can use the formula,
\begin{align*}
e^{i\phi(r)}=\frac1{i\phi'(r)} \partial_r\big(e^{i\phi(r)}\big),
\end{align*}
and integrate by parts to obtain that for any $K\in\Z^+$, there exists $c_K\in \C$, such that
\begin{align}
e^{it\Delta}\Big(& r^{-\beta}\chi_{\ge \frac14}(r)e^{2\pi i \rho r\sin \theta}\Big)\notag\\
&=
\frac{c_K}{t^{\frac d2}}e^{i\frac{|x|^2}{4t}}\int_{|\omega|=1}\int_0^{+\infty}\!\!\! e^{i\phi(r)} \partial_r\Big(\frac{1}{\phi'(r)}\partial_r\Big)^{K-1}\Big[\frac{1}{\phi'(r)}r^{d-1-\beta}\chi_{\ge \frac14}(r)\Big]\,dr d\omega.
\label{22.44}
\end{align}
Using Lemma \ref{lem:muli-Lei-formula}, we have
\begin{align*}
\partial_r\Big(\frac{1}{\phi'(r)}&\partial_r\Big)^{K-1}\Big[\frac{1}{\phi'(r)}r^{d-1-\beta}\chi_{\ge \frac14}(r)\Big]
\\
&=\sum\limits_{\begin{subarray}{c}
l_1,\cdots,l_K\in\N,l'\in\N;\\
l_j\le
j;l_1+\cdots+l_K+l'=K
\end{subarray}}
\!\!\!C_{l_1,\cdots,l_K,l'}\partial_r^{l_1}\Big(\frac{1}{\phi'(r)}\Big)\cdots
\partial_r^{l_K}\Big(\frac{1}{\phi'(r)}\Big)\>\partial_r^{l'}\Big[r^{d-1-\beta}\chi_{\ge \frac14}(r)\Big].
\end{align*}
Since $\phi^{(j)}(r)=0$ when $j\ge 3$, we find that for any integer $l\ge 0$,
$$
\partial_r^{l}\Big(\frac{1}{\phi'(r)}\Big)
=\frac{(\phi''(r))^l}{(\phi'(r))^{l+1}},
$$
it implies that
$$
\partial_r^{l_1}\Big(\frac{1}{\phi'(r)}\Big)\cdots
\partial_r^{l_K}\Big(\frac{1}{\phi'(r)}\Big)
=\frac{(\phi''(r))^{l_1+\cdots+l_K}}{(\phi'(r))^{l_1+\cdots+l_K+K}}.
$$
This together with \eqref{est:psi'}, \eqref{est:psi''}, gives
\begin{align*}
\Big|\partial_r^{l_1}\Big(\frac{1}{\phi'(r)}\Big)\cdots
\partial_r^{l_K}\Big(\frac{1}{\phi'(r)}\Big)\Big|
\lesssim
\frac{1}{(r+\rho t)^{l_1+\cdots+l_K}\big(\frac{r}{t}+\rho\big)^K}.
\end{align*}
Moreover, we have
\begin{align*}
\Big|\partial_r^{l'}\Big[r^{d-1-\beta}\chi_{\ge \frac14}(r)\Big]\Big|
\lesssim
r^{d-1-\beta-l'}\chi_{\gtrsim 1}(r).
\end{align*}
Combining the above two estimates, and noting that $l_1+\cdots+l_K+l'=K$, we obtain that 
\begin{align*}
\Big|\partial_r^{l_1}\Big(\frac{1}{\phi'(r)}\Big)\cdots
\partial_r^{l_K}\Big(\frac{1}{\phi'(r)}\Big)\Big|\>\Big|\partial_r^{l'}\Big[r^{d-1-\beta}\chi_{\ge \frac14}(r)\Big]\Big|
\lesssim
\frac{1}{r^K\big(\frac{r}{t}+\rho\big)^K}r^{d-1-\beta}\chi_{\gtrsim 1}(r).
\end{align*}
Note that when $0\le t\le 1$, it is bounded by
$$
\frac{t^3}{r^{K+3}\rho^{K-3}}r^{d-1-\beta}\chi_{\gtrsim 1}(r);
$$
when $t\ge 1$, it is bounded by
$$
\frac{1}{r^K\rho^K}r^{d-1-\beta}\chi_{\gtrsim 1}(r).
$$
Hence, choosing $K$ suitably large, we get
$$
\Big|\partial_r\Big(\frac{1}{\phi'(r)}\partial_r\Big)^K\Big[\frac{1}{\phi'(r)}r^{d-1-\beta}\chi_{\ge \frac14}(r)\Big]\Big|
\lesssim t^3\langle t\rangle^{-3}r^{-10}\rho^{-10}\chi_{\gtrsim 1}(r).
$$
Inserting this estimate into \eqref{22.44}, we obtain that
\begin{align*}
\Big|\chi_{\le \delta(1+2^kt)}e^{it\Delta}\Big(& r^{-\beta}\chi_{\ge \frac14}(r)e^{2\pi i \rho r\sin \theta}\Big)\Big|
\lesssim
\langle t\rangle^{-\frac d2}\rho^{-10}.
\end{align*}
Hence, from this estimate and \eqref{22.50}, we obtain that
\begin{align*}
\Big|\chi_{\le \delta(1+2^kt)}e^{it\Delta}\big(\chi_{\ge \frac14}F_2\big)\Big|
\lesssim
\langle t\rangle^{-\frac d2}\int_0^{+\infty}\!\!\!\big|\chi_{2^{k}}(\rho) \rho^{d-11}\mathcal F \left(P_{2^k}\big(\chi_{\ge 1}f\big)\right)(\rho)\big|\,d\rho.
\end{align*}
Using H\"older's inequality, it is further controlled by
$$
\langle t\rangle^{-\frac d2} \Big\|
\chi_{2^{k}}(\rho) \rho^{d-11}\Big\|_{L^2_\rho} \cdot \big\|\chi_{\gtrsim 2^k}(\rho)\mathcal F \left(P_{2^k}\big(\chi_{\ge 1}f\big)\right)\big\|_{L^2_\rho}.
$$
Now by Lemma \ref{lem:Ff}, it is further bounded by
$$
\langle t\rangle^{-\frac d2} 2^{(d-\frac{21}{2})k}2^{-\frac{d-1}{2}k}\big\|P_{2^k}\big(\chi_{\ge 1}f\big)\big\|_{L^2(\R^d)}.
$$
That is,
\begin{align*}
\Big|\chi_{\le \delta(1+2^kt)}e^{it\Delta}\big(\chi_{\ge \frac14}F_2\big)\Big|
\lesssim
\langle t\rangle^{-\frac d2} 2^{(\frac d2-10)k}\big\|P_{2^k}\big(\chi_{\ge 1}f\big)\big\|_{L^2(\R^d)}.
\end{align*}
Note that  the structure of $e^{it\Delta}\big(\chi_{\ge \frac14}F_2\big)$ does not essentially change when we consider the estimates on the high-order derivatives, hence in the same manner, we obtain that for any even integer $\tilde{s}\in [0,10]$,
\begin{align*}
\Big||\nabla|^{\tilde{s}}\Big[\chi_{\le \delta(1+2^kt)}e^{it\Delta}\big(\chi_{\ge \frac14}F_2\big)\Big]\Big|
\lesssim
\langle t\rangle^{-\frac d2} 2^{(\frac d2-10)k}\big\|P_{2^k}\big(\chi_{\ge 1}f\big)\big\|_{L^2(\R^d)}.
\end{align*}
Therefore, by H\"older's inequality, if
$$
\frac 1q+\frac dr<\frac d2, \quad q\ge 2,\quad  r\ge 2 
$$
(note that this condition is implied by \eqref{paremater-condition}),
then we have
\begin{align*}
\Big\||\nabla|^{\tilde{s}}\Big[\chi_{\le \delta(1+2^k t)}e^{it\Delta}\big(\chi_{\ge \frac14}F_2\big)\Big]\Big\|_{L^q_tL^r_x(\R^+\times\R^d)}
\lesssim
 2^{(\frac d2-10+\frac dr)k}\big\|P_{2^k}\big(\chi_{\ge 1}f\big)\big\|_{L^2(\R^d)}.
 \end{align*}
Since $d\le 5$, by the interpolation, we obtain that for any $s\in [0,10]$, 
\begin{align}
\Big\||\nabla|^s\Big[\chi_{\le \delta(1+2^k t)}e^{it\Delta}\big(\chi_{\ge \frac14}F_2\big)\Big]\Big\|_{L^q_tL^r_x(\R^+\times\R^d)}
 \lesssim
 2^{-5k}\big\|P_{2^k}\big(\chi_{\ge 1}f\big)\big\|_{L^2(\R^d)}.\label{est:F-2}
\end{align}
Now collecting the estimates in \eqref{est:F-1} and \eqref{est:F-2}, we prove the proposition.
\end{proof}

The following result is the incoming/outgoing linear flow's estimate related to the outside region.
\begin{prop}\label{prop:outgoing} Let $k\ge 0$ be an integer. Moreover, let $ r, \gamma_1, \gamma_2, s$ be the parameters satisfying
$$
r>2,\quad, \gamma_1\ge0,\quad  \gamma_2\ge 0,\quad s+\frac1r\ge \frac12, \quad \gamma_1+s=d\big(\frac12-\frac1r\big).
$$
Then for any $t>0$,
\begin{align*}
\Big\||\nabla|^{\gamma_2}\Big[ &\chi_{\ge \delta(1+2^k t)}e^{it\Delta}\Big(\chi_{\ge \frac14}\left(P_{2^k}\big(\chi_{\ge 1}f\big)\right)_{out,k-1\le \cdot \le k+1}\Big)\Big]\Big\|_{L^{ r}(\R^d)}\\
&\lesssim
(1+2^k t)^{-\gamma_1}2^{(\gamma_2+s+)k}\big\|P_{2^k}\big(\chi_{\ge 1}f\big)\big\|_{L^2(\R^d)}.
\end{align*}
The same estimate holds when $e^{it\Delta}$ and $ _{out}$ are replaced by $e^{-it\Delta}$ and $ _{in}$, respectively.
\end{prop}
\begin{proof} As before, we only consider the estimate on the ``outgoing'' part.   Moreover, we only need to consider the case of $r<\infty$. Indeed, when $r=\infty$, then by the Sobolev inequality,
\begin{align*}
\Big\||\nabla|^{\gamma_2}\Big[ &\chi_{\ge \delta(1+2^k t)}e^{it\Delta}\Big(\chi_{\ge \frac14}\left(P_{2^k}\big(\chi_{\ge 1}f\big)\right)_{out,k-1\le \cdot \le k+1}\Big)\Big]\Big\|_{L^{\infty}(\R^d)}\\
&\lesssim
\Big\||\nabla|^{\gamma_2}\langle\nabla\rangle^{0+}\Big[\chi_{\ge \delta(1+2^k t)}e^{it\Delta}\Big(\chi_{\ge \frac14}\left(P_{2^k}\big(\chi_{\ge 1}f\big)\right)_{out,k-1\le \cdot \le k+1}\Big)\Big]\Big\|_{L^{\infty-}(\R^d)}.
\end{align*}
Hence we only need to consider the case of $r<\infty$, by replacing $\gamma_2$ by $\gamma_2+$ if necessary. 

First, we split
$$|\nabla|^{\gamma_2}\Big[\chi_{\ge \delta(1+2^k t)} e^{it\Delta}\Big(\chi_{\ge \frac14}\left(P_{2^k}\big(\chi_{\ge 1}f\big)\right)_{out,k-1\le \cdot \le k+1}\Big)\Big]$$
into the following three terms,
\begin{subequations}
\begin{align}
 &P_{\le1}|\nabla|^{\gamma_2}\Big[\chi_{\ge \delta(1+2^k t)} e^{it\Delta}\Big(\chi_{\ge \frac14}\left(P_{2^k}\big(\chi_{\ge 1}f\big)\right)_{out,k-1\le \cdot \le k+1}\Big)\Big]\label{0.44-I}\\
&\quad+\chi_{\le \delta^2(1+2^kt)}P_{\ge1} |\nabla|^{\gamma_2}\Big[\chi_{\ge \delta(1+2^k t)} e^{it\Delta}\Big(\chi_{\ge \frac14}\left(P_{2^k}\big(\chi_{\ge 1}f\big)\right)_{out,k-1\le \cdot \le k+1}\Big)\Big]\label{0.44-II}\\
&\qquad+\chi_{\ge \delta^2(1+2^kt)} P_{\ge1}|\nabla|^{\gamma_2}\Big[\chi_{\ge \delta(1+2^k t)} e^{it\Delta}\Big(\chi_{\ge \frac14}\left(P_{2^k}\big(\chi_{\ge 1}f\big)\right)_{out,k-1\le \cdot \le k+1}\Big)\Big].\label{0.44-III}
\end{align}
\end{subequations}

Now we consider the first term \eqref{0.44-I}. From Lemma \ref{lem:radial-Sob}, under the assumptions, we have
$$
\big\||x|^{\gamma_1} g\big\|_{L^ r(\R^d)}\lesssim \|g\|_{\dot H^s(\R^d)}.
$$
Hence, this last estimate combined with the Bernstein inequality and Lemma \ref{lem:strichartz}, gives that
\begin{align*}
\big\|\eqref{0.44-I}\big\|_{L^{ r}(\R^d)}
\lesssim & \Big\|\chi_{\ge \delta(1+2^k t)} e^{it\Delta}\Big(\chi_{\ge \frac14}\left(P_{2^k}\big(\chi_{\ge 1}f\big)\right)_{out,k-1\le \cdot \le k+1}\Big)\Big\|_{L^{ r}(\R^d)}\notag\\
&\,\,\lesssim \big(1+2^k t\big)^{-\gamma_1} \Big\|e^{it\Delta}\Big(\chi_{\ge \frac14}\left(P_{2^k}\big(\chi_{\ge 1}f\big)\right)_{out,k-1\le \cdot \le k+1}\Big)\Big\|_{\dot H^{s}(\R^d)}\notag\\
 &\,\,  \lesssim\big(1+2^k t\big)^{-\gamma_1} \left\|\chi_{\ge \frac14}\left(P_{2^k}\big(\chi_{\ge 1}f\big)\right)_{out,k-1\le \cdot \le k+1}\right\|_{\dot H^{s}(\R^d)}.
\end{align*}
Therefore, by Corollary  \ref{cor:f+-Hs}, we get
\begin{align*}
\big\|\eqref{0.44-I}\big\|_{L^{ r}(\R^d)}
\lesssim &  \big(1+2^k t\big)^{-\gamma_1} 2^{sk}\big\|P_{2^k}\big(\chi_{\ge 1}f\big)\big\|_{L^2(\R^d)}.
\end{align*}

For the second term \eqref{0.44-II}, by Lemma \ref{lem:mismatch}, we have that for any $M>0$, $K\ge 1$,
$$
\Big\|\chi_{\le \delta K} P_{\ge1} |\nabla|^{\gamma_2}\big(\chi_{\ge K}g\big)\Big\|_{L^ r(\R^d)}
\lesssim_{\delta,M} K^{-M} \big\|\chi_{\ge K}g\big\|_{L^ r(\R^d)}.
$$
Then using the inequality above and treating similarly as \eqref{0.44-I}, we obtain that
\begin{align*}
\big\|\eqref{0.44-II}\big\|_{L^{ r}(\R^d)}
\lesssim_M &
(1+2^k t)^{-M}\Big\|\chi_{\ge \delta(1+2^k t)} e^{it\Delta}\Big(\chi_{\ge \frac14}\left(P_{2^k}\big(\chi_{\ge 1}f\big)\right)_{out,k-1\le \cdot \le k+1}\Big)\Big\|_{L^{ r}(\R^d)}\\
\lesssim_M &
(1+2^k t)^{-M-\gamma_1}\big\|\chi_{\ge \frac14}\left(P_{2^k}\big(\chi_{\ge 1}f\big)\right)_{out,k-1\le \cdot \le k+1}\big\|_{\dot H^s(\R^d)}\\
\lesssim_M &
\big(1+2^k t\big)^{-M-\gamma_1} 2^{sk}\big\|P_{2^k}\big(\chi_{\ge 1}f\big)\big\|_{L^2(\R^d)}.
\end{align*}

For the third term  \eqref{0.44-III}, we use the similar argument as above, and obtain that
\begin{align*}
\big\|\eqref{0.44-III}\big\|_{L^{ r}(\R^d)}
\lesssim
&(1+2^k t)^{-\gamma_1}
\Big\|\chi_{\ge \delta(1+2^k t)} 
e^{it\Delta}\Big(\chi_{\ge \frac14}\left(P_{2^k}\big(\chi_{\ge 1}f\big)\right)_{out,k-1\le \cdot \le k+1}\Big)
\Big\|_{\dot H^{\gamma_2+s}(\R^d)}.
\end{align*}
From Lemma \ref{lem:frac_Hs} and Lemma \ref{lem:strichartz}, it is further controlled by
\begin{align*}
(1+2^k t)^{-\gamma_1}\Big\|\chi_{\ge \frac14}\left(P_{2^k}\big(\chi_{\ge 1}f\big)\right)_{out,k-1\le \cdot \le k+1}
\Big\|_{H^{\gamma_2+s}(\R^d)}.
\end{align*}
Now using Corollary \ref{cor:f+-Hs}, we get
\begin{align*}
\big\|\eqref{0.44-III}\big\|_{L^{ r}(\R^d)}
\lesssim
(1+2^k t)^{-\gamma_1} 2^{(\gamma_2+s)k}\big\|P_{2^k}\big(\chi_{\ge 1}f\big)\big\|_{L^2(\R^d)}.
\end{align*}
Combining the above estimates on \eqref{0.44-I}--\eqref{0.44-III}, we finish the proof of the proposition.
\end{proof}

Now an easy consequence of this proposition is the following space-time estimate.
\begin{cor}\label{cor:outgoing}
Let $(\gamma,q,r)$ be the triple satisfying
\begin{align}
\gamma\ge 0, \,\, q\ge 1,\,\, r> 2,\,\, \frac1q<(d-1)\big(\frac12-\frac1r\big),\label{paremater-condition-2}
\end{align}
then the following estimate holds,
\begin{align}
\Big\||\nabla|^{\gamma}\Big[&\chi_{\ge \delta(1+2^k t)}e^{it\Delta}\Big(\chi_{\ge \frac14}\left(P_{2^k}\big(\chi_{\ge 1}f\big)\right)_{out,k-1\le \cdot \le k+1}\Big)\Big]\Big\|_{L^q_tL^r_x(\R^+\times\R^d)}\notag\\
&\lesssim
2^{(-\frac1q-\frac1r+\frac12+\gamma+) k}\big\|P_{2^k}\big(\chi_{\ge 1}f\big)\big\|_{L^2(\R^d)}.\label{est:out-1}
\end{align}
Moreover, for any $\delta>0$, 
\begin{align}
\Big\||\nabla|^{\gamma}\Big[&\chi_{\ge \delta(1+2^k t)}e^{it\Delta}\Big(\chi_{\ge \frac14}\left(P_{2^k}\big(\chi_{\ge 1}f\big)\right)_{out,k-1\le \cdot \le k+1}\Big)\Big]\Big\|_{L^2_tL^\infty_x([\delta,+\infty)\times\R^d)}\notag\\
&\lesssim_\delta
2^{(-\frac{d-2}2+\gamma+) k}\big\|P_{2^k}\big(\chi_{\ge 1}f\big)\big\|_{L^2(\R^d)}.\label{est:out-2}
\end{align}
The same estimate holds when $e^{it\Delta}$ and $ _{out}$ are replaced by $e^{-it\Delta}$ and $ _{in}$, respectively.
\end{cor}
\begin{proof}
From Proposition \ref{prop:outgoing}, we have
\begin{align}
\Big\||\nabla|^{\gamma}\Big[&\chi_{\ge \delta(1+2^k t)}e^{it\Delta}\Big(\chi_{\ge \frac14}\left(P_{2^k}\big(\chi_{\ge 1}f\big)\right)_{out,k-1\le \cdot \le k+1}\Big)\Big]\Big\|_{L^q_tL^r_x(\R^+\times\R^d)}\notag\\
&\lesssim
2^{(\gamma+s+)k}\big\|(1+2^k t)^{-\gamma_1}\big\|_{L^q_t(\R^+)}\big\|P_{2^k}\big(\chi_{\ge 1}f\big)\big\|_{L^2(\R^d)},
\label{11.47}
\end{align}
where  the parameters 
$$
s=\frac12-\frac1r,\quad \gamma_1= (d-1)\big(\frac12-\frac1r\big).
$$
By the condition \eqref{paremater-condition-2}, we have that $\frac1q<\gamma_1$, and thus
\begin{align*}
\big\|(1+2^k t)^{-\gamma_1}\big\|_{L^q_t(\R^+)}
\lesssim
2^{-\frac1q k}.
\end{align*}
Combining this estimate with \eqref{11.47}, we prove \eqref{est:out-1}. Since
\begin{align*}
\big\|(1+2^k t)^{-\gamma_1}\big\|_{L^q_t([\delta,+\infty))}
\lesssim_\delta
2^{-\gamma_1 k},
\end{align*}
we have \eqref{est:out-2}. This finishes the proof of the corollary.
\end{proof}

\subsection{Improved Strichartz's estimates}
Now we collect the estimates above, and obtain the following results.
\begin{prop}\label{prop:f+} Let $N\ge 1, s_0> 0$. Suppose that $(s,q,r)$ is the triple satisfying that $(q,r)=(\infty,2)$ or 
\begin{align*}
q\ge 2,\,\, r\ge 2,\,\,  
\quad \frac1{q}<(d-1)\big(\frac12-\frac1{r}\big);
\end{align*}
\begin{align*}
0\le  s\le 1, \quad s+\frac d2-\left(\frac2q+\frac dr\right)\le \mu(d);
\end{align*}
and
\begin{align*}
\gamma_0=s_0-s-\max\left\{(d-5)-\left(\frac 2{q}+\frac d{r}\right), \frac12-\frac1{q}-\frac1{r}\right\}>0.
\end{align*}
Then  
 \begin{align}\label{est:Lqr_out_in}
 \left\||\nabla|^{s}e^{it\Delta} \left(P_{\ge N}\big(\chi_{\ge 1}f\big)\right)_{out}\right\|_{L^{q}_t L^{r}_x(\R^{+}\times\R^d)}
\lesssim
N^{-\gamma_0+}\|P_{\ge N}\chi_{\ge 1}f\|_{H^{s_0}(\R^d)}.
\end{align}
Moreover, for any $\delta$,
 \begin{align}\label{est:Lqr_out_in-2}
 \left\|\nabla e^{it\Delta} \left(P_{\ge N}\big(\chi_{\ge 1}f\big)\right)_{out}\right\|_{L^{2}_t L^\infty_x([\delta,+\infty)\times\R^d)}
\lesssim
N^{1-s_0-\frac{d-2}2+}\|P_{\ge N}\chi_{\ge 1}f\|_{H^{s_0}(\R^d)}.
\end{align}
The same estimate holds when $e^{it\Delta}$ and $ _{out}$ are replaced by $e^{-it\Delta}$ and $ _{in}$, respectively.
\end{prop}
\begin{remark}\label{rem:par}
Now we list some triples $(s,q,r)$ which satisfy the conditions in Proposition \ref{prop:f+}  and will be used later. 
First, for any 
$
r> \frac{2(d-1)}{d-2}, 0\le s\le 1, s_0>s-\frac1r,
$
we have 
\begin{align}\label{L2Lr}
\left\||\nabla|^se^{it\Delta} \left(P_{\ge N}\big(\chi_{\ge 1}f\big)\right)_{out}\right\|_{L^2_tL^r_x(\R^{+}\times\R^d)}
\lesssim N^{s-s_0-\frac1r+}\big\|P_{\ge N}\chi_{\ge 1}f\big\|_{H^{s_0}(\R^d)}.
\end{align}
Second, for any $r>2,s_0>\frac12-\frac1r$, 
\begin{align}\label{LinftyLr}
\left\|e^{it\Delta} \left(P_{\ge N}\big(\chi_{\ge 1}f\big)\right)_{out}\right\|_{L^\infty_tL^r_x(\R^{+}\times\R^d)}
\lesssim N^{-s_0+\frac12-\frac1r+}\big\|P_{\ge N}\chi_{\ge 1}f\big\|_{H^{s_0}(\R^d)}.
\end{align}
Third, for any $p\ge \max\{1,\frac4d\}$ and any $s_0>\frac12-\frac{d+2}{2dp}$, 
\begin{align}\label{L2pLdp}
\left\|e^{it\Delta} \left(P_{\ge N}\big(\chi_{\ge 1}f\big)\right)_{out}\right\|_{L^{2p}_tL^{dp}_x(\R^{+}\times\R^d)}
\lesssim N^{-s_0+\frac12-\frac{d+2}{2dp}+}\big\|P_{\ge N}\chi_{\ge 1}f\big\|_{H^{s_0}(\R^d)}.
\end{align}
\end{remark}
\begin{proof}[Proof of Proposition \ref{prop:f+}]
Again, we only consider the estimates on the ``outgoing" part, since the ``incoming" part can be treated in the same way.
Moreover, by choosing $s<s_0$, the case of $(q,r)=(\infty,2)$ follows from Corollary \ref{cor:f+-Hs}. Hence, we only consider the case when $r>2$. 
Let $N=2^{k_0}$ for some $k_0\in \N$. 

First, we recall the reductions given at the beginning of Subsection \ref{sec:In/out-linearflow}. We write
\begin{align*}
e^{it\Delta}\big(P_{\ge N} &\chi_{\ge 1}f\big)_{out}
=\sum\limits_{k=k_0}^\infty e^{it\Delta} \left(P_{2^k}\big(\chi_{\ge 1}f\big)\right)_{out}\\
=&\sum\limits_{k=k_0}^\infty e^{it\Delta}\Big(\chi_{\le \frac14}\left(P_{2^k}\big(\chi_{\ge 1}f\big)\right)_{out,k-1\le \cdot \le k+1}+h_k\Big)\\
&\qquad +\sum\limits_{k=k_0}^\infty e^{it\Delta}\Big( \chi_{\ge \frac14}\left(P_{2^k}\big(\chi_{\ge 1}f\big)\right)_{out,k-1\le \cdot \le k+1}\Big).
\end{align*}

Let $(s,q,r)$ be the triple satisfying 
$$
s\ge 0,\quad \frac2q+\frac {2d-1}r< \frac{2d-1}{2},\quad s+\frac d2-\left(\frac2q+\frac dr\right)\le \mu(d).
$$
In particular,  $(s,q,r)=(2,2,\frac{2d}{d-2})$ verifies the condition above. 
Then we use Lemma \ref{lem:radial-Str}, \eqref{est:remainer1} and \eqref{est:remainer2}, to get
\begin{align}
\sum\limits_{k=k_0}^\infty \Big\|e^{it\Delta}&|\nabla|^{s}\Big(\chi_{\le \frac14}\left(P_{2^k}\big(\chi_{\ge 1}f\big)\right)_{out,k-1\le \cdot \le k+1}+h_k\Big)\Big\|_{L^{q}_tL^{r}_x(\R^{+}\times\R^d)}\notag\\
\lesssim &
\sum\limits_{k=k_0}^\infty\left( \big\|\chi_{\le \frac14}\left(P_{2^k}\big(\chi_{\ge 1}f\big)\right)_{out,k-1\le \cdot \le k+1}\big\|_{H^{\mu(d)}(\R^d)}+\big\|h_k\big\|_{H^{\mu(d)}(\R^d)}\right)\notag\\
 &\quad \lesssim
\sum\limits_{k=k_0}^\infty 2^{-2k}\big\|P_{2^k}\big(\chi_{\ge 1}f\big)\big\|_{L^2(\R^d)}
\lesssim \big\|P_{\ge N}\chi_{\ge 1}f\big\|_{H^{-1}(\R^d)}.\label{15.09}
\end{align}


Furthermore, we write
\begin{align*}
\sum\limits_{k=k_0}^\infty &e^{it\Delta}\Big( \chi_{\ge \frac14}\left(P_{2^k}\big(\chi_{\ge 1}f\big)\right)_{out,k-1\le \cdot \le k+1}\Big)\\
= &
\sum\limits_{k=k_0}^\infty \chi_{\le \delta(1+2^k t)}e^{it\Delta}\Big(\chi_{\ge \frac14}\left(P_{2^k}\big(\chi_{\ge 1}f\big)\right)_{out,k-1\le \cdot \le k+1}\Big)\\
&\qquad +\sum\limits_{k=k_0}^\infty \chi_{\ge \delta(1+2^k t)}e^{it\Delta}\Big(\chi_{\ge \frac14}\left(P_{2^k}\big(\chi_{\ge 1}f\big)\right)_{out,k-1\le \cdot \le k+1}\Big).
\end{align*}
On one hand, from Proposition \ref{prop:pre-outgoing}, we have that for any $(s,q,r)$ satisfying \eqref{paremater-condition} (note that the condition \eqref{paremater-condition} is implied by the conditions in this proposition), 
\begin{align*}
\sum\limits_{k=k_0}^\infty & \Big\||\nabla|^{s}\Big[\chi_{\le \delta(1+2^k t)}e^{it\Delta}\Big(\chi_{\ge \frac14}\left(P_{2^k}\big(\chi_{\ge 1}f\big)\right)_{out,k-1\le \cdot \le k+1}\Big)\Big]\Big\|_{L^{q}_tL^{r}_x(\R^{+}\times\R^d)}\\
\lesssim &
\sum\limits_{k=k_0}^\infty 2^{-\big(5-d-(s-\frac 2{q}-\frac d{r})\big)k}
\big\|P_{2^k}\big(\chi_{\ge 1}f\big)\big\|_{L^2(\R^d)}
\lesssim N^{-\gamma_0+}\|P_{\ge N}\chi_{\ge 1}f\|_{H^{s_0}(\R^d)}.
\end{align*}

On the other hand, from Corollary \ref{cor:outgoing}, we obtain that
\begin{align*}
\sum\limits_{k=k_0}^\infty & \Big\||\nabla|^{s}\Big[\chi_{\ge \delta(1+2^k t)}e^{it\Delta}\Big(\chi_{\ge \frac14}\left(P_{2^k}\big(\chi_{\ge 1}f\big)\right)_{out,k-1\le \cdot \le k+1}\Big)\Big]\Big\|_{L^{q}_tL^{r}_x(\R^{+}\times\R^d)}\\
\lesssim &
\sum\limits_{k=k_0}^\infty 2^{(-\frac1{q}-\frac1{r}+\frac12+s+) k}
\big\|P_{2^k}\big(\chi_{\ge 1}f\big)\big\|_{L^2(\R^d)}
\lesssim N^{-\gamma_0+}\|P_{\ge N}\chi_{\ge 1}f\|_{H^{s_0}(\R^d)};
\end{align*}
and similarly, 
\begin{align*}
\sum\limits_{k=k_0}^\infty & \Big\|\nabla\Big[\chi_{\ge \delta(1+2^k t)}e^{it\Delta}\Big(\chi_{\ge \frac14}\left(P_{2^k}\big(\chi_{\ge 1}f\big)\right)_{out,k-1\le \cdot \le k+1}\Big)\Big]\Big\|_{L^{2}_t L^\infty_x([\delta,+\infty)\times\R^d)}\\
&\lesssim N^{1-s_0-\frac{d-2}2+}\|P_{\ge N}\chi_{\ge 1}f\|_{H^{s_0}(\R^d)}.
\end{align*}
Then collecting the estimates above, we give the desired estimate and thus complete the proof of the proposition.
\end{proof}


\section{Proof of the Theorem \ref{thm:main1}}
\label{sec:proof_of_THM1}

In this section, we are ready to prove Theorem \ref{thm:main1}. To do this, we need the following preliminary.

\subsection{Definitions of the modified incoming and outgoing components}
With the preparations in the previous sections, we can define the modified incoming and outgoing components, say $f_+$ and $f_-$, of the function $f$.
First of all, we split the function $f$ as follows,
$$
f=P_{\le N}f+P_{\ge N}\chi_{\le 1}f+P_{\ge N}\chi_{\ge 1}f.
$$
\begin{definition}\label{def:+--I}
Let the radial function $f\in \mathcal S  (\R^d)$.
We define the modified outgoing component of $f$ as
$$
f_+=\frac12P_{\le N}f+\frac12P_{\ge N}\chi_{\le 1}f+\big(P_{\ge N}\chi_{\ge 1}f\big)_{out};
$$
the modified incoming component of $f$ as
$$
f_-=\frac12P_{\le N}f+\frac12P_{\ge N}\chi_{\le 1}f+\big(P_{\ge N}\chi_{\ge 1}f\big)_{in}.
$$
\end{definition}
From the definitions, we have
$$
f=f_++f_-.
$$
Moreover, Definition \ref{def:+--I} combining with Proposition \ref{prop:f+} gives the proof of Proposition \ref{prop:out}.

To prove Theorem \ref{thm:main1}, we only consider the outgoing part, since the incoming part can be  treated in the same way. That is, initial data  $u_0=f_+$.



Denote $v_L=e^{it\Delta}\big(P_{\ge N}\chi_{\ge 1}f \big)_{out}$. Let $w=u-v_L$, then $w$ obeys the equation
 \begin{equation}\label{13.24}
   \left\{ \aligned
    &i\partial_{t}w+\Delta w=|u|^pu,
    \\
    &w(0,x)  =w_0(x),
   \endaligned
  \right.
 \end{equation}
 where
 $$
 w_0=\frac12P_{\le N}f+\frac12P_{\ge N}\chi_{\le 1}f.
 $$
Then by Bernstein's inequality, and \eqref{condition:thm}  (in which $\varepsilon_0=1$), we have
\begin{align}
\big\|w_0\big\|_{H^1(\R^d)}\lesssim N^{1-s_0}, \label{w0}
\end{align}
where the implicit constant depends on $\|\chi_{\le 1}f\|_{H^1(\R^d)}$ and $\|\chi_{\ge 1}f\|_{H^{s_0}(\R^d)}$.

The first result is the following local well-posedness theorem, in which the indices are not sharp but enough for this paper.
\begin{prop} Let $d=3,4,5$. Then there exists $p_1(d)<\frac{4}{d-2}$, such that for any $s_0\ge \frac12$ and any $p\in [p_1,\frac{4}{d-2})$, the following is true. Let $f$ be the function under the same hypothesis on Theorem \ref{thm:main1} with $\varepsilon_0=1$ and $w_0\in H^1(\R^d)$, then there exists $T_{max}>0$, such that the Cauchy problem  \eqref{13.24} is locally well-posed on $[0,T_{max})$.
Moreover, the blowup criterion holds: If $T_{max}<\infty$, then
$$
\lim\limits_{t\to T_{max}}\|w(t)\|_{H^1(\R^d)}=+\infty.
$$
\end{prop}
\begin{proof}
It follows from the standard fixed point argument, and thus the proof is much brief. According to the Duhamel formula, we denote
$$
\Phi(w)(t)=e^{it\Delta}w_0+\int_0^t e^{i(t-s)\Delta} |u(s)|^pu(s)\,ds.
$$
For short, we further denote the parameters
$$
\frac1{\tilde r}=\frac 12-\frac{3d-2}{d(2d-1)}-\frac{\varepsilon} d;\quad \tilde s =1-\frac{d-1}{2d-1}-\varepsilon,
$$
where $\varepsilon$ is an arbitrary small positive constant.
Fixing $T>0$, we denote the norm
$$
\big|\!\big|\!\big|w\big|\!\big|\!\big|:=\|w\|_{L^\infty_tH^1_x([0,T)\times \R^d)}+\big\|\langle \nabla\rangle^{\tilde s} w\big\|_{L^2_tL^{\tilde r}_x([0,T)\times \R^d)}.
$$

From Lemma \ref{lem:radial-Str} and Lemma \ref{lem:Frac_Leibniz}, we have
\begin{align}
\big|\!\big|\!\big|\Phi(w)\big|\!\big|\!\big|\lesssim & \|w_0\|_{H^1_x(\R^d)}+\big\|\langle \nabla\rangle^{\tilde s} \big(|u|^pu\big)\big\|_{L^2_tL^{\frac{2d}{d+2\tilde s}}_x([0,T)\times \R^d)}\notag\\
\lesssim &
\|w_0\|_{H^1_x(\R^d)}+\big\|\langle \nabla\rangle^{\tilde s}v_L\big\|_{L^2_tL^\infty_x([0,T)\times \R^d)}\left(\|v_L\|_{L^\infty_tL^{\frac{2dp}{d+2\tilde s}}_x([0,T)\times \R^d)}^p+\|w\|_{L^\infty_tL^{\frac{2dp}{d+2\tilde s}}_x([0,T)\times \R^d)}^p\right)\notag\\
&
\quad+\big\|\langle \nabla\rangle^{\tilde s}w\big\|_{L^\infty_tL^{r_1}_x([0,T)\times \R^d)}\left(\|v_L\|_{L^{2p}_tL^{dp}_x([0,T)\times \R^d)}^p+\|w\|_{L^{2p}_tL^{dp}_x([0,T)\times \R^d)}^p\right),\label{18.16}
\end{align}
where $r_1$ is the parameters satisfying
$$
 \frac1{r_1}=\frac12-\frac1d(1-\tilde s).
$$
Note that when $p_1(d)$ is suitably close to $\frac4{d-2}$, we have $2<r_1, \frac{2dp}{d+2\tilde s}<\frac{2d}{d-2}$.
From \eqref{L2Lr}--\eqref{L2pLdp}, we have that for any $s_0\ge \frac12$, 
\begin{align*}
\big\|\langle \nabla\rangle^{\tilde s}v_L\big\|_{L^2_tL^{\infty}_x(\R^+\times \R^d)}+
\|v_L\|_{L^\infty_tL^{\frac{2dp}{d+2\tilde s}}_x(\R^+\times \R^d)}
+\|v_L\|_{L^{2p}_tL^{dp}_x(\R^+\times \R^d)}\lesssim \|f\|_{H^{s_0}(\R^d)},
\end{align*}
Hence, we get 
\begin{align*}
\big\|\langle \nabla\rangle^{\tilde s}v_L&\big\|_{L^2_tL^\infty_x([0,T)\times \R^d)}+\|v_L\|_{L^\infty_tL^{\frac{2dp}{d+2\tilde s}}_x([0,T)\times \R^d)}
+\|v_L\|_{L^{2p}_tL^{dp}_x([0,T)\times \R^d)}=r(T),
\end{align*}
where $r(T)\to 0$ when $T\to 0$. Further, by the Sobolev inequality and interpolation, we have
\begin{align*}
\|w\|_{L^\infty_tL^\frac{2dp}{d+2\tilde s}_x([0,T)\times \R^d)}+&\|\langle \nabla\rangle^{\tilde s}w\|_{L^\infty_tL^{r_1}_x([0,T)\times \R^d)}\lesssim \|w\|_{L^\infty_tH^1_x([0,T)\times \R^d)};\\
&
\|w\|_{L^{2p}_tL^{dp}_x([0,T)\times \R^d)}\lesssim T^\theta \big|\!\big|\!\big|w\big|\!\big|\!\big|,
\end{align*}
where $\theta=\frac12(1-s_c)>0$. Hence, inserting these estimates into \eqref{18.16}, we obtain that
\begin{align*}
\big|\!\big|\!\big|\Phi(w)\big|\!\big|\!\big|\lesssim & \|w_0\|_{H^1_x(\R^d)}+
r(T)\big(1+\big|\!\big|\!\big|w\big|\!\big|\!\big|^{p+1}\big).
\end{align*}
Similarly, by choosing $T=T(\|f\|_{H^{s_0}(\R^d)},\|w_0\|_{H^1(\R^d)})>0$ small enough, we have that for any $w_j$ satisfying $\big|\!\big|\!\big|w_j\big|\!\big|\!\big|\le 2\|w_0\|_{H^1(\R^d)}$ for $j=1,2$,
$$
\big|\!\big|\!\big|\Phi(w_1)-\Phi(w_2)\big|\!\big|\!\big|\le   \frac12 \big|\!\big|\!\big|w_1-w_2\big|\!\big|\!\big|.
$$
Then the proposition follows from the fixed point theory.
\end{proof}

In the following subsections, we will prove the boundedness of $w$ in the energy space and some space-time spaces. To this end, we define the working space as follows.
We denote $X_N(I)$ for $I\subset \R^+$ to be the space under the norms
 \begin{align*}
\|h\|_{X_N(I)}:=&N^{3(s_0-1)}\|h\|_{L^\infty_t \dot H^1_x(I\times \R^d)}+N^{\alpha(d,p)\cdot(s_0- 1)}\|h\|_{L^{r_0}_{tx}(I\times \R^d)},
\end{align*}
where the positive constants
\begin{align*}
 r_0=p+2+\frac{2}{d-2},
 \quad \mbox{and }\quad
\alpha(d,p)=\frac{3d}{r_0(d-2)}.
\end{align*}
In the following, we restrict our attention to 3,4 and 5 dimensions.
Moreover, fixing $\delta_0>0$, we set $N=N(\delta_0)>0$, such that
\begin{align}
\|P_{\ge N}\chi_{\ge 1}f \|_{H^{s_0}(\R^d)}\le \delta_0.
\label{N0}
\end{align}
First of all, we have the uniform boundedness of the $L^\infty_t L^2_x(I\times \R^d)$ norm of $w$. Indeed, choosing $s_0> 0$, we have $u_0\in L^2(\R^d)$. Using the mass conservation law, we have
$$
\|u\|_{L^\infty_t L^2_x(I\times \R^d)}=\|u_0\|_{L^2(\R^d)}.
$$
Moreover, by Proposition \ref{prop:bound-fpm-L2},
$$
\|v_L\|_{L^\infty_t L^2_x(I\times \R^d)}
\lesssim
 \big\|f\big\|_{H^{s_0}(\R^d)}.
$$
Hence, we obtain that
\begin{align}
\|w\|_{L^\infty_t L^2_x(I\times \R^d)}\lesssim \big\|f\big\|_{H^{s_0}(\R^d)}.\label{w-L2}
\end{align}

\subsection{Space-time estimates}
In this subsection,  based on the $L^{r_0}_{tx}$-norm in $\|w\|_{X_N(I)}$. we shall prove the general space-time estimates we will use below.
\begin{lem}\label{lem:L-xt} 
Let $p_1(d)\ge\max\{\frac3{d-2},1\}$, then there exists $s_0\in (0,1)$, such that for any  $p\in [p_1(d), \frac{4}{d-2})$, there exist some constants $\alpha_1(d,p)>0,\alpha_2(d,p)>1$ such that 
\begin{align*}
\big\|w\big\|_{L^2_tL^{\frac{2d}{d-2}}_x(I\times \R^d)}
\lesssim 1+ N^{\alpha_1(d,p)\cdot (1-s_0)}\big\|w\big\|_{X_N(I)}^{\alpha_2(d,p)}.
\end{align*}
\end{lem}
\begin{proof}
By Lemma \ref{lem:strichartz}, we have
\begin{align}
\big\| w\big\|_{L^2_tL^{\frac{2d}{d-2}}_x(I\times \R^d)}\lesssim \big\| w_0\big\|_{L^2_x(\R^d)}
+\big\||u|^pu\big\|_{L^2_tL^{\frac{2d}{d+2}}_x(I\times \R^d)}.\label{3.10}
\end{align}
From \eqref{w0}, we have
\begin{align}
\big\| w_0\big\|_{L^2_x(\R^d)}\lesssim 1.
\label{3.11}
\end{align}
Moreover, using $u=w+v_L$, we have
\begin{align*}
\big\||u|^pu\big\|_{L^2_tL^{\frac{2d}{d+2}}_x(I\times \R^d)}
 &\lesssim
 \big\||v_L|^pv_L\big\|_{L^2_tL^{\frac{2d}{d+2}}_x(I\times \R^d)}+ \big\||w|^pw\big\|_{L^2_tL^{\frac{2d}{d+2}}_x(I\times \R^d)}.
 \end{align*}
 For the first term, we get 
 \begin{align*}
 \big\||v_L|^pv_L\big\|_{L^2_tL^{\frac{2d}{d+2}}_x(I\times \R^d)}
 \lesssim 
  \big\| v_L\big\|_{L^2_tL^{\frac{2d}{d-2}}_x(I\times \R^d)}\big\|v_L\big\|_{L^\infty_tL^{\frac{dp}{2}}_x(I\times \R^d)}^p.
  \end{align*}
  Since $\frac{dp}2>2$, then
by \eqref{L2Lr}--\eqref{L2pLdp} and \eqref{N0}, we have
$$
\big\| v_L\big\|_{L^2_tL^{\frac{2d}{d-2}}_x(I\times \R^d)}\lesssim  \delta_0 N^{\frac{d+2}{2d}-s_0+},
$$
and
$$
\big\|v_L\big\|_{L^\infty_tL^{\frac{dp}{2}}_x(I\times \R^d)}\lesssim \delta_0 N^{\frac12-\frac{2}{dp}-s_0+}.
$$
Choosing $1-s_0$ small enough such that for any $p\in[1,4]$, 
$$
-\frac{d+2}{2d}-\frac p2+(1-s_0)(p+1)<0,
$$
we have that 
 \begin{align*}
 \big\||v_L|^pv_L\big\|_{L^2_tL^{\frac{2d}{d+2}}_x(I\times \R^d)}
 \lesssim 1.
  \end{align*}
   For the second term, when $d=3$, we get 
\begin{align*}  
 \big\||w|^pw\big\|_{L^2_tL^{\frac{6}{5}}_x(I\times \R^3)}
\lesssim & 
\big\|w\big\|_{L^{r_0}_{tx}(I\times \R^3)}^{\frac{r_0}{2}}\big\|w\big\|_{L^\infty_tL^{r_2}_x(I\times \R^3)}^{p+1-\frac{r_0}2},
\end{align*}
where  the parameter  
$$
r_2=3(p+1-\frac{r_0}{2}).
$$
Since $2<r_2<6$, we further get 
\begin{align*}  
 \big\||w|^pw\big\|_{L^2_tL^{\frac{6}5}_x(I\times \R^3)}
\lesssim & 
\big\|w\big\|_{L^{r_0}_{tx}(I\times \R^3)}^{\frac{r_0}{2}}\big\|w\big\|_{L^\infty_tH^1_x(I\times \R^3)}^{p+1-\frac{r_0}2}\\
\lesssim &
N^{[\alpha(3,p)\frac{r_0}{2}+3(p+1-\frac{r_0}{2})](1-s_0)}\big\|w\big\|_{X_N(I)}^{p+1}.
\end{align*}
Setting 
$$
\alpha_1(3,p)=\alpha(3,p)\frac{r_0}{2}+3(p+1-\frac{r_0}{2}), \quad \alpha_2(3,p)=p+1,
$$
we obtain that 
\begin{align*}
\big\||w|^pw\big\|_{L^2_tL^{\frac{6}5}_x(I\times \R^3)}
 &\lesssim
1+ N^{\alpha_1(3,p)\cdot (1-s_0)}\big\|w\big\|_{X_N(I)}^{\alpha_2(3,p)}.
 \end{align*}
Therefore,  we get
\begin{align*}
\big\| w\big\|_{L^2_tL^6_x(I\times \R^3)}
\lesssim 1+ N^{\alpha_1(3,p)\cdot (1-s_0)}\big\|w\big\|_{X_N(I)}^{\alpha_2(3,p)},
\end{align*}
and thus we get the desired estimate.

 When $d=4$, we have that
 \begin{align*}  
 \big\||w|^pw\big\|_{L^2_tL^{\frac{4}{3}}_x(I\times \R^4)}
\lesssim & 
\big\|w\big\|_{L^2_{t}L^4_x(I\times \R^4)}^{a_1}\big\|w\big\|_{L^{r_0}_{tx}(I\times \R^4)}^{a_2}\big\|w\big\|_{L^\infty_tL^{2}_x(I\times \R^4)}^{a_3}\\
\lesssim & 
N^{\alpha(d,p)(1-s_0)\cdot a_2}\big\|w\big\|_{L^2_{t}L^4_x(I\times \R^4)}^{a_1}\big\|w\big\|_{X_N(I)}^{a_2},
\end{align*}
 where the parameters 
 $$
 a_1=\frac{r_0-2p-1}{r_0-3},\quad a_2=\frac{r_0(p-1)}{r_0-3},\quad a_3=p+1-a_1-a_2.
 $$
 Note that $0<a_1<\frac13$ when $p\ge \frac32$. Hence, setting 
$$
\alpha_1(4,p)=\alpha(4,p)a_2(1-a_1)^{-1}, \quad \alpha_2(4,p)=a_2(1-a_1)^{-1},
$$
and by the Cauchy-Schwarz inequality, we get
\begin{align*}
\big\| w\big\|_{L^2_tL^4_x(I\times \R^4)}
\lesssim 1+ N^{\alpha_1(4,p)\cdot (1-s_0)}\big\|w\big\|_{X_N(I)}^{\alpha_2(4,p)},
\end{align*}
 
 When $d=5$, similarly we have that
 \begin{align*}  
 \big\||w|^pw\big\|_{L^2_tL^{\frac{10}{7}}_x(I\times \R^5)}
\lesssim & 
\big\|w\big\|_{L^2_{t}L^{\frac{10}{3}}_x(I\times \R^5)}^{a_1}\big\|w\big\|_{L^{r_0}_{tx}(I\times \R^5)}^{a_2}\big\|w\big\|_{L^\infty_tL^{2}_x(I\times \R^5)}^{a_3}.
\end{align*}
 Here the parameters 
 $$
 a_1=\frac{62-15r_0}{15r_0-42},\quad a_2=\frac{r_0(15r_0-52)}{15r_0-42},\quad a_3=p+1-a_1-a_2.
 $$
 Note that $\frac19<a_1<\frac34$ when $p\ge 1$. Setting 
$$
\alpha_1(5,p)=\alpha(5,p)a_2(1-a_1)^{-1}, \quad \alpha_2(5,p)=a_2(1-a_1)^{-1},
$$
and by the Cauchy-Schwarz inequality, we get
\begin{align*}
\big\| w\big\|_{L^2_tL^{\frac{10}{3}}_x(I\times \R^5)}
\lesssim 1+ N^{\alpha_1(5,p)\cdot (1-s_0)}\big\|w\big\|_{X_N(I)}^{\alpha_2(5,p)},
\end{align*}
and thus we get the desired estimate.
\end{proof}
\begin{remark}
Taking $p_1(d)\ge 1$, then we can roughly estimate that  $\alpha_1(d,p), \alpha_2(d,p)\le 20$. Hence,  $\alpha_1(d,p),\alpha_2(d,p)$ have a uniform upper bound  when $p_1(d)$ is close to $\frac{4}{d-2}$.  Based on these facts, the parameters $\alpha_3(d,p),\alpha_4(d,p)$ defined later also  have a uniform upper bound  when $p_1(d)$ is close to $\frac{4}{d-2}$. 
\end{remark}

\subsection{Morawetz estimates}
 In this subsection, we consider the Morawetz-type estimate of Lin-Strauss \cite{Lin-Strauss}, see also \cite{CKSTT04} for the interaction Morawetz estimates. For convenience, we rewrite the equation of $w$ in the following way,
 $$
 i\partial_{t}w+\Delta w=|w|^pw+F(v_L,w),
 $$
 where $F(v_L,w)= |v_L+w|^p(v_L+w)-|w|^pw$.
 Let
 $$
 M(t)=\mbox{Im} \int_{\R^d}\frac{x}{|x|}\cdot \nabla w(t,x)\bar w(t,x)\,dx.
 $$
Then we have the following lemma.
\begin{lem}\label{lem:L-p+2}  Under the same assumption as in Lemma \ref{lem:L-xt},  for any time interval $I$ such that  $0\in I\subset \R^+$,
\begin{align*}
\int_I\int_{\R^d} \frac{|w(t,x)|^{p+2}}{|x|}\,dx dt
\lesssim N^{3(1-s_0)}\Big(\|w\|_{X_N(I)}+\delta_0 \|w\|_{X_N(I)}^{1+p\alpha_2(d,p)}\Big).
\end{align*}
\end{lem}
\begin{proof}
Note that
\begin{align}
M'(t)=\mbox{Im} \int \Big(2\frac{x}{|x|}\cdot\nabla w+\frac{d-1}{|x|}w\Big)\overline{w_t}\,dx.\label{18.12}
\end{align}
Since
$$
\overline{w_t}=-i\Delta \bar w+i|w|^p\bar w+i\overline{F(v_L,w)},
$$
we shall consider the following three terms,
\begin{subequations}
\begin{align}
\mbox{Im} \int \Big(2\frac{x}{|x|}\cdot\nabla w+\frac{d-1}{|x|}w\Big)(-i\Delta \bar w)\,dx;\label{16.10}
\end{align}
\begin{align}
\mbox{Im} \int \Big(2\frac{x}{|x|}\cdot\nabla w+\frac{d-1}{|x|}w\Big)(i|w|^p\bar w)\,dx;\label{16.11}
\end{align}
and
\begin{align}
\mbox{Im} \int \Big(2\frac{x}{|x|}\cdot\nabla w+\frac{d-1}{|x|}w\Big)\left(i\overline{F(v_L,w)}\right)\,dx.\label{16.12}
\end{align}
\end{subequations}
By a direct computation, we have
\begin{align}
\eqref{16.10}\ge 0. \label{est:16.10}
\end{align}
Indeed, by integration-by-parts,
\begin{align*}
\eqref{16.10}=\int \frac1{|x|}\left(|\nabla w|^2-\left|\frac x{|x|}\cdot \nabla w\right|^2\right)\,dx+G(w),
\end{align*}
where
\begin{align*}
G(w)=\left\{\aligned
    &2\pi |w(t,0)|^2,\quad \mbox{if } d=3,
    \\
    &\frac{(d-1)(d-3)}{4}\int \frac{|w(t,x)|^2}{|x|^3}\,dx,\quad \mbox{if } d>3.
   \endaligned
  \right.
\end{align*}
Hence, we obtain \eqref{est:16.10}. Using integration-by-parts again, we find
\begin{align}
\eqref{16.11}=\frac{(d-1)p}{2(p+1)}\int \frac{|w(t,x)|^{p+2}}{|x|}\,dx. \label{est:16.11}
\end{align}
For \eqref{16.12}, using the H\"older inequality and the Hardy inequality, we have
\begin{align*}
|\eqref{16.12}|
\le  &
\Big|\int \Big(2\frac{x}{|x|}\cdot\nabla w+\frac{d-1}{|x|}w\Big)\left(i\overline{F(v_L,w)}\right)\,dx\Big|\\
\lesssim &
\|F(v_L,w)\|_{L^2_x(\R^d)}\Big(\big\|\nabla w\big\|_{L^2_x(\R^d)}+\Big\|\frac w{|x|}\Big\|_{L^2_x(\R^d)}\Big)\\
\lesssim &
\big\|\nabla w\big\|_{L^2_x(\R^d)}\|F(v_L,w)\|_{L^2_x(\R^d)}.
\end{align*}
Hence, this last estimate  combining \eqref{est:16.10} and \eqref{est:16.11}, and integrating in time in \eqref{18.12}, we obtain that
\begin{align}
\int_I\int_{\R^d} \frac{|w(t,x)|^{p+2}}{|x|}\,dx dt
\lesssim &
\max\limits_{t\in I} M(t)+ \big\|\nabla w\big\|_{L^\infty_tL^2_x(I\times\R^d)}\|F(v_L,w)\|_{L^1_tL^2_x(I\times\R^d)}. \label{19.45}
\end{align}
By the H\"older inequality and \eqref{w-L2}, we have 
\begin{align}
\max\limits_{t\in I} M(t)\lesssim N^{3(1-s_0)}\|w\|_{X_N(I)}.\label{18.20}
\end{align}
Now we claim that by choosing $s_0$ close enough to 1,
\begin{align}
\|F(v_L,w)\|_{L^1_tL^2_x(I\times\R^d)}
\lesssim 1+\delta_0 \|w\|_{X_N(I)}^{p\alpha_2(d,p)}.\label{18.40}
\end{align}
Indeed,  we have
\begin{align*}
\|F(v_L,w)&\|_{L^1_tL^2_x(I\times\R^d)}
\lesssim 
\big\||v_L|^{p+1}\big\|_{L^1_tL^2_x(I\times\R^d)}+\big\||v_L||w|^{p}\big\|_{L^1_tL^2_x(I\times\R^d)}
\end{align*}
For the first term, by \eqref{est:Lqr_out_in}, we have that 
\begin{align*}
\big\||v_L|^{p+1}\big\|_{L^1_tL^2_x(I\times\R^d)}
\lesssim 
\big\|v_L\big\|_{L^{p+1}_tL^{2(p+1)}_x(I\times\R^d)}^{p+1}\lesssim1.
\end{align*}
For the second term, when $d=3$, noting that $r_0>2p$, we have that 
\begin{align*}
\big\||v_L||w|^{p}\big\|_{L^1_tL^2_x(I\times\R^3)}
\lesssim &
\|v_L\|_{L^2_tL^\infty_x(I\times\R^3)}\|w\|_{L^{r_0}_{tx}(I\times\R^3)}^{\frac{(p-2)r_0}{r_0-4}}\|w\|_{L^4_{tx}(I\times\R^3)}^{\frac{2r_0-4p}{r_0-4}}.
\end{align*}
Note that by the interpolation and Lemma \ref{lem:L-xt}, 
\begin{align}\label{w-L4}
\|w\|_{L^4_{tx}(I\times\R^3)}\lesssim 
\|w\|_{L^\infty_tL^3_x(I\times\R^3)}^\frac12\|w\|_{L^2_tL^6_x(I\times\R^3)}^\frac12
\lesssim 
1+ N^{[\frac32+\frac12\alpha_1(3,p)](1-s_0)}\big\|w\big\|_{X_N(I)}^{\alpha_2(3,p)}.
\end{align}
By \eqref{L2Lr}, we get 
\begin{align*}
\|v_L\|_{L^2_tL^\infty_x(I\times\R^3)}
\lesssim \delta_0N^{-s_0+}.
\end{align*}
Therefore, 
\begin{align*}
\big\||v_L||w|^{p}\big\|_{L^1_tL^2_x(I\times\R^3)}
\lesssim &
\delta_0N^{-s_0+}
\left(1+ N^{\big(\frac{(p-2)r_0}{r_0-4}\alpha(3,p)+\frac{r_0-2p}{r_0-4}[3+\alpha_1(3,p)]\big)(1-s_0)}\right)\|w\|_{X_N(I)}^{p\alpha_2(3,p)}.
\end{align*}
Choosing $1-s_0$ small enough such that for any $p\in [2,4]$, 
$$
(1-s_0)\Big(1+\frac{(p-2)r_0}{r_0-4}\alpha(3,p)+\frac{r_0-2p}{r_0-4}\big[3+\alpha_1(3,p)\big]\Big)<1,
$$
 we obtain \eqref{18.40}. When $d=4,5$, since $1\le p<2$, by \eqref{est:Lqr_out_in} and Lemma \ref{lem:L-xt} we have that 
\begin{align*}
\big\||v_L||w|^{p}\big\|_{L^1_tL^2_x(I\times\R^d)}
\lesssim &
\|v_L\|_{L^{\frac2{2-p}}_tL^{\frac{2d}{d-(d-2)p}}_x(I\times\R^d)}\|w\|_{L^2_{t}L^{\frac{2d}{d-2}}_x(I\times\R^d)}^p\\
\lesssim &
\delta_0\big(1+ N^{-s_0+\frac12+p\alpha_1(d,p)\cdot (1-s_0)}\big)\|w\|_{X_N(I)}^{p\alpha_2(d,p)}.
\end{align*} 
Choosing $1-s_0$ small enough such that for any $p\in [2,4]$, 
$$
(1-s_0)\big(1+p\alpha_1(d,p)\big)<\frac12,
$$
 we obtain \eqref{18.40} again.

Therefore, together with \eqref{19.45}, \eqref{18.20} and \eqref{18.40}, we get
\begin{align*}
\int_I\int_{\R^d} \frac{|w(t,x)|^{p+2}}{|x|}\,dx dt
\lesssim N^{3(1-s_0)}\Big(\|w\|_{X_N(I)}+\delta_0 \|w\|_{X_N(I)}^{1+p\alpha_2(d,p)}\Big).
\end{align*}
This finishes the proof of the lemma.
\end{proof}

The following is a consequence of the previous lemma.
\begin{cor}\label{cor:L-2p+2}
Under the same assumptions as in Lemma \ref{lem:L-xt}, there exists $\alpha_3(d,p)>0$ such that 
\begin{align*}
\|w\|_{L^{r_0}_{tx}(I\times \R^d)}
\lesssim N^{\alpha(d,p)\cdot(1- s_0)}\Big(\|w\|_{X_N(I)}^{\frac d{r_0(d-2)}}+\delta_0^{\frac1{r_0}} \|w\|_{X_N(I)}^{\alpha_3(d,p)}\Big).
\end{align*}
\end{cor}
\begin{proof}
By the H\"older inequality,
\begin{align*}
\int_I\int_{\R^d}|w(t,x)|^{r_0}\,dxdt
\lesssim
\big\||x|^{\frac{d-2}{2}}w\big\|_{L^\infty_{tx}(I\times\R^d)}^\frac{2}{d-2}\int_I\int_{\R^d}\frac{|w(t,x)|^{p+2}}{|x|}\,dxdt.
\end{align*}
From Lemma \ref{lem:radial-Sob}, we have
$$
\big\||x|^{\frac{d-2}{2}}w\big\|_{L^\infty_{tx}(I\times \R^d)}
\lesssim \|\nabla w\|_{L^\infty_tL^2_x(I\times\R^d)}
\lesssim N^{3(1-s_0)}\|w\|_{X_N(I)}.
$$
From Lemma \ref{lem:L-p+2}, we have
$$
\int_I\int_{\R^d}\frac{|w(t,x)|^{p+2}}{|x|}\,dxdt
\lesssim N^{3(1-s_0)}\Big(\|w\|_{X_N(I)}+\delta_0 \|w\|_{X_N(I)}^{1+p\alpha_2(d,p)}\Big).
$$
Hence, we obtain
\begin{align*}
\int_I\int_{\R^d}|w(t,x)|^{r_0}\,dxdt
\lesssim & N^{\frac {3d}{d-2}(1-s_0)}\Big(\|w\|_{X_N(I)}^\frac d{d-2}+\delta_0 \|w\|_{X_N(I)}^{\frac d{d-2}+p\alpha_2(d,p)}\Big).
\end{align*}
Let
$$
\alpha_3(d,p)=\frac1{r_0}\left[\frac d{d-2}+p\alpha_2(d,p)\right],
$$
we have the desired estimate.
This finishes the proof of the corollary.
\end{proof}

\subsection{Energy estimate} 
In this subsection, we consider the energy estimate for $w$. Since the energy of $w$ is not conserved, the nonlinear estimates are needed in this subsection.
The main result in this subsection is the following  $\dot H^1(\R^d)$-norm bound for $w$. 
\begin{lem}\label{lem:H1} There exists $\alpha_4(d,p)>1$ such that for any $0\in I\subset \R^+$,
\begin{align}
\sup\limits_{t\in I}\big\|\nabla w(t)\big\|_{L^2(\R^d)}\lesssim N^{3(1-s_0)}\Big(1+\delta_0^\frac12\|w\|_{X_N(I)}^{\alpha_4(d,p)}\Big).
\end{align}
\end{lem}
\begin{proof}
Let $t\in I$. Denote that 
$$
\tilde E(t)=\frac12\big\|\nabla w(t)\big\|_{L^2(\R^d)}^2+\frac1{p+2}\big\| u(t)\big\|_{L^{p+2}(\R^d)}^{p+2}.
$$
Then taking product with $w_t$ on the equation \eqref{13.24} and integrating  in space and in time from $0$ to $t$, we get
\begin{align*}
\tilde E(t)
=\tilde E(0)-\mbox{Im}\int_0^t\!\!\int_{\R^d}\nabla\big(|u|^pu\big)\cdot \nabla \overline{v_L}\,dxdt'.
\end{align*}
Since $p\le 4$, by \eqref{w0}, \eqref{LinftyLr} and the Sobolev inequality,  it follows that 
$$
\tilde E(0)\lesssim N^{(p+2)(1-s_0)}\lesssim N^{6(1-s_0)}.
$$
From \eqref{condition:thm}, Proposition \ref{prop:f+} and Lemma \ref{lem:strichartz}, we have that
$$
\big\|\nabla e^{it\Delta}u_0\big\|_{L^2_tL^\frac{2d}{d-2}_x(\R^+\times \R^d)}\lesssim \|\chi_{\le 1}f\|_{H^1(\R^d)}+\|\chi_{\ge 1}f\|_{H^{s_0}(\R^d)}.
$$
Then it follows from the standard fixed point argument that  there exists $\delta>0$ depends only on $\|\chi_{\le 1}f\|_{H^1(\R^d)}$ and $\|\chi_{\ge 1}f\|_{H^{s_0}(\R^d)}$, such that 
\begin{align}
\big\|\nabla u\big\|_{L^2_tL^\frac{2d}{d-2}_x([0,\delta]\times \R^d)}\lesssim \|\chi_{\le 1}f\|_{H^1(\R^d)}+\|\chi_{\ge 1}f\|_{H^{s_0}(\R^d)}.\label{local-u}
\end{align}
Accordingly, we write 
$$
\mbox{Im}\int_0^t\!\!\int_{\R^d}\nabla\big(|u|^pu\big)\cdot \nabla \overline{v_L}\,dxdt'
=\mbox{Im}\int_0^\delta\!\!\int_{\R^d}\nabla\big(|u|^pu\big)\cdot \nabla \overline{v_L}\,dxdt'
+\mbox{Im}\int_\delta^t\!\!\int_{\R^d}\nabla\big(|u|^pu\big)\cdot \nabla \overline{v_L}\,dxdt'.
$$

For the first term, we have that 
\begin{align*}
\left|\mbox{Im}\int_0^\delta\!\!\int_{\R^d}\nabla\big(|u|^pu\big)\cdot \nabla \overline{v_L}\,dxdt'\right|
\lesssim 
\big\|\nabla u\big\|_{L^2_tL^{\frac{2d}{d-2}}_x([0,\delta]\times \R^d)} 
\big\|\nabla v_L\big\|_{L^2_tL^{\frac{2d}{d-2}}_x(I\times \R^d)} 
\big\|u\big\|_{L^\infty_tL^{\frac{dp}{2}}_x(I\times \R^d)}^p.
\end{align*}
By \eqref{L2Lr} and \eqref{LinftyLr}, we get 
\begin{align*}
\big\|\nabla v_L&\big\|_{L^2_tL^{\frac{2d}{d-2}}_x(I\times \R^d)}
\lesssim \delta_0N^{1-s_0-\frac{d-2}{2d}+};\quad 
 \big\|u\big\|_{L^\infty_{t}L^{\frac{dp}{2}}_x(I\times \R^d)}\lesssim 1+ N^{3(1-s_0)}\|w\|_{X_N(I)}.
\end{align*}
Hence, by further \eqref{local-u}, it infers that 
\begin{align*}
\left|\mbox{Im}\int_0^\delta\!\!\int_{\R^d}\nabla\big(|u|^pu\big)\cdot \nabla \overline{v_L}\,dxdt'\right|
\lesssim 
\delta_0N^{1-s_0-\frac{d-2}{2d}+}
\big(1+ N^{3p(1-s_0)}\|w\|_{X_N(I)}^p\big).
\end{align*}
Choosing $1-s_0$ small enough such that 
$$
(1-s_0)\big(3p-5\big)<\frac{d-2}{2d},
$$
we obtain that 
\begin{align*}
\left|\mbox{Im}\int_0^\delta\!\!\int_{\R^d}\nabla\big(|u|^pu\big)\cdot \nabla \overline{v_L}\,dxdt'\right|
\lesssim 
N^{6(1-s_0)}
\big(1+ \delta_0\|w\|_{X_N(I)}^p\big).
\end{align*}

For the second term, we have that 
\begin{align*}
\left|\mbox{Im}\int_\delta^t\!\!\int_{\R^d}\nabla\big(|u|^pu\big)\cdot \nabla \overline{v_L}\,dxdt'\right|
\lesssim 
\big\|\nabla u\big\|_{L^\infty_tL^2_x(I\times \R^d)} 
\big\|\nabla v_L\big\|_{L^2_tL^\infty_x([\delta,+\infty]\times \R^d)} 
\big\|u\big\|_{L^{2p}_{tx}(I\times \R^d)}^p.
\end{align*}
By \eqref{est:Lqr_out_in} and the interpolation, 
\begin{align*}
\big\|u\big\|_{L^{2p}_{tx}([0,\delta]\times \R^d)}^p
\lesssim & \|u\|_{L^{r_0}_{tx}(I\times\R^d)}^{\frac{r_0(3p-4)}{3r_0-8}}\|u\|_{L^\frac{8}{3}_{tx}(I\times\R^d)}^{\frac{4(r_0-2p)}{3r_0-8}}\\
\lesssim & \|u\|_{L^{r_0}_{tx}(I\times\R^d)}^{\frac{r_0(3p-4)}{3r_0-8}}\left(\|u\|_{L^\infty_tH^1_x(I\times\R^d)}+\|u\|_{L^2_tL^{\frac{2d}{d-2}}_x(I\times\R^d)}\right)^{\frac{4(r_0-2p)}{3r_0-8}}\\
\lesssim &
\left(1+ N^{p[\alpha(d,p)+\alpha_1(d,p)+3]\cdot(1-s_0)}\right)\|w\|_{X_N(I)}^{p\alpha_2(d,p)}.
\end{align*}
Further, by \eqref{est:Lqr_out_in-2},
$$
\big\|\nabla v_L\big\|_{L^2_tL^\infty_x([\delta,+\infty]\times \R^d)} \lesssim \delta_0N^{1-s_0-\frac{d-2}{2}+}.
$$
Hence we have that 
\begin{align*}
\Big|\mbox{Im}&\int_\delta^t\!\!\int_{\R^d}\nabla\big(|u|^pu\big)\cdot \nabla \overline{v_L}\,dxdt'\Big|\\
\lesssim &
\delta_0N^{1-s_0-\frac{d-2}2+}
\left(1+ N^{\big(3+p[\alpha(d,p)+\alpha_1(d,p)+3]\big)\cdot (1-s_0)}\right)\|w\|_{X_N(I)}^{1+p\alpha_2(d,p)}.
\end{align*}
Then choosing $1-s_0$ small enough such that 
$$
\big(p[\alpha(d,p)+\alpha_1(d,p)+3]-2\big)\cdot (1-s_0)<\frac{d-2}{2},
$$
we obtain that 
\begin{align*}
\Big|\mbox{Im}\int_\delta^t\!\!\int_{\R^d}\nabla\big(|u|^pu\big)\cdot \nabla \overline{v_L}\,dxdt'\Big|
\lesssim &
N^{6(1-s_0)}
\big(1+ \delta_0\|w\|_{X_N(I)}^{1+p\alpha_2(d,p)}\big).
\end{align*}

Setting 
$$
\alpha_4(d,p)=\frac12\max\{p, 1+p\alpha_2(d,p)\},
$$
we get that 
\begin{align*}
\Big|\mbox{Im}\int_0^t\!\!\int_{\R^d}\nabla\big(|u|^pu\big)\cdot \nabla \overline{v_L}\,dxdt'\Big|
\lesssim &
N^{6(1-s_0)}
\big(1+ \delta_0\|w\|_{X_N(I)}^{2\alpha_4(d,p)}\big),
\end{align*}
and thus 
\begin{align*}
\tilde E(t)
\lesssim&N^{6(1-s_0)}\Big(1+\delta_0\|w\|_{X_N(I)}^{2\alpha_4(d,p)}\Big).
\end{align*}
Then we obtain the desired estimates.
\end{proof}

\subsection{The proofs}
Now we are ready to prove Theorem \ref{thm:main1}.
First, we show that for any $I$ such that  $0\in I\subset \R^+$,
\begin{align}
\|w\|_{X_N(I)}\lesssim 1.\label{0.16}
\end{align}
Indeed, from Corollary \ref{cor:L-2p+2} and Lemma \ref{lem:H1}, we have
\begin{align*}
N^{-\alpha(d,p)\cdot(1- s_0)}&\|w\|_{L^{r_0}_{tx}(I\times \R^d)}+N^{-3(1-s_0)}\|w\|_{L^\infty_t \dot H^1_x(I\times \R^d)}\\
&\lesssim 1+\|w\|_{X_N(I)}^{\frac d{r_0(d-2)}}+\delta_0^{\frac1{r_0}} \|w\|_{X_N(I)}^{\alpha_3(d,p)}
+\delta_0^\frac12\|w\|_{X_N(I)}^{\alpha_4(d,p)}.
\end{align*}
That is,
\begin{align*}
\|w\|_{X_N(I)}\lesssim 1+\|w\|_{X_N(I)}^{\frac d{r_0(d-2)}}+\delta_0^{\frac1{r_0}} \|w\|_{X_N(I)}^{\alpha_3(d,p)}
+\delta_0^\frac12\|w\|_{X_N(I)}^{\alpha_4(d,p)}.
\end{align*}
We note that
$$
\frac d{r_0(d-2)}<1,\quad \alpha_4(d,p)>1.
$$
Then using the Cauchy-Schwarz inequality, the second term can be absorbed by the term in the left-hand side, and thus we have
\begin{align*}
\|w\|_{X_N(I)}\lesssim 1+\delta_0^{\frac1{r_0}} \|w\|_{X_N(I)}^{\alpha_3(d,p)}
+\delta_0^\frac12\|w\|_{X_N(I)}^{\alpha_4(d,p)}.
\end{align*}
Furthermore, if $\alpha_3(d,p)\le 1$, then choosing $\delta_0$ suitably small,  we have
\begin{align*}
\|w\|_{X_N(I)}\lesssim 1+\delta_0^\frac12\|w\|_{X_N(I)}^{\alpha_4(d,p)}.
\end{align*}
Then using a continuity argument, we obtain \eqref{0.16}. If $\alpha_3(d,p)> 1$, then choosing $\delta_0$ suitably small, and using a continuity argument, we also obtain \eqref{0.16}.

Since the estimate in \eqref{0.16} is uniform in the time interval $I$, we have $I=\R^+$. This proves the global existence in the forward time. Moreover,
\begin{align}
\|w\|_{L^\infty_t \dot H^1_x(\R^+\times \R^d)}+\|w\|_{L^{r_0}_{tx}(\R^+\times \R^d)}\lesssim A(N),\label{1.25}
\end{align}
for some constant $A$ depending on $N$.

Now we set
$$
u^+=f_++\int_0^{+\infty} e^{-is\Delta}\big(|u|^pu\big)\,ds.
$$
Then we have
$$
u(t)-e^{it\Delta}u^+=\int_t^{+\infty} e^{i(t-s)\Delta}\big(|u|^pu\big)\,ds.
$$
Using Proposition \ref{prop:f+} and \eqref{1.25}, we have (the details are omitted here since similar treatment was presented above)
\begin{align*}
\Big\|\int_t^{+\infty} e^{i(t-s)\Delta}\big(|u|^pu\big)\,ds\Big\|_{H^1(\R^d)}\lesssim
\big\|\langle\nabla \rangle \big(|u|^pu\big)\big\|_{L^2_tL^{\frac{2d}{d+2}}_x([t,+\infty)\times \R)}\to 0,
\end{align*}
as $t\to +\infty$. This proves the scattering statement.

\section*{Acknowledgements}

Part of this work was done while M. Beceanu, A. Soffer and Y. Wu were visiting CCNU (C.C.Normal University) Wuhan, China. The authors thank the institutions for their hospitality and the support. M.B. is partially supported by NSF grant DMS 1700293.  Q.D.is partially supported by NSFC 11661061, 11671163 and 11771165, and the Fundamental Research Funds for the Central Universities (CCNU18QN030, CCNU18CXTD04).  A.S is partially supported by NSF grant DMS 01600749. The research was also supported by NSFC 11971191 and 11671163.  Y.W. is partially supported by NSFC 11771325 and National Youth Topnotch Talent Support Program in China, and A.S. and Y.W. would like to thank  Ch. Miao for useful discussions.

\end{document}